\documentclass[12pt]{amsart}%
\usepackage{amsfonts}
\usepackage{amsmath}
\usepackage{amssymb}
\usepackage{graphicx}%
\providecommand{\U}[1]{\protect\rule{.1in}{.1in}}
\newtheorem{theorem}{Theorem}
\theoremstyle{plain}

\newtheorem{corollary}{Corollary}

\newtheorem{definition}{Definition}

\newtheorem{lemma}{Lemma}

\newtheorem{proposition}{Proposition}
\newtheorem{remark}{Remark}

\numberwithin{equation}{section}
\numberwithin{equation}{section}
\numberwithin{theorem}{section}
\numberwithin{lemma}{section}
\numberwithin{remark}{section}
\numberwithin{example}{section}
\numberwithin{proposition}{section}
\numberwithin{definition}{section}
\numberwithin{corollary}{section}
\begin{document}
\title[Homological framework ]{Homological framework of noncommutative complex analytic geometry and
functional calculus}
\author{Anar Dosi}
\address{College of Mathematical Sciences, Harbin Engineering University, Nangang
District, Harbin, 150001, China}
\email{dosiev@yahoo.com dosiev@hrheu.edu.cn}
\date{December 23, 2025}
\dedicatory{To the blessed memory of Joseph L. Taylor}
\begin{abstract}
In the paper we propose topological homology framework of noncommutative
complex analytic geometries of Fr\'{e}chet algebras, and investigate the
related functional calculus and spectral mapping properties. It turns out that
an ideal analytic geometry of a Fr\'{e}chet algebra $A$ can be described in
terms of a \v{C}ech category over $A$. The functional calculus problem within
a particular \v{C}ech $A$-category, and a left Fr\'{e}chet $A$-module $X$ is
solved in term of the homological spectrum of $X$ with respect to that
category. As an application, we use the formal $q$-geometry of a contractive
operator $q$-plane, and solve the related noncommutative holomorphic
functional calculus problem. The related spectrum is reduced to Putinar
spectrum of a Fr\'{e}chet $q$-module. In the case of a Banach $q$-module we
come up with the closure of its Taylor spectrum.

\end{abstract}
\maketitle

\section{Introduction\label{secInt}}

The methods of topological homology are highly productive in many different
problems of topological algebras and noncommutative geometry. The foundations
of topological homology were developed mainly by A. Ya. Helemskii
\cite{HelHom} and J. L. Taylor \cite{Tay1} in the beginning of 70's of the
previous century. One of the main and surprising achievements of these methods
is to propose an alternative approach to the multioperator functional calculus
problem and the joint spectral theory developed by J. L. Taylor \cite{Tay2}.
Furthermore, the general framework of Taylor's holomorphic functional calculus
launched a promising light toward its noncommutative treatments. Actually, it
was the first steps how to figure out noncommutative complex analytic geometry
based on topological homology methods. The categorical foundations of the
algebraic treatments of the derived analytic geometry were reflected in the
recent manuscript \cite{Oren} by O. Ben-Bassat, J. Kelly and K. Kremnizer.

The problem of noncommutative holomorphic functional calculus can be
considered as a key part of the complex analytic geometry. Taylor's ideas
within noncommutative complex analytic geometry got their further progress
later in the beginning of the present century mainly by O. Yu. Aristov, A. S.
Fainstein, A. Yu. Pirkovskii and the author in their several papers. The
algebraic versions of these ideas were considered in \cite{DosiMMJ},
\cite{DosiTJM}, \cite{DosiBJMA24}, and \cite{DosiAA25}.

As a starting position in the construction of a complex analytic geometry of a
given (noncommutative) finitely generated algebra $\mathfrak{A}$ is to fix its
certain multinormed envelope $A$ representing noncommutative entire functions
in the generators of $\mathfrak{A}$. The choice of $A$ is constrained by the
assumption that we can draw the (pure algebraic) homology over $\mathfrak{A}$
to topological homology over $A$ through the canonical homomorphism
$\mathfrak{A}\rightarrow A$, that is, $\mathfrak{A}\rightarrow A$ should be a
localization in the sense of J. L. Taylor \cite{Tay2}, or a homotopy
epimorphism \cite{Oren}. The next step is to describe the set $\mathfrak{X}$
of all irreducible Banach space representations \cite[6.2]{Hel} of $A$
equipped with a Zariski type (used to be non-Hausdorff) topology. The
description of $\mathfrak{X}$ being not so easy task can be handled whenever
the algebra $A$ is commutative modulo its Jacobson radical, which is another
constraint on $A$. In this case, $\mathfrak{X}$ turns out to be the set
$\operatorname{Spec}\left(  A\right)  $ of all continuous characters of $A$ or
$1$-dimensional irreducible representations of $A$. By an analytic geometry of
$A$, we mean a ringed space $\left(  \mathfrak{X},\mathcal{O}_{\mathfrak{X}%
}\right)  $ of a Fr\'{e}chet algebra presheaf $\mathcal{O}_{\mathfrak{X}}$ on
a topological space $\mathfrak{X}$ containing $\operatorname{Spec}\left(
A\right)  $, so that $A=\mathcal{O}_{\mathfrak{X}}\left(  \mathfrak{X}\right)
$ and $U\supseteq\operatorname{Spec}\left(  \mathcal{O}_{\mathfrak{X}}\left(
U\right)  \right)  $ for every open subset $U\subseteq\mathfrak{X}$. The
functional calculus problem within a geometry $\left(  \mathfrak{X}%
,\mathcal{O}_{\mathfrak{X}}\right)  $ of $A$ is to extend the structure of a
given left Fr\'{e}chet $A$-module $X$ to a left $\mathcal{O}_{\mathfrak{X}%
}\left(  U\right)  $-module $X$ one for an open subset $U\subseteq
\mathfrak{X}$. The problem is closely related to the concept of (joint)
spectrum $\sigma\left(  \mathcal{O}_{\mathfrak{X}},X\right)  $ of $X$ with
respect the presheaf $\mathcal{O}_{\mathfrak{X}}$. In the case of
$\mathfrak{X=}\operatorname{Spec}\left(  A\right)  $, we come up with the
standard analytic geometry of $A$ (one can also refer it to $\mathfrak{A}$).
The presence of a standard geometry strongly depends on the choice of a
multinormed envelope $A$ of $\mathfrak{A}$. It should be a nontrivial
Arens-Michael-Fr\'{e}chet algebra, which is commutative modulo its Jacobson
radical, and the canonical homomorphism $\mathfrak{A}\rightarrow A$ should be
a localization.

The multinormed envelopes of an algebra with respect to a class of its Banach
algebra actions were introduced by O. Yu. Aristov \cite{Aris}. If we choose
the class of all Banach algebra actions of $\mathfrak{A}$, the related
envelope $\widehat{\mathfrak{A}}$ called the Arens-Michael envelope of
$\mathfrak{A}$ (see \cite[5]{Tay1}, \cite[5.2.21]{Hel}) represents the algebra
of all (noncommutative) entire functions in the elements of $\mathfrak{A}$.
The choice of Banach PI-algebra or Banach N-algebra actions of $\mathfrak{A}$
result in the PI-envelope $\widehat{\mathfrak{A}}^{\text{PI}}$ and the
N-envelope $\widehat{\mathfrak{A}}^{\text{N}}$ of $\mathfrak{A}$,
respectively. The class of Banach PI-algebras refers to the Banach algebras
with polynomial identities whereas the class Banach N-algebras consists of
Banach algebras with their nilpotent ideals whose quotients are commutative.
These envelopes are larger than the Arens-Michael envelope
$\widehat{\mathfrak{A}}$, and they are distinct in general. In some particular
cases, they coincide, and in some other cases they can be trivial depending on
the base algebra $\mathfrak{A}$ (see \cite{DosiSS}).

If $\mathfrak{A}=\mathbb{C}\left[  x_{1},\ldots,x_{n}\right]  $ is the
polynomial algebra, then we come up with the standard complex geometry
$\left(  \mathbb{C}^{n},\mathcal{O}\right)  $ of the analytic space
$\mathbb{C}^{n}$ with the Fr\'{e}chet algebra sheaf $\mathcal{O}$ of germs of
holomorphic functions on $\mathbb{C}^{n}$. In this case,
\[
A=\widehat{\mathfrak{A}}=\widehat{\mathfrak{A}}^{\text{PI}}%
=\widehat{\mathfrak{A}}^{\text{N}}=\mathcal{O}\left(  \mathbb{C}^{n}\right)
\]
is the algebra of all entire functions in $n$ complex variables. The related
functional calculus problem is solved in terms of the joint spectrum
\cite{Tay}, \cite[2.5]{EP}, which is well known as the Taylor spectrum of a
Banach $A$-module $X$.

In the case of the universal enveloping algebra $\mathfrak{A}=\mathcal{U}%
\left(  \mathfrak{g}\right)  $ of a finite dimensional nilpotent Lie algebra
$\mathfrak{g}$, we have $\widehat{\mathfrak{A}}^{\text{PI}}%
=\widehat{\mathfrak{A}}^{\text{N}}$ and it is reduced to the algebra of all
formally radical entire functions $\mathcal{F}_{\mathfrak{g}}\left(
\mathbb{C}^{m}\right)  $ in elements of $\mathfrak{g}$ \cite{DosiAC}, where
$\mathbb{C}^{m}=\mathfrak{g/}\left[  \mathfrak{g},\mathfrak{g}\right]  $.
Moreover, $\mathcal{F}_{\mathfrak{g}}\left(  \mathbb{C}^{m}\right)  $ is
commutative modulo its Jacobson radical and $\mathbb{C}^{m}%
=\operatorname{Spec}\left(  \widehat{\mathfrak{A}}^{\text{PI}}\right)  $. The
canonical homomorphism $\mathcal{U}\left(  \mathfrak{g}\right)  \rightarrow
\mathcal{F}_{\mathfrak{g}}\left(  \mathbb{C}^{m}\right)  $ turns out to be an
absolute localization, and we come up with the standard geometric
configuration $\left(  \mathbb{C}^{m},\mathcal{F}_{\mathfrak{g}}\right)  $
with the structure sheaf
\[
\mathcal{F}_{\mathfrak{g}}=\mathcal{O}\widehat{\otimes}\mathbb{C}\left[
\left[  \omega_{1},\ldots,\omega_{k}\right]  \right]
\]
of stalks of formally radical functions on $\mathbb{C}^{m}$, where
$k=\dim\left(  \left[  \mathfrak{g},\mathfrak{g}\right]  \right)  $
\cite{DosJOT10}, \cite{DosiAC}. Thus $\left(  \mathbb{C}^{m},\mathcal{F}%
_{\mathfrak{g}}\right)  $ is a standard geometry that stands for the
PI-envelope of $\mathcal{U}\left(  \mathfrak{g}\right)  $. The related
$\mathcal{F}_{\mathfrak{g}}$-functional calculus problem is solved in
\cite{DosJOT10}. The case of the Arens-Michael envelope $\widehat{\mathfrak{A}%
}=\mathcal{O}_{\mathfrak{g}}$ remains unclear though the related localization
problem $\mathcal{U}\left(  \mathfrak{g}\right)  \rightarrow\mathcal{O}%
_{\mathfrak{g}}$ solved positively (see \cite{DosIzv}, \cite{PirSF},
\cite{Aris20}).

The algebra $\mathcal{F}_{\mathfrak{g}}\left(  \mathbb{C}^{m}\right)  $ is
defined for a solvable Lie algebra $\mathfrak{g}$ case too, and $C^{\infty}%
$-version of $\mathcal{F}_{\mathfrak{g}}$ including the related structure
sheaf were constructed in \cite{ArisI}. The case of a solvable, but
non-nilpotent Lie algebra $\mathfrak{g}$ with both $C^{\infty}$ and
$\mathcal{F}_{\mathfrak{g}}$ structure sheaves was proposed in \cite{Aris3}.
It turns out that the algebra of all noncommutative global $C^{\infty}%
$-sections from \cite{ArisI} is the envelope with respect to the class of all
Banach algebras with the polynomial growth (PG algebras) \cite{Aris}.

In the case of a contractive quantum plane $\mathfrak{A=A}_{q}=\mathbb{C}%
\left\langle x,y\right\rangle /\left(  xy-q^{-1}yx\right)  $, $q\in\mathbb{C}%
$, $\left\vert q\right\vert \neq1$, its Arens-Michael envelope
$\widehat{\mathfrak{A}}$ (the algebra of all entire functions in
noncommutative variables $x$, $y$) turns out to be commutative modulo its
Jacobson radical (see \cite{Dosi24}). Moreover, $\operatorname{Spec}\left(
\widehat{\mathfrak{A}}\right)  =\mathbb{C}_{xy}$ is the union of two complex
lines $\mathbb{C}_{x}=\mathbb{C\times}\left\{  0\right\}  $ and $\mathbb{C}%
_{y}=\left\{  0\right\}  \times\mathbb{C}$ in $\mathbb{C}^{2}$, and the
canonical homomorphism $\mathfrak{A}\rightarrow\widehat{\mathfrak{A}}$ is an
absolute localization \cite{Pir}. One can equip $\mathbb{C}_{x}$ with the
$q$-topology and the other line $\mathbb{C}_{y}$ with the disk topology (see
\cite{Dosi24}) or vice-versa. Then $\mathbb{C}_{xy}$ is equipped with the
final topology so that both inclusions of these lines into $\mathbb{C}_{xy}$
are continuous. The (noncommutative) analytic space $\left(  \mathbb{C}%
_{xy},\mathcal{O}_{q}\right)  $ that stands for $\widehat{\mathfrak{A}}$ is
given by the topological space $\mathbb{C}_{xy}$ and the structure Fr\'{e}chet
algebra presheaf $\mathcal{O}_{q}$ \cite{Dosi24}. The related $\mathcal{O}%
_{q}$-functional calculus problem for a left Banach $\mathfrak{A}_{q}$-module
was solved in \cite{Dosi24}.

The PI envelope $A=\widehat{\mathfrak{A}_{q}}^{\text{PI}}$ of $\mathfrak{A}%
_{q}$ admits the standard analytic geometry $\left(  \mathbb{C}_{xy}%
,\mathcal{F}_{q}\right)  $ referred as the formal geometry of $\mathfrak{A}%
_{q}$, and it has the structure sheaf $\mathcal{F}_{q}$ not just a presheaf.
Namely, first let us notify that the PI and N envelopes of $\mathfrak{A}_{q}$
coincide, that is, $A=\widehat{\mathfrak{A}_{q}}^{\text{N}}$ (see
\cite{DosiSS}) and $\operatorname{Spec}\left(  A\right)  =\mathbb{C}_{xy}$
holds too. Further, the canonical homomorphism $\mathfrak{A}_{q}\rightarrow A$
turns out to be localization \cite{DosiLoc}. The same spectrum $\mathbb{C}%
_{xy}$ is equipped with the $q$-topology over both lines, and the Fr\'{e}chet
algebra sheaf $\mathcal{F}_{q}$ on $\mathbb{C}_{xy}$ is obtained as the
fibered product
\[
\mathcal{F}_{q}=\mathcal{O}\left[  \left[  y\right]  \right]
\underset{\mathbb{C}\left[  \left[  x,y\right]  \right]  }{\times}\left[
\left[  x\right]  \right]  \mathcal{O}%
\]
of the formal power series sheaves $\mathcal{O}\left[  \left[  y\right]
\right]  $ and $\left[  \left[  x\right]  \right]  \mathcal{O}$ over the
constant sheaf $\mathbb{C}\left[  \left[  x,y\right]  \right]  $, where
$\mathcal{O}$ is the sheaf of stalks of holomorphic functions on the complex
$q$-plane (see \cite{Dosi25}). The related functional calculus problem for a
left Fr\'{e}chet $A$-module $X$ is going to be solved in the present work as a
particular case of the general functional calculus of \v{C}ech categories.

We introduce an $A$-category $\mathcal{S}$ to be a certain category of
Fr\'{e}chet $A$-algebras, that is, there are compatible morphisms
$A\rightarrow\mathcal{A}$ into objects $\mathcal{A}$ of $\mathcal{S}$ denoted
briefly by $A\rightarrow\mathcal{S}$. A morphism of $A\rightarrow\mathcal{S}$
into another $B\rightarrow\mathcal{G}$ we mean a functor
$F:\mathcal{S\rightarrow G}$ along with a family of the compatible morphisms
$f:A\rightarrow B$ and $\mathcal{A}\rightarrow F\left(  \mathcal{A}\right)  $
for all objects $\mathcal{A}$ of $\mathcal{S}$. The category $\mathcal{T}_{A}$
of all trivial modules $\mathbb{C}\left(  \lambda\right)  $, $\lambda
\in\operatorname{Spec}\left(  A\right)  $ with their trivial morphisms is an
$A$-category. One can complete an $A$-category $\mathcal{S}$ by adding up all
trivial modules $\mathbb{C}\left(  \lambda\right)  $ obtained by $\lambda
\in\operatorname{Spec}\left(  \mathcal{A}\right)  $ for all objects
$\mathcal{A}$ from $\mathcal{S}$, which results in a new $A$-category
$\mathcal{S}^{\sim}$ called \textit{the point completion of }$\mathcal{S}$. A
spectral theory of the left Fr\'{e}chet $A$-modules with respect to
$A$-categories is developed below in Section \ref{sectionFM}. Our approach is
based on topological homology constructions developing the key ideas of J. L.
Taylor \cite{Tay2}, M. Putinar \cite{Put}, and the author's \cite{DosJOT10}.
The \textit{resolvent set} $\operatorname{res}\left(  \mathcal{S},X\right)  $
of a left Fr\'{e}chet $A$-module $X$ with respect to an $A$-category\textit{
}$\mathcal{S}$ is defined to be the set of those objects $\mathcal{A}$ of
$\mathcal{S}$ such that $\mathcal{B}\perp X$ holds (that is,
$\operatorname{Tor}_{k}^{A}\left(  \mathcal{B},X\right)  =\left\{  0\right\}
$, $k\geq0$) for every morphism $\mathcal{A\rightarrow B}$ in $\mathcal{S}$.
The complement
\[
\sigma\left(  \mathcal{S},X\right)  =\mathcal{S}\backslash\operatorname{res}%
\left(  \mathcal{S},X\right)
\]
is called \textit{the spectrum of the }$A$\textit{-module} $X$\textit{ with
respect to} $\mathcal{S}$. The spectrum is closed with respect to the natural
(Aleksandrov) topology of $\mathcal{S}$ (see Subsection \ref{subsecAT}). The
set%
\[
\sigma\left(  \mathcal{S}^{\sim}\cap\mathcal{T}_{A},X\right)  =\left(
\mathcal{S}^{\sim}\cap\mathcal{T}_{A}\right)  \backslash\operatorname{res}%
\left(  \mathcal{S}^{\sim}\cap\mathcal{T}_{A},X\right)
\]
is called \textit{the Taylor spectrum of} \textit{the }$A$\textit{-module}
$X$\textit{ with respect to} $\mathcal{S}$. We also consider the following set
$\operatorname{res}_{\operatorname{P}}\left(  \mathcal{S},X\right)
=\operatorname{res}\left(  \mathcal{S},X\right)  ^{\sim}\cap\mathcal{T}_{A}$,
where $\operatorname{res}\left(  \mathcal{S},X\right)  ^{\sim}$ is the point
completion of the open subcategory $\operatorname{res}\left(  \mathcal{S}%
,X\right)  $. Thus $\mathbb{C}\left(  \lambda\right)  $ is an object of
$\operatorname{res}_{\operatorname{P}}\left(  \mathcal{S},X\right)  $ iff
$\lambda\in\operatorname{Spec}\left(  \mathcal{A}\right)  $ for a certain
object $\mathcal{A}\in\operatorname{res}\left(  \mathcal{S},X\right)  $. The
set
\[
\sigma_{\operatorname{P}}\left(  \mathcal{S},X\right)  =\left(  \mathcal{S}%
^{\sim}\cap\mathcal{T}_{A}\right)  \backslash\operatorname{res}%
_{\operatorname{P}}\left(  \mathcal{S},X\right)
\]
is called \textit{the Putinar spectrum of the }$A$-\textit{module }$X$
\textit{with respect to} $\mathcal{S}$. Our first central asserts that if $A$
is a finite type algebra (having a finite free $A$-bimodule resolution),
$\mathcal{S}$ a nuclear $A$-category, $\mathcal{A}$ an object of $\mathcal{S}$
dominating over a left $A$-module $X$ (the presence of an $\mathcal{A}%
$-calculus on $X$), then
\[
\sigma\left(  \mathcal{S},X\right)  |_{\mathcal{A}}=\sigma\left(
U_{\mathcal{A}},X\right)  ,\quad\sigma_{\operatorname{P}}\left(
\mathcal{S},X\right)  |_{\mathcal{A}}=\sigma_{\operatorname{P}}\left(
U_{\mathcal{A}},X\right)  ,\quad\sigma\left(  \mathcal{S}^{\sim}%
\cap\mathcal{T}_{A},X\right)  |_{\mathcal{A}}=\sigma\left(  \mathcal{T}%
_{\mathcal{A}},X\right)
\]
hold (see Theorem \ref{thSMT}), where $U_{\mathcal{A}}$ is the open (least)
subcategory of $\mathcal{S}$ containing all possible morphisms from
$\mathcal{A}$ into objects of $\mathcal{S}$, which is considered to be an
$\mathcal{A}$-category.

Further in Section \ref{sectionAFA}, we consider the poset $A$-categories,
which form the complete lattices. It turns out the unital complete lattice
categories with their strong morphisms are Fr\'{e}chet algebra presheaves
indeed (see Proposition \ref{corCes1}). We introduce a complex analytic
geometry of $A$ as a unital complete-lattice $A$-category $\mathcal{S}$ such
that for every open subset $U\subseteq\mathcal{S}$ the inclusion
$\operatorname{Spec}\left(  \wedge U\right)  \subseteq U^{\sim}$ holds, where
$\wedge U$ is the least upper bound of $U$ in $\mathcal{S}$. It is equivalent
to the presence of a Fr\'{e}chet algebra presheaf $\mathcal{P}$ on a
topological space $\omega$ containing $\operatorname{Spec}\left(  A\right)  $
such that $A=\Gamma\left(  \omega,\mathcal{P}\right)  $ and $U\cap
\operatorname{Spec}\left(  A\right)  =\operatorname{Spec}\left(
\mathcal{P}\left(  U\right)  \right)  $ for every open subset $U\subseteq
\omega$, which is compatible with the arguments and examples considered above.

A noncommutative covering (or a basis) of a unital complete-lattice
$A$-category $\mathcal{S}$ we mean a countable family $\mathfrak{b}$ of its
objects such that $\wedge U=\wedge\left(  U\cap\mathfrak{b}\right)  $ for
every nontrivial open subset $U\subseteq\mathcal{S}$. If the augmented
\v{C}ech complex of $\mathfrak{b}$ is exact, then we say that $\mathcal{S}$ is
\textit{a \v{C}ech }$A$-\textit{category}. The central result on the
functional calculus asserts that is $X$ is a finite-free left Fr\'{e}chet
$A$-module and $U$ is an open neighborhood of the spectrum $\sigma\left(
\mathcal{S},X\right)  $ in a \v{C}ech $A$-category $\mathcal{S}$ with a
nuclear basis $\mathfrak{b}$, then $X$ turns out to be a left Fr\'{e}chet
$\wedge U$-module extending its $A$-module structure through the morphism
$A\rightarrow\wedge U$.

In particular, if $A$ is a finite type algebra with its complex analytic
geometry $\left(  \omega,\mathcal{P}\right)  $ such that $\mathcal{P}$ is a
nuclear Fr\'{e}chet algebra presheaf, $\mathfrak{b}=\left\{  V_{i}\right\}  $
a countable topology base of $\omega$ such that its augmented \v{C}ech complex
is exact, and $U\subseteq\omega$ is an open subset containing the spectrum
$\sigma\left(  \mathcal{P},X\right)  $, then $X$ turns out to be a left
$\mathcal{P}\left(  U\right)  $-module extending its $A$-module structure
through the restriction morphism $A\rightarrow\mathcal{P}\left(  U\right)  $.

In the case of a standard analytic geometry $\left(  \operatorname{Spec}%
\left(  A\right)  ,\mathcal{F}\right)  $ of $A$ with the structure nuclear
Fr\'{e}chet algebra sheaf $\mathcal{F}$, we come up with the following result
that generalizes all functional calculi obtained in \cite{Tay2}, \cite{Put},
\cite{Dos} and \cite{DosJOT10}. If all finite intersections of a a countable
basis $\mathfrak{b}$ for $\operatorname{Spec}\left(  A\right)  $ (including
the space itself) are $\mathcal{F}$-acyclic, and $X$ is a finite-free left
$A$-module, then $X$ turns out to be a Fr\'{e}chet left $\mathcal{F}\left(
U\right)  $-module extending its $A$-module structure whenever $U$ is an open
neighborhood of the Putinar spectrum $\sigma_{\operatorname{P}}\left(
\mathcal{F},X\right)  $. In the case of the various multinormed envelopes $A$
of the contractive quantum plane $\mathfrak{A}_{q}$ we obtain a number of
exotic analytic geometries and the related functional calculi that considered
in Section \ref{sectionCQP}.

\section{Preliminaries\label{sPre}}

All considered vector spaces are assumed to be complex, and (associative)
algebras are unital. The algebra of all continuous linear operators acting on
a polynormed (or locally convex) space $X$ is denoted by $\mathcal{L}\left(
X\right)  $, whereas the same algebra is denoted by $\mathcal{B}\left(
X\right)  $ in the case of a normed space $X$. If the topology of a complete
polynormed algebra $A$ is defined by means of a family of multiplicative
seminorms, then $A$ is called an \textit{Arens-Michael algebra} \cite[1.2.4]%
{Hel}. The main category of the underlying polynormed spaces considered in the
paper is the category $\mathfrak{Fs}$ of all Fr\'{e}chet spaces, whereas
$\mathfrak{Fa}$ denotes the category of the Fr\'{e}chet algebras. The category
of all cochain complexes over $\mathfrak{Fs}$ is denoted by $\overline
{\mathfrak{Fs}}$.

The set of all continuous characters of a polynormed algebra $A$ is denoted by
$\operatorname{Spec}\left(  A\right)  $. If $\lambda\in\operatorname{Spec}%
\left(  A\right)  $ then the algebra homomorphism $\lambda:A\rightarrow
\mathbb{C}$ defines (Banach) $A$-module structure on $\mathbb{C}$ through
$\lambda$. This module called \textit{a trivial module} is denoted by
$\mathbb{C}\left(  \lambda\right)  $.

\subsection{Transversality of modules\label{subsecTrans}}

Let $A$ be a unital Fr\'{e}chet algebra. The category of all left (resp.,
right) Fr\'{e}chet $A$-modules is denoted by $A$-$\operatorname{mod}$ (resp.,
$\operatorname{mod}$-$A$), whereas $A$-$\operatorname{mod}$-$A$ denotes the
category of all Fr\'{e}chet $A$-bimodules. A left $A$-module $X$ is said to be
a free left $A$-module if $X=A\widehat{\otimes}E$ is the projective tensor
product of $A$ and some Fr\'{e}chet space $E$ equipped with the natural left
$A$-module structure. A retract in the category $A$-$\operatorname*{mod}$ of a
free $A$-module is called a projective left $A$-module. A left $A$-module $X$
is said to be \textit{a finite-projective} (resp., \textit{finite-free})
module if it has a finite projective (resp., free) resolution $\mathcal{P}%
^{\bullet}=\left\{  P_{i}:-n\leq i\leq0\right\}  $ in the category
$A$-$\operatorname{mod}$ (see \cite[3.2]{HelHom}), that is, the complex
\[
0\rightarrow P_{-n}\rightarrow\cdots\rightarrow P_{-1}\rightarrow
P_{0}\rightarrow X\rightarrow0
\]
with a connecting (augmentation) morphism $\varepsilon:P_{0}\rightarrow X$ is
admissible (splits in the category $\mathfrak{Fs}$). It is convenient for us
to use the negative indices for a resolution of $X$ in its cochain version.
Briefly we say that $\mathcal{P}^{\bullet}\rightarrow X\rightarrow0$ is
admissible, and $\mathcal{P}^{\bullet}$ is a finite projective resolution of
$X$. If $Y$ is a right $A$-module then recall \cite[3.3.1]{HelHom} that $Y$ is
in the transversality relation with respect to the left $A$-module $X$ if all
derived functors of $Y\widehat{\otimes}_{A}\circ$ applied to the module $X$
are vanishing, that is, $\operatorname{Tor}_{k}^{A}\left(  Y,X\right)
=\left\{  0\right\}  $ for all $k\geq0$. In this case, we write $Y\perp X$.

If the algebra $A$ itself has a finite free $A$-bimodule resolution
$\mathcal{R}^{\bullet}=\left\{  R_{i}:-n\leq i\leq0\right\}  $, then every
left $A$-module is a finite-free module. In this case, $R_{i}%
=A\widehat{\otimes}E_{i}\widehat{\otimes}A$ for some Fr\'{e}chet spaces
$E_{i}$, and for brevity we write $\mathcal{R}^{\bullet}=A\widehat{\otimes
}\mathfrak{e}^{\bullet}\widehat{\otimes}A$ with $\mathfrak{e}^{\bullet
}=\left\{  E_{i}:-n\leq i\leq0\right\}  $. Since $\mathcal{R}^{\bullet
}\rightarrow A\rightarrow0$ is an admissible complex of the free
$A$-bimodules, we deduce \cite[3.1.18]{HelHom} that it splits in the category
$\operatorname{mod}$-$A$. By applying the functor $\circ\widehat{\otimes}%
_{A}X$ to the complex $\mathcal{R}^{\bullet}\rightarrow A\rightarrow0$, we
derive that $\mathcal{R}^{\bullet}\widehat{\otimes}_{A}X\rightarrow
X\rightarrow0$ is an admissible complex, that is, $\mathcal{P}^{\bullet
}=\mathcal{R}^{\bullet}\widehat{\otimes}_{A}X=A\widehat{\otimes}%
\mathfrak{e}^{\bullet}\widehat{\otimes}X$ is a finite free resolution of the
left $A$-module $X$. Moreover, $\operatorname{Tor}_{k}^{A}\left(  Y,X\right)
$ are the homology groups of $Y\widehat{\otimes}_{A}\mathcal{R}^{\bullet
}\widehat{\otimes}_{A}X$, which is the complex $Y\widehat{\otimes}%
\mathfrak{e}^{\bullet}\widehat{\otimes}X$. In applications, we come up with
the finite free $A$-bimodule resolutions $\mathcal{R}^{\bullet}%
=A\widehat{\otimes}\mathfrak{e}^{\bullet}\widehat{\otimes}A$ with the finite
dimensional (or nuclear) spaces $\mathfrak{e}^{\bullet}$. So are the Koszul
type resolutions. In this case, we say that $A$ is of \textit{finite type}.

\subsection{Taylor localizations\label{subsecTL}}

Now let $\mathfrak{A}$ be a polynormed algebra. For example, it can be a pure
algebra equipped with the finest polynormed topology. Assume that
$\iota:\mathfrak{A}\rightarrow A$ is a continuous algebra homomorphism from
$\mathfrak{A}$ into a Fr\'{e}chet algebra $A$. The homomorphism $\iota
:\mathfrak{A}\rightarrow A$ is said to be \textit{a localization} if it
induces the natural isomorphisms $H_{n}\left(  \mathfrak{A},X\right)
=H_{n}\left(  A,X\right)  $, $n\geq0$ of all homology groups for every
Fr\'{e}chet $A$-bimodule $X$. In particular, the multiplication mapping
$A\overline{\otimes}_{\mathfrak{A}}A\rightarrow A$, $a_{1}\otimes
_{\mathfrak{A}}a_{2}\mapsto a_{1}a_{2}$ (on the completed inductive tensor
product) is a topological isomorphism, and $H_{n}\left(  \mathfrak{A}%
,A\widehat{\otimes}A\right)  =\operatorname{Tor}_{n}^{\mathfrak{A}}\left(
A,A\right)  =\left\{  0\right\}  $ for all $n>0$, where the algebra $A$ is
considered to be a $\mathfrak{A}$-bimodule via the homomorphism $\iota$. If
$\mathfrak{P}\rightarrow\mathfrak{A}\rightarrow0$ is an admissible projective
bimodule resolution of $\mathfrak{A}$, then the application of the functor
$A\overline{\otimes}_{\mathfrak{A}}\circ\overline{\otimes}_{\mathfrak{A}}A$ to
$\mathfrak{P}$ augmented by the multiplication morphism results in the
following cochain complex
\begin{equation}
A\overline{\otimes}_{\mathfrak{A}}\mathfrak{P}\overline{\otimes}%
_{\mathfrak{A}}A\longrightarrow A\rightarrow0. \label{TAc}%
\end{equation}
If (\ref{TAc}) is admissible for some (actually for every) $\mathfrak{P}$,
then certainly $\iota:\mathfrak{A}\rightarrow A$ is a localization called
\textit{an absolute localization }(or $A$ is \textit{stably flat over}
$\mathfrak{A}$). In this case, $H_{n}\left(  \mathfrak{A},X\right)
=H_{n}\left(  A,X\right)  $, $n\geq0$ hold for every $\widehat{\otimes}%
$-bimodule $X$ (not necessarily a Fr\'{e}chet one). If $A$ and $A\overline
{\otimes}_{\mathfrak{A}}\mathfrak{P}\overline{\otimes}_{\mathfrak{A}}A$
consists of nuclear Fr\'{e}chet spaces, then $\iota:\mathfrak{A}\rightarrow A$
is a localization (see \cite[Proposition 1.6]{Tay2}) whenever (\ref{TAc}) is exact.

Finally, notice that if $\iota:\mathfrak{A}\rightarrow A$ is a localization
and $A\rightarrow B$ is a morphism of the Fr\'{e}chet algebras, then the
compose homomorphism $\mathfrak{A}\rightarrow B$ is a localization if and only
if so is $A\rightarrow B$ \cite[Proposition 1.8]{Tay2}. If $\mathfrak{A}$ is
of finite type with its finite free resolution $\mathfrak{P}=\mathfrak{A}%
\overline{\otimes}\mathfrak{e}^{\bullet}\overline{\otimes}\mathfrak{A}$ (in
this case, $\mathfrak{e}^{\bullet}$ consists of finite dimensional spaces),
then so is $A$ whenever $\iota:\mathfrak{A}\rightarrow A$ is a localization
and $A$ is nuclear. Thus, $\mathcal{R}^{\bullet}=A\widehat{\otimes
}\mathfrak{e}^{\bullet}\widehat{\otimes}A$ is a finite free $A$-bimodule
resolution of $A$.

\subsection{The dominating complex over a module\label{subsecDomM}}

Let $\mathcal{Y}=\left\{  Y_{s}:s\geq0\right\}  $ be a complex in
$\overline{\operatorname*{mod}\text{-}A}$, and let $X$ be a finite-projective
left $A$-module with its finite projective right $A$-module resolution
$\mathcal{P}^{\bullet}=\left\{  P_{i}:-n\leq i\leq0\right\}  $. Then we have
the well defined bicomplex $\overline{\mathcal{Y}\widehat{\otimes}%
_{A}\mathcal{P}^{\bullet}}$ with the rows $\mathcal{Y}\widehat{\otimes}%
_{A}P_{k}$ (rightward directed), $k\leq0$ and the columns $Y_{s}%
\widehat{\otimes}_{A}\mathcal{P}^{\bullet}$ (upward directed), $s\geq0$, whose
total complex is denoted by $\mathcal{Y}\widehat{\otimes}_{A}\mathcal{P}%
^{\bullet}$. Thus
\[
\circ\widehat{\otimes}_{A}\mathcal{P}^{\bullet}:\overline{\operatorname*{mod}%
\text{-}A}\longrightarrow\overline{\mathfrak{Fs}},\quad\mathcal{Y\longmapsto
Y}\widehat{\otimes}_{A}\mathcal{P}^{\bullet},\quad\overline{\varphi}%
\mapsto\overline{\varphi}\otimes_{A}1_{\mathcal{P}^{\bullet}}%
\]
is the well defined functor. If $H^{k}$ is the $k$-th (co)homology functor on
$\overline{\mathfrak{Fs}}$ then as in \cite[6.2]{HelHom}, we define
$\operatorname{Tor}_{A}^{k}\left(  \mathcal{Y},X\right)  $ to be the composite
functor $H^{k}\left(  \circ\widehat{\otimes}_{A}\mathcal{P}^{\bullet}\right)
$ applied to $\mathcal{Y}$, that is, $\operatorname{Tor}_{A}^{k}\left(
\mathcal{Y},X\right)  =H^{k}\left(  \mathcal{Y}\widehat{\otimes}%
_{A}\mathcal{P}^{\bullet}\right)  $. The homology groups $\operatorname{Tor}%
_{A}^{k}\left(  \mathcal{Y},X\right)  $ do not depend on the particular choice
of $\mathcal{P}^{\bullet}$, and $\operatorname{Tor}_{A}^{k}\left(
\mathcal{Y},X\right)  =\operatorname{Tor}_{k}^{A}\left(  Y,X\right)  $
whenever $\mathcal{Y}$ is reduced to a right $A$-module $Y$. In the latter
case, we have $\mathcal{Y}\widehat{\otimes}_{A}\mathcal{P}^{\bullet}%
=\overline{\mathcal{Y}\widehat{\otimes}_{A}\mathcal{P}^{\bullet}%
}=Y\widehat{\otimes}_{A}\mathcal{P}^{\bullet}$. If $\mathcal{Y}=A$ then
$\mathcal{Y}\widehat{\otimes}_{A}\mathcal{P}^{\bullet}=A\widehat{\otimes}%
_{A}\mathcal{P}^{\bullet}=\mathcal{P}^{\bullet}$. Therefore
$\operatorname{Tor}_{A}^{0}\left(  A,X\right)  =X$ and $\operatorname{Tor}%
_{A}^{k}\left(  A,X\right)  =\left\{  0\right\}  $, $k\neq0$.

By \textit{an augmentation of} $\mathcal{Y}$ we mean a morphism of the right
$A$-modules $\eta:A\rightarrow Y_{0}$ so that $A\rightarrow\mathcal{Y}$ is a
cochain complex. The pair $\left(  \mathcal{Y},\eta\right)  $ is called
\textit{an augmented complex of the right }$A$\textit{-modules.} The morphisms
of the augmented complexes are defined in the standard way. An augmentation
$\eta$ defines a morphism $\overline{\eta}:A\rightarrow\mathcal{Y}$ in
$\overline{\operatorname*{mod}\text{-}A}$ as%
\[%
\begin{tabular}
[c]{lll}%
$\vdots$ &  & $\vdots$\\
$\uparrow$ &  & $\uparrow$\\
$0$ & $\longrightarrow$ & $Y_{1}$\\
$\uparrow$ &  & $\uparrow$\\
$A$ & $\overset{\eta}{\longrightarrow}$ & $Y_{0}$\\
$\uparrow$ &  & $\uparrow$\\
$0$ &  & $0$%
\end{tabular}
\
\]
The morphism $\overline{\eta}:A\rightarrow\mathcal{Y}$ in turn defines a
morphism $\overline{\eta}\otimes_{A}\mathcal{P}^{\bullet}:\overline
{A\widehat{\otimes}_{A}\mathcal{P}^{\bullet}}\longrightarrow\overline
{\mathcal{Y}\widehat{\otimes}_{A}\mathcal{P}^{\bullet}}$ of the bicomplexes in
terms of the following commutative diagram%
\[%
\begin{array}
[c]{ccccccccc}%
A\widehat{\otimes}_{A}P^{0} & \longrightarrow & Y^{0}\widehat{\otimes}%
_{A}P^{0} & \rightarrow & Y^{1}\widehat{\otimes}_{A}P^{0} & \rightarrow\cdots
& \rightarrow & Y^{n}\widehat{\otimes}_{A}P^{0} & \rightarrow\cdots\\
\uparrow &  & \uparrow & \ddots & \uparrow &  &  & \uparrow & \\
\vdots &  & \vdots &  & \vdots &  &  &  & \\
\uparrow &  & \uparrow &  & \uparrow &  &  & \uparrow & \\
A\widehat{\otimes}_{A}P^{-n+1} & \longrightarrow & Y^{0}\widehat{\otimes}%
_{A}P^{-n+1} & \rightarrow & Y^{1}\widehat{\otimes}_{A}P^{-n+1} & \ddots &  &
Y^{n}\widehat{\otimes}_{A}P^{-n+1} & \rightarrow\cdots\\
\uparrow &  & \uparrow &  & \uparrow &  &  & \uparrow & \\
A\widehat{\otimes}_{A}P^{-n} & \longrightarrow & Y^{0}\widehat{\otimes}%
_{A}P^{-n} & \rightarrow & Y^{1}\widehat{\otimes}_{A}P^{-n} & \rightarrow
\cdots & \rightarrow & Y^{n}\widehat{\otimes}_{A}P^{-n} & \rightarrow\cdots
\end{array}
\]
The related morphism $A\widehat{\otimes}_{A}\mathcal{P}^{\bullet
}\longrightarrow\mathcal{Y}\widehat{\otimes}_{A}\mathcal{P}^{\bullet}$ (or
$\mathcal{P}^{\bullet}\longrightarrow\mathcal{Y}\widehat{\otimes}%
_{A}\mathcal{P}^{\bullet}$) of the total complexes are given by the diagonals
\[
A\widehat{\otimes}_{A}P^{-n}\rightarrow Y^{0}\widehat{\otimes}_{A}%
P^{-n},\text{ }A\widehat{\otimes}_{A}P^{-n+1}\rightarrow%
\begin{array}
[c]{c}%
Y^{0}\widehat{\otimes}_{A}P^{-n+1}\\
\oplus\\
Y^{1}\widehat{\otimes}_{A}P^{-n}%
\end{array}
,\text{ \ldots, }A\widehat{\otimes}_{A}P^{0}\rightarrow%
\begin{array}
[c]{c}%
Y^{0}\widehat{\otimes}_{A}P^{0}\\
\oplus\\
\vdots\\
\oplus\\
Y^{n}\widehat{\otimes}_{A}P^{-n}%
\end{array}
.
\]
In particular, there are morphisms
\[
H^{k}\left(  \overline{\eta}\otimes_{A}1_{\mathcal{P}^{\bullet}}\right)
:\operatorname{Tor}_{A}^{k}\left(  A,X\right)  \rightarrow\operatorname{Tor}%
_{A}^{k}\left(  \mathcal{Y},X\right)  ,\quad k\in\mathbb{Z}\text{.}%
\]
Taking into account that $\operatorname{Tor}_{A}^{k}\left(  A,X\right)
=\left\{  0\right\}  $ for all $k\neq0$, we conclude that $H^{k}\left(
\overline{\eta}\otimes_{A}1_{\mathcal{P}^{\bullet}}\right)  =0$, $k\neq0$, and%
\[
\eta_{\ast}=H^{0}\left(  \overline{\eta}\otimes_{A}1_{\mathcal{P}^{\bullet}%
}\right)  :X\longrightarrow\operatorname{Tor}_{A}^{0}\left(  \mathcal{Y}%
,X\right)  ,\quad\eta_{\ast}\left(  \varepsilon\left(  z\right)  \right)
=\left(  \eta\left(  1_{A}\right)  \otimes_{A}z\right)  ^{\sim}.
\]
An augmented complex of the right $A$-modules $\left(  \mathcal{Y}%
,\eta\right)  $ is said to be \textit{dominating over} $X$, in this case we
used to write $\left(  \mathcal{Y},\eta\right)  \gg X$, if $\eta_{\ast}$ is a
topological isomorphism and $\operatorname{Tor}_{A}^{k}\left(  \mathcal{Y}%
,X\right)  =\left\{  0\right\}  $ for all $k$, $k\neq0$. If $\mathcal{Y}$ is
reduced to a right $A$-module $Y$ and $\eta:A\rightarrow Y$ is a right
$A$-module morphism, then $\left(  Y,\eta\right)  \gg X$ means that
$\operatorname{Tor}_{k}^{A}\left(  Y,X\right)  =\left\{  0\right\}  $ for all
$k$, $k\neq0$, and the morphism $\eta_{\ast}:X\rightarrow\operatorname{Tor}%
_{0}^{A}\left(  Y,X\right)  $ is a topological isomorphism. Thus
$Y\widehat{\otimes}_{A}\mathcal{P}^{\bullet}\longrightarrow X\rightarrow0$
turns out to be an exact complex.

Now let $Y=\mathcal{B}$ be a unital Fr\'{e}chet algebra with a morphism
$\iota:A\rightarrow\mathcal{B}$ (a continuous algebra homomorphism), thereby
$\mathcal{B}$ is an object of the category $A$-$\operatorname*{mod}$-$A$
(through $\iota$). In this case, $\left(  \mathcal{B},\iota\right)  \gg X$
means that $\operatorname{Tor}_{k}^{A}\left(  \mathcal{B},X\right)  =\left\{
0\right\}  $ for all $k$, $k\neq0$, $\operatorname{Tor}_{0}^{A}\left(
\mathcal{B},X\right)  $ is Hausdorff and $X=\mathcal{B}\widehat{\otimes}_{A}X$
up to an isomorphism in $\mathfrak{Fs}$ (see \cite[Lemma 2.2]{DosJOT10}).

\begin{lemma}
\label{lemDom12}Let $\iota:A\rightarrow\mathcal{B}$ be a localization of
Fr\'{e}chet algebras and $X$ a left $A$-module. Then $\left(  \mathcal{B}%
,\iota\right)  \gg X$ holds if and only if the $A$-action over $X$ is lifted
to a left Fr\'{e}chet $\mathcal{B}$-module action on $X$ through $\iota$.
\end{lemma}

\begin{proof}
Notice that the $A$-action over $X$ is lifted to a left Fr\'{e}chet
$\mathcal{B}$-module action on $X$ through $\iota$ iff $X=\mathcal{B}%
\widehat{\otimes}_{A}X$ up to an isomorphism in $\mathfrak{Fs}$. In
particular, $\left(  \mathcal{B},\iota\right)  \gg X$ implies that $X$ is a
left Fr\'{e}chet $\mathcal{B}$-module. Conversely, suppose that $X$ is a left
$\mathcal{B}$-module that lifts its $A$-module structure. By assumption
$\iota:A\rightarrow\mathcal{B}$ is a localization, therefore
$\operatorname{Tor}_{k}^{A}\left(  \mathcal{B},X\right)  =\operatorname{Tor}%
_{k}^{\mathcal{B}}\left(  \mathcal{B},X\right)  =\left\{  0\right\}  $ for all
$k$, $k\neq0$, and $\operatorname{Tor}_{0}^{A}\left(  \mathcal{B},X\right)
=\operatorname{Tor}_{0}^{\mathcal{B}}\left(  \mathcal{B},X\right)
=\mathcal{B}\widehat{\otimes}_{\mathcal{B}}X=X$ up to a topological
isomorphism (see Subsection \ref{subsecTL}). It follows that
$\operatorname{Tor}_{k}^{A}\left(  \mathcal{B},X\right)  =\left\{  0\right\}
$ for all $k$, $k\neq0$, $\operatorname{Tor}_{0}^{A}\left(  \mathcal{B}%
,X\right)  $ is Hausdorff. Hence $\left(  \mathcal{B},\iota\right)  \gg X$ holds.
\end{proof}

In particular, if $\mathcal{B}=A$ and $\iota=\iota_{A}$ is the identity map,
then $\left(  A,\iota_{A}\right)  \gg X$ holds automatically. The argument can
be generalized (see \cite[Theorem 2.4]{DosJOT10}) to the case of a dominating
complex in the following way.

\begin{theorem}
\label{tHTf}Let $\iota:A\rightarrow\mathcal{B}$ be a morphism of the
Fr\'{e}chet algebras, $X$ a finite-projective left $A$-module, and let
$\left(  \mathcal{Y},\eta\right)  $ be an augmented complex of the right
$A$-modules such that $\mathcal{Y}$ is an object of the category
$\overline{\mathcal{B}\text{-}\operatorname*{mod}\text{-}A}$, and
$\overline{\eta}:A\rightarrow\mathcal{Y}$ is a morphism of the left
$A$-modules too. If $\left(  \mathcal{Y},\eta\right)  \gg X$ then $X$ turns
out to be a left $\mathcal{B}$-module such that its $A$-module structure via
$\iota$ is reduced to the original one.
\end{theorem}

Notice that $\operatorname{Tor}_{A}^{0}\left(  \mathcal{Y},X\right)  $
possesses a natural left $\mathcal{B}$-module structure, which can be drawn
back to $X$ by means of the topological isomorphism $\eta_{\ast}$. Theorem
\ref{tHTf} plays the central role in Taylor-Helemskii-Putinar framework of the
noncommutative functional calculus developed in \cite{DosJOT10}. The following
assertion is a bit modified version of one from \cite[Proposition 2.7
(b)]{Tay2}.

\begin{proposition}
\label{propDom12}Let $A$ be a Fr\'{e}chet algebra of finite type, $X$ a left
$A$-module, and let $\mathcal{A}$ and $\mathcal{B}$ be nuclear Fr\'{e}chet
$A$-algebras such that $\left(  \mathcal{A},\iota\right)  \gg\mathcal{B}$,
where $\iota:A\rightarrow\mathcal{A}$ is the related homomorphism. If the
transversality relation $\mathcal{A}\perp X$ holds, then $\mathcal{B}\perp X$
holds too. In particular, if $\gamma\in\operatorname{Spec}\left(
\mathcal{A}\right)  $ with the trivial module $\mathcal{B}=\mathbb{C}\left(
\gamma\right)  $, then $\mathcal{A}\perp X$ implies that $\mathbb{C}\left(
\gamma\right)  \perp X$.
\end{proposition}

\begin{proof}
As above in Subsection \ref{subsecTrans}, we assume that $\mathcal{R}%
^{\bullet}=A\widehat{\otimes}\mathfrak{e}^{\bullet}\widehat{\otimes}A$ is a
finite free $A$-bimodule resolution of $A$ with nuclear Fr\'{e}chet spaces
$\mathfrak{e}^{\bullet}=\left\{  E_{i}:-n\leq i\leq0\right\}  $. Since
$\mathcal{B}$ is a left $A$-module, it follows that $\mathcal{R}^{\bullet
}\widehat{\otimes}_{A}\mathcal{B}$ is a finite free resolution of
$\mathcal{B}$ in $A$-mod. The condition $\left(  \mathcal{A},\iota\right)
\gg\mathcal{B}$ means that the cochain complex
\begin{equation}
\mathcal{A}\widehat{\otimes}_{A}\mathcal{R}^{\bullet}\widehat{\otimes}%
_{A}\mathcal{B\rightarrow B}\rightarrow0 \label{ARAB}%
\end{equation}
is exact. Note that $\mathcal{A}\widehat{\otimes}_{A}\mathcal{R}^{\bullet
}\widehat{\otimes}_{A}\mathcal{B}=\mathcal{A}\widehat{\otimes}\mathfrak{e}%
^{\bullet}\widehat{\otimes}\mathcal{B}$ is a complex from $\overline
{\text{mod-}A}$, and it consists of nuclear Fr\'{e}chet spaces.

Assume that the transversality relation $\mathcal{A}\perp X$ holds. Let us
prove that $\left(  \mathcal{A}\widehat{\otimes}E_{k}\widehat{\otimes
}\mathcal{B}\right)  \perp X$ holds for all $k$, that is, $\left(
\mathcal{A}\widehat{\otimes}\mathfrak{e}^{\bullet}\widehat{\otimes}%
\mathcal{B}\right)  \perp X$. As in the case of $\mathcal{B}$, the complex
$\mathcal{R}^{\bullet}\widehat{\otimes}_{A}X=A\widehat{\otimes}\mathfrak{e}%
^{\bullet}\widehat{\otimes}X\mathcal{\ }$denoted by $\mathcal{R}_{X}^{\bullet
}$ provides a free resolution of $X$ in $A$-mod. Since $\mathcal{A}\perp X$,
it follows that $\mathcal{A}\widehat{\otimes}_{A}\mathcal{R}_{X}^{\bullet}$
remains exact and $\mathcal{A}\widehat{\otimes}_{A}\mathcal{R}_{X}^{\bullet
}=\mathcal{A}\widehat{\otimes}\mathfrak{e}^{\bullet}\widehat{\otimes}X$. But
$\mathcal{B}$ and $E_{k}$ are nuclear spaces, therefore so is $\mathcal{B}%
\widehat{\otimes}E_{k}$, and $\mathcal{B}\widehat{\otimes}E_{k}%
\widehat{\otimes}\left(  \mathcal{A}\widehat{\otimes}_{A}\mathcal{R}%
_{X}^{\bullet}\right)  $ remains exact too. Taking into account that
\begin{align*}
\mathcal{B}\widehat{\otimes}E_{k}\widehat{\otimes}\left(  \mathcal{A}%
\widehat{\otimes}_{A}\mathcal{R}_{X}^{\bullet}\right)   &  =\mathcal{B}%
\widehat{\otimes}E_{k}\widehat{\otimes}\mathcal{A}\widehat{\otimes
}\mathfrak{e}^{\bullet}\widehat{\otimes}X=\mathcal{A}\widehat{\otimes}%
E_{k}\widehat{\otimes}\mathcal{B}\widehat{\otimes}\mathfrak{e}^{\bullet
}\widehat{\otimes}X=\mathcal{A}\widehat{\otimes}E_{k}\widehat{\otimes
}\mathcal{B}\widehat{\otimes}_{A}A\widehat{\otimes}\mathfrak{e}^{\bullet
}\widehat{\otimes}X\\
&  =\mathcal{A}\widehat{\otimes}E_{k}\widehat{\otimes}\mathcal{B}%
\widehat{\otimes}_{A}\mathcal{R}_{X}^{\bullet},
\end{align*}
we conclude that $\left(  \mathcal{A}\widehat{\otimes}E_{k}\widehat{\otimes
}\mathcal{B}\right)  \perp X$ holds for each $k$. Thus $X$ is in the
transversality relation with respect to all members of the finite length
complex $\mathcal{A}\widehat{\otimes}_{A}\mathcal{R}^{\bullet}\widehat{\otimes
}_{A}\mathcal{B}$ from (\ref{ARAB}). Using \cite[Proposition 2.4]{Tay2} (see
also \cite[Corollary 3.1.16]{EP}), we deduce that $\mathcal{B}\perp X$ holds too.
\end{proof}

\begin{corollary}
\label{corDom12}Let $A$ be a Fr\'{e}chet algebra of finite type, $\mathcal{A}$
and $\mathcal{B}$ nuclear Fr\'{e}chet $A$-algebras such that $\iota
:A\rightarrow\mathcal{A}$ is a localization and there is an $A$-morphism
$\mathcal{A\rightarrow B}$. If $\mathcal{A}\perp X$ then $\mathcal{B}\perp X$
holds automatically for every left $A$-module $X$.
\end{corollary}

\begin{proof}
Note that the left $A$-action on $\mathcal{B}$ is lifted to a left
$\mathcal{A}$-module action through the homomorphism $\mathcal{A\rightarrow
B}$. By Lemma \ref{lemDom12}, the dominance property $\left(  \mathcal{A}%
,\iota\right)  \gg\mathcal{B}$ holds. It remains to use Proposition
\ref{propDom12}.
\end{proof}

\section{The spectra of a left Fr\'{e}chet module\label{sectionFM}}

In this section we introduce spectra of a left Fr\'{e}chet $A$-module with
respect to an $A$-category, and prove the related spectral mapping properties.

\subsection{$A$-category of Fr\'{e}chet algebras}

Let $A$ be a unital Fr\'{e}chet algebra, and let $\mathcal{S}$ be a category
of some Fr\'{e}chet $A$-algebras called\textit{ an }$A$\textit{-category}.
Thus, there are compatible morphisms $A\rightarrow\mathcal{A}$ of the
Fr\'{e}chet algebras for all objects $\mathcal{A}$ from $\mathcal{S}$, and
briefly we write $A\rightarrow\mathcal{S}$. We also assume that the trivial
algebra $\left\{  0\right\}  $ and the trivial morphisms belong to
$\mathcal{S}$. So is the category $\mathcal{T}$, whose objects are only
trivial modules $\mathbb{C}\left(  \lambda\right)  $, $\lambda\in
\operatorname{Spec}\left(  A\right)  $ with their trivial morphisms, whenever
$\operatorname{Spec}\left(  A\right)  $ is not empty. An $A$-category
$\mathcal{S}$ is a subcategory of $\mathfrak{Fa}$ (or $A$-$\operatorname{mod}%
$-$A$). For every object $\mathcal{A}$ of $\mathcal{S}$ and $\mu
\in\operatorname{Spec}\left(  \mathcal{A}\right)  $, we deduce that
$\mathbb{C}\left(  \mu\right)  $ is a trivial $A$-module too, which is
included into $\mathcal{T}$. We say that $\mathcal{S}$ is \textit{a point
complete }$A$\textit{-category }if it contains all its trivial modules. In
this case, $\mathcal{S\cap T}$ is a new (nonempty) $A$-subcategory of
$\mathcal{S}$.

One can complete $\mathcal{S}$ by adding up all the trivial modules of its
objects with their canonical morphisms $\mathcal{A}\rightarrow\mathbb{C}%
\left(  \lambda\right)  $. The completion of $\mathcal{S}$ denoted by
$\mathcal{S}^{\sim}$ is called \textit{the point-completion of }$\mathcal{S}$.
An object $\mathcal{B}$ of $\mathcal{S}$ is said to be \textit{a local object
}if $A\rightarrow\mathcal{B}$ is a localization. If the base algebra $A$ is
included into $\mathcal{S}$ we say that $\mathcal{S}$ is \textit{a unital }%
$A$-\textit{category. }The category $\mathcal{T}$ is not unital in the case of
a nontrivial algebra $A$. If $\mathcal{S}$ consists of nuclear Fr\'{e}chet
algebras, we say that $\mathcal{S}$ is a \textit{nuclear }$A$%
\textit{-category. }

By a morphism of an $A$-category $\mathcal{S}$ into another $B$-category
$\mathcal{G}$ we mean a covariant functor $F:\mathcal{S\rightarrow G}$ along
with a family of the compatible morphisms $f:A\rightarrow B$ and
$\mathcal{A}\rightarrow F\left(  \mathcal{A}\right)  $ for all objects
$\mathcal{A}$ of $\mathcal{S}$. Thus, we come up with the following functor
transformation
\[%
\begin{array}
[c]{ccc}%
\mathcal{S} & \overset{F}{\longrightarrow} & \mathcal{G}\\
\uparrow &  & \uparrow\\
A & \overset{f}{\longrightarrow} & B
\end{array}
\]
along with the acting morphisms $\mathcal{A}\rightarrow F\left(
\mathcal{A}\right)  $ of the Fr\'{e}chet algebras.

\subsection{The Alexandrov topology of an $A$-category\label{subsecAT}}

Let $\mathcal{S}$ be an $A$-category. A subcategory $U\subseteq\mathcal{S}$ is
said to be \textit{open} if for every object $\mathcal{A}$ in $U$, it contains
every object $\mathcal{B}$ with a connecting morphism $\mathcal{A\rightarrow
B}$ in $\mathcal{S}$. In particular, $U$ contains all trivial modules
$\mathbb{C}\left(  \mu\right)  $, $\mu\in\operatorname{Spec}\left(
\mathcal{A}\right)  $ whenever $\mathcal{A}$ is an object of $U$ and
$\mathcal{S}$ is point complete. One can easily verify that the family of all
open subcategories defines a topology in the set of all objects of
$\mathcal{S}$. The intersection of all open neighborhoods of a fixed algebra
$\mathcal{A}$ is the set $U_{\mathcal{A}}$ of all objects $\mathcal{B}$ with
the morphisms $\mathcal{A\rightarrow B}$ in $\mathcal{S}$. Indeed, if
$\mathcal{B}$ is an object of $U_{\mathcal{A}}$ with a morphism
$\mathcal{B\rightarrow C}$ in $\mathcal{S}$ then $\mathcal{A\rightarrow
B\rightarrow C}$ provides a morphism $\mathcal{A\rightarrow C}$ from the
category $\mathcal{S}$, therefore $\mathcal{C}$ is included into
$U_{\mathcal{A}}$. The algebra $\mathcal{A}$ itself is included into
$U_{\mathcal{A}}$ out of the identity morphism $\mathcal{A\rightarrow A}$.
Hence $U_{\mathcal{A}}$ is open and contains $\mathcal{A}$. If $U$ is an open
subset containing $\mathcal{A}$ then it should contain all morphisms
$\mathcal{A\rightarrow B}$ by its very definition, that is, $U_{\mathcal{A}%
}\subseteq U$. Thus $\mathcal{S}$ is an Alexandrov topological space, which
means every $\mathcal{A}$ has the least neighborhood $U_{\mathcal{A}}$ being
an open quasicompact set \cite{Spe}.

Notice that $U_{\mathbb{C}\left(  \lambda\right)  }=\left\{  \mathbb{C}\left(
\lambda\right)  \right\}  $ for a possible trivial module $\mathbb{C}\left(
\lambda\right)  $ from $\mathcal{S}$. The topological space $\mathcal{S\cap
T}$ (in particular, $\mathcal{T}$ itself) is discreet.

Further, if $\mathcal{B}$ is an object of $\mathcal{S}$, then the closure of
the singleton $\left\{  \mathcal{B}\right\}  $ consists of those objects
$\mathcal{A}$ which admit morphisms $\mathcal{A\rightarrow B}$. The trivial
module $\mathbb{C}\left(  \lambda\right)  $ is closed in $\mathcal{S\cap T}$
whereas its closure in $\mathcal{S}$ consists of those algebras $\mathcal{A}$
such that $\lambda\in\operatorname{Spec}\left(  \mathcal{A}\right)  $. If
$\mathcal{S}$ is unital, then $A$ is included into $\mathcal{S}$ and $U_{A}$
consists of all object from $\mathcal{S}$, and $\mathcal{S}$ is a quasicompact
space with the closed point $\left\{  A\right\}  $.

If $F:\mathcal{S\rightarrow G}$ is a morphism of an $A$-category into a
$B$-category, then $F\left(  U_{\mathcal{A}}\right)  \subseteq U_{F\left(
\mathcal{A}\right)  }$ for every object $\mathcal{A}$ of $\mathcal{S}$, which
means that $F$ is continuous with respect to the related Alexandrov topologies.

\subsection{A point basis for an $A$-category\label{subsecPB}}

The original topology of an $A$-category $\mathcal{S}$ can be extended to the
point completion $\mathcal{S}^{\sim}$ automatically. In this case,
$U_{\mathcal{A}}$ in $\mathcal{S}^{\sim}$ contains all possible trivial
modules $\mathbb{C}\left(  \lambda\right)  $, $\lambda\in\operatorname{Spec}%
\left(  \mathcal{A}\right)  $. Every $\mathbb{C}\left(  \lambda\right)  $ from
$\mathcal{S}^{\sim}\cap\mathcal{T}$ has a neighborhood filter base (in
$\mathcal{S}^{\sim}$) of all $U_{\mathcal{A}}$ with $\lambda\in
\operatorname{Spec}\left(  \mathcal{A}\right)  $, $\mathcal{A\in S}$. A
countable family $\mathfrak{t=}\left\{  \mathcal{B}\right\}  $ of the objects
of an $A$-category $\mathcal{S}$ is said to be \textit{a point basis for
}$\mathcal{S}$ if $\left\{  U_{\mathcal{B}}:\mathcal{B}\in\mathfrak{t}%
\right\}  $ is a topology base in the point completion $\mathcal{S}^{\sim}$ of
$\mathcal{S}$. Thus for every trivial module $\mathbb{C}\left(  \lambda
\right)  \in\mathcal{S}^{\sim}$ the family $\left\{  U_{\mathcal{B}%
}:\mathcal{B}\in\mathfrak{t},\mathbb{C}\left(  \lambda\right)  \in
U_{\mathcal{B}}\right\}  $ is a neighborhood filter base of $\mathbb{C}\left(
\lambda\right)  $ in $\mathcal{S}^{\sim}$. If a point basis $\mathfrak{t}$
consists of local algebras, we say that $\mathfrak{t}$ is a
\textit{localizing} \textit{basis for }$\mathcal{S}^{\sim}$. The fact that a
point basis is countable, it is not necessary below in this section.

\subsection{The $\mathcal{S}$-spectrum of a left $A$-module\label{subsecSSA}}

Let $\mathcal{S}$ be an $A$-category and let $X$ be a (nonzero) left
Fr\'{e}chet $A$-module. We use the transversality relation (see Subsection
\ref{subsecTrans}) of the algebras from $\mathcal{S}$ as right $A$-modules
versus the left $A$-module $X$ to define the spectrum of $X$.

\begin{definition}
\label{def11}The \textit{resolvent set} $\operatorname{res}\left(
\mathcal{S},X\right)  $ \textit{of the }$A$\textit{-module }$X$\textit{ with
respect to an }$A$-\textit{category }$\mathcal{S}$ is defined to be the set of
those objects $\mathcal{A}$ of $\mathcal{S}$ such that $\mathcal{B}\perp X$
for every morphism $\mathcal{A\rightarrow B}$ in $\mathcal{S}$. The complement
$\sigma\left(  \mathcal{S},X\right)  =\mathcal{S}\backslash\operatorname{res}%
\left(  \mathcal{S},X\right)  $ is called \textit{the spectrum of the }%
$A$\textit{-module} $X$\textit{ with respect to} $\mathcal{S}$. The set
$\sigma\left(  \mathcal{S}^{\sim}\cap\mathcal{T},X\right)  =\left(
\mathcal{S}^{\sim}\cap\mathcal{T}\right)  \backslash\operatorname{res}\left(
\mathcal{S}^{\sim}\cap\mathcal{T},X\right)  $ is called \textit{the Taylor
spectrum of} \textit{the }$A$\textit{-module} $X$\textit{ with respect to}
$\mathcal{S}$.
\end{definition}

In the case of $\mathcal{S=T}$, the set $\sigma\left(  \mathcal{T},X\right)
=\mathcal{T}\backslash\operatorname{res}\left(  \mathcal{T},X\right)  $ is
called \textit{the Taylor spectrum of the }$A$\textit{-module} $X$, and we use
the notation $\sigma\left(  A,X\right)  $ instead of $\sigma\left(
\mathcal{T},X\right)  $.

Notice that if $\mathcal{A\in}\operatorname{res}\left(  \mathcal{S},X\right)
$ then $\mathcal{A}\perp X$ holds too out of the identity morphism
$\mathcal{A\rightarrow A}$. By its very definition, $\operatorname{res}\left(
\mathcal{S},X\right)  $ is open, therefore the spectrum $\sigma\left(
\mathcal{S},X\right)  $ is a closed set. The transversality $A\perp X$ for the
base algebra $A$ does not hold, for $\operatorname{Tor}_{0}^{A}\left(
A,X\right)  =X\neq\left\{  0\right\}  $. It follows that $A\in\sigma\left(
\mathcal{S},X\right)  $ whenever $\mathcal{S}$ is unital. In the case of
$\mathcal{S=T}$, we obtain that $\operatorname{res}\left(  \mathcal{T}%
,X\right)  $ consists of those $\lambda\in\operatorname{Spec}\left(  A\right)
$ such that $\mathbb{C}\left(  \lambda\right)  \perp X$. Moreover,
$\operatorname{res}\left(  \mathcal{S}^{\sim},X\right)  \cap\mathcal{T}%
=\operatorname{res}\left(  \mathcal{S}^{\sim}\cap\mathcal{T},X\right)  $ or
$\sigma\left(  \mathcal{S}^{\sim}\cap\mathcal{T},X\right)  =\sigma\left(
\mathcal{S}^{\sim},X\right)  \cap\mathcal{T}$, and $\sigma\left(
\mathcal{S}^{\sim}\cap\mathcal{T},X\right)  =\sigma\left(  \mathcal{T}%
,X\right)  $ whenever $\mathcal{S}$ is unital.

Along with the resolvent set $\operatorname{res}\left(  \mathcal{T},X\right)
$ in $\mathcal{T}$, we consider the following set $\operatorname{res}%
_{\operatorname{P}}\left(  \mathcal{S},X\right)  =\operatorname{res}\left(
\mathcal{S},X\right)  ^{\sim}\cap\mathcal{T}$, where $\operatorname{res}%
\left(  \mathcal{S},X\right)  ^{\sim}$ is the point completion of the open
subcategory $\operatorname{res}\left(  \mathcal{S},X\right)  $. Thus
$\mathbb{C}\left(  \lambda\right)  $ is an object of $\operatorname{res}%
_{\operatorname{P}}\left(  \mathcal{S},X\right)  $ iff $\lambda\in
\operatorname{Spec}\left(  \mathcal{A}\right)  $ for a certain object
$\mathcal{A}\in\operatorname{res}\left(  \mathcal{S},X\right)  $.

\begin{definition}
\label{def12}The set $\sigma_{\operatorname{P}}\left(  \mathcal{S},X\right)
=\left(  \mathcal{S}^{\sim}\cap\mathcal{T}\right)  \backslash
\operatorname{res}_{\operatorname{P}}\left(  \mathcal{S},X\right)  $ is called
\textit{the Putinar spectrum of the }$A$-\textit{module }$X$ \textit{with
respect to} $\mathcal{S}$.
\end{definition}

One can easily see that $\sigma_{\operatorname{P}}\left(  \mathcal{T}%
,X\right)  =\sigma\left(  \mathcal{T},X\right)  =\sigma\left(  A,X\right)  $.
In general, these spectra from Definitions \ref{def11} and \ref{def12} are
distinct (closed) subsets. Nevertheless, there are some key inclusions between
them in some particular cases clarified below. Moreover, $\sigma
_{\operatorname{P}}\left(  \mathcal{S},X\right)  $ and $\sigma\left(
\mathcal{S}^{\sim}\cap\mathcal{T},X\right)  $ do relate to the points or the
trivial modules from the point completion $\mathcal{S}^{\sim}$, whereas
$\sigma\left(  \mathcal{S},X\right)  $ consists of the algebras from the
$A$-category $\mathcal{S}$.

If the local transversality properties $\mathbb{C}\left(  \gamma\right)  \perp
X$, $\gamma\in\operatorname{Spec}\left(  \mathcal{B}\right)  $ imply the
global one $\mathcal{B}\perp X$ for every local object $\mathcal{B}$ of
$\mathcal{S}$, then we say that $X$ is \textit{an }$\mathcal{S}$\textit{-local
left }$A$\textit{-module}. In this case, we assume automatically that
$\operatorname{Spec}\left(  \mathcal{B}\right)  $ is not empty for every local
object $\mathcal{B}$ of $\mathcal{S}$. Practically, the $\mathcal{S}$-local
left Fr\'{e}chet modules are Banach $A$-modules indeed.

\begin{proposition}
\label{propLocG}Let $A$ be a finite type algebra, $\mathcal{S}$ a nuclear
$A$-category with its point completion $\mathcal{S}^{\sim}$, which has a
localizing basis $\mathfrak{t}$, and let $X$ be a \textit{left }%
$A$\textit{-module. }Then
\[
\sigma\left(  \mathcal{S}^{\sim}\cap\mathcal{T},X\right)  \subseteq
\sigma_{\operatorname{P}}\left(  \mathcal{S},X\right)  .
\]
Moreover, if $X$ is an $\mathcal{S}$\textit{-local left }$A$\textit{-module,
then }%
\[
\sigma_{\operatorname{P}}\left(  \mathcal{S},X\right)  \subseteq\sigma\left(
\mathcal{S}^{\sim}\cap\mathcal{T},X\right)  ^{-}%
\]
with respect to the topology of the point completion $\mathcal{S}^{\sim}$.
\end{proposition}

\begin{proof}
Take a trivial module $\mathbb{C}\left(  \gamma\right)  $ from
$\operatorname{res}_{\operatorname{P}}\left(  \mathcal{S},X\right)  $. By
Definition \ref{def12}, we have $\gamma\in\operatorname{Spec}\left(
\mathcal{A}\right)  $ for a certain object $\mathcal{A}\in\operatorname{res}%
\left(  \mathcal{S},X\right)  $. It follows that $\mathbb{C}\left(
\gamma\right)  \in U_{\mathcal{A}}$ in the point completion $\mathcal{S}%
^{\sim}$. Then $\mathbb{C}\left(  \gamma\right)  \in U_{\mathcal{B}}$ and
$U_{\mathcal{B}}\subseteq U_{\mathcal{A}}$ for a certain $\mathcal{B}%
\in\mathfrak{t}$ (see Subsection \ref{subsecPB}), that is, there are morphisms
$\mathcal{A\rightarrow B}\rightarrow\mathbb{C}\left(  \gamma\right)  $ in
$\mathcal{S}^{\sim}$. But $\mathcal{B}$ is nuclear, $A\rightarrow\mathcal{B}$
is a localization, and the transversality $\mathcal{B}\perp X$ holds. Using
Corollary \ref{corDom12} (see also Proposition \ref{propDom12}), we deduce
that $\mathbb{C}\left(  \gamma\right)  \perp X$ holds too, that is,
$\mathbb{C}\left(  \gamma\right)  \in\operatorname{res}\left(  \mathcal{S}%
^{\sim}\cap\mathcal{T},X\right)  $. Hence $\sigma\left(  \mathcal{S}^{\sim
}\cap\mathcal{T},X\right)  \subseteq\sigma_{\operatorname{P}}\left(
\mathcal{S},X\right)  $.

Now assume that $X$ is an $\mathcal{S}$-local left $A$-module, and take a
trivial $A$-module $\mathbb{C}\left(  \gamma\right)  $ from $\sigma
_{\operatorname{P}}\left(  \mathcal{S},X\right)  $. If $\mathbb{C}\left(
\gamma\right)  $ stays out of the closure $\sigma\left(  \mathcal{S}^{\sim
}\cap\mathcal{T},X\right)  ^{-}$ (in $\mathcal{S}^{\sim}$), then
$U_{\mathcal{A}}\cap\sigma\left(  \mathcal{S}^{\sim}\cap\mathcal{T},X\right)
=\varnothing$ for some object $\mathcal{A}$ of the category $\mathcal{S}$.
Since $\mathfrak{t}$ is a point base for the point completion $\mathcal{S}%
^{\sim}$, it follows that $\mathbb{C}\left(  \gamma\right)  \in U_{\mathcal{B}%
}$ and $U_{\mathcal{B}}\subseteq U_{\mathcal{A}}$ with $\mathcal{B}%
\in\mathfrak{t}$. In particular, $U_{\mathcal{B}}\cap\mathcal{T\subseteq
}\operatorname{res}\left(  \mathcal{S}^{\sim}\cap\mathcal{T},X\right)  $,
which means that $\mathbb{C}\left(  \lambda\right)  \perp X$ for all
$\lambda\in\operatorname{Spec}\left(  \mathcal{B}\right)  $. But $\mathcal{B}$
is local and $X$ is an $\mathcal{S}$-local left $A$-module, therefore
$\mathcal{B}\perp X$ holds too. Using again Corollary \ref{corDom12}, we
conclude that $\mathcal{C}\perp X$ holds for every every morphism
$\mathcal{B\rightarrow C}$ in $\mathcal{S}$. Recall that both $\mathcal{B}$
and $\mathcal{C}$ are nuclear. By Definition \ref{def11}, we have
$\mathcal{B\in}\operatorname{res}\left(  \mathcal{S},X\right)  $ and
$\gamma\in\operatorname{Spec}\left(  \mathcal{B}\right)  $, which in turn
implies that $\mathbb{C}\left(  \gamma\right)  \in\operatorname{res}%
_{\operatorname{P}}\left(  \mathcal{S},X\right)  $ (see Definition
\ref{def12}), a contradiction. Whence $\mathbb{C}\left(  \gamma\right)
\in\sigma\left(  \mathcal{S}^{\sim}\cap\mathcal{T},X\right)  ^{-}$.
\end{proof}

\subsection{The spectral mapping theorem}

Let $\mathcal{S}$ be an $A$-category and let $\mathcal{A}$ be an object of
$\mathcal{S}$. The object $\mathcal{A}$ in turn defines the $A$-category
$U_{\mathcal{A}}$ to be an open subcategory of $\mathcal{S}$, whose point
completion $U_{\mathcal{A}}^{\sim}$ is obtained by adding up all trivial
modules $\mathbb{C}\left(  \lambda\right)  $, $\lambda\in\operatorname{Spec}%
\left(  \mathcal{A}\right)  $. Actually, $U_{\mathcal{A}}$ is in turn a unital
$\mathcal{A}$-category by treating $\mathcal{A}$ as a new base algebra.
Moreover, the morphism $A\rightarrow\mathcal{A}$ can be extended up to a
morphism $F:\mathcal{S}^{\sim}\rightarrow U_{\mathcal{A}}^{\sim}$ of the
categories. Namely, $F$ is the identity map over all objects and morphisms
from $U_{\mathcal{A}}^{\sim}$, and it is trivial otherwise. For every
subcategory $\mathcal{V\subseteq S}^{\sim}$ we use the notation $\mathcal{V}%
|_{\mathcal{A}}$ instead of $F\left(  \mathcal{V}\right)  $, that is,
$\mathcal{V}|_{\mathcal{A}}$ is the projection $\mathcal{V}|_{\mathcal{A}%
}=\mathcal{V}\cap U_{\mathcal{A}}^{\sim}$ of $\mathcal{V}$, which turns out to
be an $\mathcal{A}$-category. In particular, $\mathcal{T}|_{\mathcal{A}%
}=\mathcal{T}\cap U_{\mathcal{A}}^{\sim}=\operatorname{Spec}\left(
\mathcal{A}\right)  $.

\begin{theorem}
\label{thSMT}Let $A$ be a finite type algebra, $\mathcal{S}$ a nuclear
$A$-category with its point completion $\mathcal{S}^{\sim}$ and a point basis
$\mathfrak{t}$, $\mathcal{A}$ an object of $\mathcal{S}$ with its canonical
homomorphism $\iota:A\rightarrow\mathcal{A}$, and let $X$ be a left $A$-module
such that the dominance $\left(  \mathcal{A},\iota\right)  \gg X$ holds. Then
$X$ has a left $\mathcal{A}$-module structure extending its original one
through $\iota$, and the following equalities
\[
\sigma\left(  \mathcal{S},X\right)  |_{\mathcal{A}}=\sigma\left(
U_{\mathcal{A}},X\right)  ,\quad\sigma_{\operatorname{P}}\left(
\mathcal{S},X\right)  |_{\mathcal{A}}=\sigma_{\operatorname{P}}\left(
U_{\mathcal{A}},X\right)  ,\quad\sigma\left(  \mathcal{S}^{\sim}%
\cap\mathcal{T},X\right)  |_{\mathcal{A}}=\sigma\left(  \operatorname{Spec}%
\left(  \mathcal{A}\right)  ,X\right)
\]
hold, where $U_{\mathcal{A}}$ is considered to be an $\mathcal{A}$-category.
\end{theorem}

\begin{proof}
As above in the proof of Proposition \ref{propDom12}, consider a finite free
$A$-bimodule resolution $\mathcal{R}^{\bullet}=A\widehat{\otimes}%
\mathfrak{e}^{\bullet}\widehat{\otimes}A$ of $A$ with a finite sequence
$\mathfrak{e}^{\bullet}=\left\{  E_{i}:-n\leq i\leq0\right\}  $ of (nuclear)
Fr\'{e}chet spaces. Since $X$ is a left $A$-module, it follows that
$\mathcal{R}^{\bullet}\widehat{\otimes}_{A}X$ is a finite free resolution of
$X$ in $A$-mod. The condition $\left(  \mathcal{A},\iota\right)  \gg X$ means
that the cochain complex
\[
\mathcal{A}\widehat{\otimes}_{A}\mathcal{R}^{\bullet}\widehat{\otimes}%
_{A}X\rightarrow X\rightarrow0
\]
is exact. In particular, $X$ has a left $\mathcal{A}$-module structure
extending its original one through $\iota$, and $\mathcal{A}\widehat{\otimes
}_{A}\mathcal{R}^{\bullet}\widehat{\otimes}_{A}X=\mathcal{A}\widehat{\otimes
}\mathfrak{e}^{\bullet}\widehat{\otimes}X$ turns out to be a left free
(non-split) resolution of $X$ in $\overline{\mathcal{A}\text{-mod}}$.
Moreover, every $\mathcal{B}\in U_{\mathcal{A}}$ is a right nuclear
$\mathcal{A}$-module. Using \cite[Corollary 3.1.13]{EP} (see also
\cite[Proposition 2.6]{Tay2}), we conclude that $\mathcal{B}\perp X$ over
$\mathcal{A}$ holds iff $\mathcal{B}\widehat{\otimes}_{\mathcal{A}}%
\mathcal{A}\widehat{\otimes}_{A}\mathcal{R}^{\bullet}\widehat{\otimes}_{A}X$
is exact, that is, the complex $\mathcal{B}\widehat{\otimes}_{A}%
\mathcal{R}^{\bullet}\widehat{\otimes}_{A}X$ is exact. But the exactness of
$\mathcal{B}\widehat{\otimes}_{A}\mathcal{R}^{\bullet}\widehat{\otimes}_{A}X$
means in turn that $\mathcal{B}\perp X$ over the original algebra $A$. Thus
$\mathcal{B}\perp X$ holds (or not) over $\mathcal{A}$ and $A$ simultaneously.
In particular,%
\[
\operatorname{res}\left(  \mathcal{S},X\right)  |_{\mathcal{A}}%
=\operatorname{res}\left(  \mathcal{S},X\right)  \cap U_{\mathcal{A}}^{\sim
}=\operatorname{res}\left(  \mathcal{S},X\right)  \cap U_{\mathcal{A}%
}=\operatorname{res}\left(  U_{\mathcal{A}},X\right)
\]
or $\sigma\left(  \mathcal{S},X\right)  |_{\mathcal{A}}=\sigma\left(
U_{\mathcal{A}},X\right)  $. In a similar way, we have
\[
\sigma\left(  \mathcal{S}^{\sim}\cap\mathcal{T},X\right)  |_{\mathcal{A}%
}=\sigma\left(  U_{\mathcal{A}}^{\sim}\cap\mathcal{T},X\right)  =\sigma\left(
\operatorname{Spec}\left(  \mathcal{A}\right)  ,X\right)  ,
\]
which is the Taylor spectrum of the left $\mathcal{A}$-module $X$.

It remains to prove the same equality for the Putinar spectrum. Take
$\mathbb{C}\left(  \lambda\right)  \in\sigma_{\operatorname{P}}\left(
\mathcal{S},X\right)  |_{\mathcal{A}}$ with $\lambda\in\operatorname{Spec}%
\left(  \mathcal{A}\right)  $. If $\mathbb{C}\left(  \lambda\right)
\in\operatorname{res}_{\operatorname{P}}\left(  U_{\mathcal{A}},X\right)  $,
then $\lambda\in\operatorname{Spec}\left(  \mathcal{B}\right)  $ for a certain
object $\mathcal{B}\in\operatorname{res}\left(  U_{\mathcal{A}},X\right)  $.
If $\mathcal{B\rightarrow C}$ is a morphism in $\mathcal{S}$, then it is a
morphism of the $\mathcal{A}$-category $U_{\mathcal{A}}$ too, therefore
$\mathcal{C}\perp X$ holds over $\mathcal{A}$. As we have just seen above, in
this case, the same transversality holds over $A$ too. It follows that
$\mathcal{B}\in\operatorname{res}\left(  \mathcal{S},X\right)  $, which means
that $\mathbb{C}\left(  \lambda\right)  \in\operatorname{res}%
_{\operatorname{P}}\left(  \mathcal{S},X\right)  $, a contradiction. Hence
$\sigma_{\operatorname{P}}\left(  \mathcal{S},X\right)  |_{\mathcal{A}%
}\subseteq\sigma_{\operatorname{P}}\left(  U_{\mathcal{A}},X\right)  $.

Conversely, take $\mathbb{C}\left(  \lambda\right)  \in\sigma
_{\operatorname{P}}\left(  U_{\mathcal{A}},X\right)  $. If $\mathbb{C}\left(
\lambda\right)  \notin\sigma_{\operatorname{P}}\left(  \mathcal{S},X\right)
|_{\mathcal{A}}$ then $\mathbb{C}\left(  \lambda\right)  \in\operatorname{res}%
_{\operatorname{P}}\left(  \mathcal{S},X\right)  $, which means that
$\lambda\in\operatorname{Spec}\left(  \mathcal{B}\right)  $ for some
$\mathcal{B\in}\operatorname{res}\left(  \mathcal{S},X\right)  $. Then
$\mathbb{C}\left(  \lambda\right)  $ belongs to $U_{\mathcal{A}}\cap
U_{\mathcal{B}}$ in the point completion $\mathcal{S}^{\sim}$. But there is a
point basis $\mathfrak{t}$ for $\mathcal{S}$, therefore $\mathbb{C}\left(
\lambda\right)  \in U_{\mathcal{C}}$ and $U_{\mathcal{C}}\subseteq
U_{\mathcal{A}}\cap U_{\mathcal{B}}$ for some object $\mathcal{C}$ in
$\mathcal{S}$. If $\mathcal{C\rightarrow D}$ is a morphism in $\mathcal{S}$,
then we come up with the compose morphism $\mathcal{B\rightarrow C\rightarrow
D}$, and $\mathcal{B}\perp X$ implies that $\mathcal{D}\perp X$ over $A$. But
$\mathcal{C}$ is an object of the $\mathcal{A}$-category $U_{\mathcal{A}}$ and
$\mathcal{D}\perp X$ holds over $\mathcal{A}$ too (as we have seen above) for
every morphism $\mathcal{C\rightarrow D}$. It follows that $\mathcal{C\in
}\operatorname{res}\left(  U_{\mathcal{A}},X\right)  $. But $\lambda
\in\operatorname{Spec}\left(  \mathcal{C}\right)  $, therefore $\mathbb{C}%
\left(  \lambda\right)  \in\operatorname{res}_{\operatorname{P}}\left(
U_{\mathcal{A}},X\right)  $, a contradiction.
\end{proof}

\begin{remark}
In the proof of Theorem \ref{thSMT},it suffices to have just a finite free
resolution $A\widehat{\otimes}\mathfrak{e}^{\bullet}\widehat{\otimes}A$ of
$A$, whose $\mathfrak{e}^{\bullet}$ are just Fr\'{e}chet spaces not
necessarily nuclear. Moreover, we used a point base $\mathfrak{t}$ for
$\mathcal{S}^{\sim}$ only in the proof of the equality $\sigma
_{\operatorname{P}}\left(  \mathcal{S},X\right)  |_{\mathcal{A}}%
=\sigma_{\operatorname{P}}\left(  U_{\mathcal{A}},X\right)  $.
\end{remark}

\section{Analytic geometries of Fr\'{e}chet algebras\label{sectionAFA}}

In this section we introduce a (noncommutative) complex analytic geometry of a
given Fr\'{e}chet algebra $A$ in terms of the unital complete-lattice
categories over $A$.

\subsection{The complete lattice $A$-category}

As above let $\mathcal{S}$ be an $A$-category. Assume that the set
$\operatorname{Hom}\left(  \mathcal{A},\mathcal{B}\right)  $ in the category
$\mathcal{S}$ consists of at most one element and the isomorphisms in
$\mathcal{S}$ are only identity maps, that is, $\mathcal{S}$ is a poset
category. In particular, $\operatorname{Hom}\left(  \mathcal{B},\mathcal{B}%
\right)  =\left\{  1_{\mathcal{B}}\right\}  $ and $\operatorname{Hom}\left(
A,\mathcal{B}\right)  $ is the unique arrow $A\rightarrow\mathcal{B}$ defining
the $A$-algebra structure on $\mathcal{B}$ for every object $\mathcal{B}$,
whenever $\mathcal{S}$ is unital. We put $\mathcal{A\leq B}$ if
$\operatorname{Hom}\left(  \mathcal{A},\mathcal{B}\right)  \neq\varnothing$,
which defines a partial order structure on the objects of $\mathcal{S}$ (we
can identity $\mathcal{S}$ with its objects). In this case, $A$ is the unique
least element if it is included into $\mathcal{S}$, and all trivial modules in
$\mathcal{S}$ are the maximal elements. Moreover, $U_{\mathcal{A}}=\left[
\mathcal{A},\infty\right]  $ is the right interval $\left\{  \mathcal{B}%
:\mathcal{A\leq B}\right\}  $, and the original topology of $\mathcal{S}$ is
reduced to the right order topology of the poset $\mathcal{S}$. A morphism
$F:\mathcal{S\rightarrow G}$ of the poset categories corresponds to a map of
the posets preserving the orders in $\mathcal{S}$ and $\mathcal{G}$, that is,
$\mathcal{A\leq B}$ in $\mathcal{S}$ implies that $F\left(  \mathcal{A}%
\right)  \leq F\left(  \mathcal{B}\right)  $ in $\mathcal{G}$.

\begin{lemma}
\label{lemCes1}A poset $A$-category $\mathcal{S}$ is a complete lattice if and
only if there exists infimum $\wedge U$ for every open subset $U\subseteq
\mathcal{S}$ bounded below. In this case, $\mathcal{S}$ is an irreducible
topological space.
\end{lemma}

\begin{proof}
Suppose $\wedge U$ does exist for open subset $U\subseteq\mathcal{S}$ bounded
below. Take a subset $M\subseteq\mathcal{S}$ bounded below, and define
$U=\cup\left\{  U_{\mathcal{A}}:\mathcal{A\in}M\right\}  $ to be an open set
containing $M$. If $\mathcal{B}$ is a lower bound of $M$ and $\mathcal{C\in}%
U$, then $\mathcal{B\leq A\leq C}$ for a certain $\mathcal{A\in}M$. It follows
that $U$ is bounded below by $\mathcal{B}$ itself, and there exists $\wedge U$
by assumption. Moreover, $\mathcal{B\leq}\wedge U\leq M$ for every lower bound
$\mathcal{B}$ of $M$. Hence $\wedge U=\wedge M$. Thus every subset bounded
below has the greatest lower bound, which in turn implies that every subset
bounded above has the least upper bound, that is, $\mathcal{S}$ is a complete lattice.

Finally, $U_{\mathcal{A}}\cap U_{\mathcal{B}}=\left[  \mathcal{A}%
,\infty\right]  \cap\left[  \mathcal{B},\infty\right]  =\left[  \mathcal{A}%
\vee\mathcal{B},\infty\right]  =U_{\mathcal{A}\vee\mathcal{B}}$ for all
$\mathcal{A}$ and $\mathcal{B}$ whenever $\mathcal{S}$ is a lattice. It
follows that $\mathcal{S}$ is an irreducible topological space.
\end{proof}

Now assume that $\mathcal{S}$ is a complete lattice $A$-category containing
the algebra $A$ itself, that is, $\mathcal{S}$ is\textit{ a unital complete
lattice }$A$\textit{-category}. In this case, every subset is bounded below by
$A$ being the least element of $\mathcal{S}$. Moreover $\mathcal{S}$ turns out
to be a quasicompact, irreducible topological space (see Lemma \ref{lemCes1}),
and $\mathcal{T=}\operatorname{Spec}\left(  A\right)  \subseteq\mathcal{S}%
^{\sim}$. The (poset) category of all open subsets of the topological space
$\mathcal{S}$ is denoted by $\mathcal{S}_{\tau}$. The functor
\[
\mathcal{S}_{\tau}^{\operatorname{op}}\rightarrow\mathfrak{Fa},\quad
U\mapsto\mathcal{S}^{\tau}\left(  U\right)  ,\quad\mathcal{S}^{\tau}\left(
U\right)  =\wedge U,\quad U\in\mathcal{S}_{\tau}%
\]
defines a Fr\'{e}chet algebra presheaf on $\mathcal{S}$, whose global sections
are the elements of the original algebra $A$. Thus $\mathcal{S}^{\tau}$ is a
Fr\'{e}chet algebra $A$-presheaf that corresponds to a unital complete lattice
$A$-category $\mathcal{S}$. The stalks of $\mathcal{S}^{\tau}$ consists of the
algebras $\mathcal{A}$ from $\mathcal{S}$, and the sections of the sheaf
$\left(  \mathcal{S}^{\tau}\right)  ^{+}$ associated to the presheaf
$\mathcal{S}^{\tau}$ over an open subset $U$ consists of the compatible
families $\left\{  e\left(  \mathcal{A}\right)  :\mathcal{A}\in U\right\}  $
from $\prod\left\{  \mathcal{A}:\mathcal{A}\in U\right\}  $ in the sense of
that $e\left(  \mathcal{A}\right)  \mapsto e\left(  \mathcal{B}\right)  $
whenever $\mathcal{A\leq B}$ in $U$. In particular, $\Gamma\left(
\mathcal{S},\left(  \mathcal{S}^{\tau}\right)  ^{+}\right)  =A$ (recall that
$A$ is an object of $\mathcal{S}$), and $\mathcal{S}^{\tau}\left(  U\right)
=\left(  \mathcal{S}^{\tau}\right)  ^{+}\left(  U\right)  =\mathbb{C}\left(
\lambda\right)  $ whenever $U=\left\{  \mathbb{C}\left(  \lambda\right)
\right\}  $ is a trivial open set in $\mathcal{S}$. By passing to the point
completion of $\mathcal{S}$, the related sheaf $\left(  \mathcal{S}^{\tau
}\right)  ^{+}$ is completed by the trivial stalks.

Now let $\mathcal{G}$ be another unital complete lattice $B$-category, and let
$F:\mathcal{S\rightarrow G}$ be a morphism of the unital complete lattice
categories. Thus $F$ is a functor with a compatible morphisms $\mathcal{A}%
\rightarrow F\left(  \mathcal{A}\right)  $ and $f:A\rightarrow B$ of the
Fr\'{e}chet algebras. As we have mentioned above in Subsection \ref{subsecAT},
$F$ is continuous with respect to their Aleksandrov topologies, that is, it
defines the functor $\mathcal{G}_{\tau}\rightarrow\mathcal{S}_{\tau}$,
$V\mapsto F^{-1}\left(  V\right)  $ of the related poset categories of open
subsets. If the functor $\mathcal{G}_{\tau}\rightarrow\mathcal{S}_{\tau}$
implements a category equivalence, we say that $F$ is \textit{a strong
morphism. }In this case, we obtain that $F$ is open ($F\left(  U\right)
\in\mathcal{G}_{\tau}$ for every $U\in\mathcal{S}_{\tau}$) and it defines the
category equivalence $F:\mathcal{S}_{\tau}\rightarrow\mathcal{G}_{\tau}$ of
the related poset categories. In particular, we can identify the poset
categories $\mathcal{S}_{\tau}$ and $\mathcal{G}_{\tau}$ by claiming that they
have a common poset category of open subsets denoted by $\Omega$. The identity
functor $\mathcal{S\rightarrow S}$ is a strong morphism, and the compose of
two strong morphisms is a strong one. Thus, we come up with a new category of
the unital complete lattice categories with their strong morphisms and the
common poset category $\Omega$ of open subsets, which is denoted by
$\mathfrak{C}_{\Omega}$. We use the notation $\mathfrak{Fa}_{\Omega}$ for the
category of all Fr\'{e}chet algebra presheaves on $\Omega$, which is the
functor category $\Omega^{\operatorname{op}}\rightarrow\mathfrak{Fa}$.

\begin{lemma}
\label{lemCes2}If $F:\mathcal{S}\rightarrow\mathcal{G}$ is a morphism of the
category $\mathfrak{C}_{\Omega}$, then it defines a morphism of the presheaves
$\mathcal{S}^{\tau}\rightarrow\mathcal{G}^{\tau}$ on $\Omega$. Thus we have
the functor $\mathfrak{C}_{\Omega}\rightarrow\mathfrak{Fa}_{\Omega}$.
\end{lemma}

\begin{proof}
First take an object $\mathcal{A}$ from $\mathcal{S}$. Then $U_{\mathcal{A}%
}\in\Omega$ and $F\left(  U_{\mathcal{A}}\right)  $ is open in $\mathcal{G}$
with $F\left(  U_{\mathcal{A}}\right)  \subseteq U_{F\left(  \mathcal{A}%
\right)  }$. But $F\left(  \mathcal{A}\right)  \in F\left(  U_{\mathcal{A}%
}\right)  $ and $U_{F\left(  \mathcal{A}\right)  }$ is the smallest open
subset containing $F\left(  \mathcal{A}\right)  $. Therefore $F\left(
U_{\mathcal{A}}\right)  =U_{F\left(  \mathcal{A}\right)  }$. In particular,
for every morphism $F\left(  \mathcal{A}\right)  \rightarrow\mathcal{C}$ in
$\mathcal{G}$, we deduce that $\mathcal{C\in}U_{F\left(  \mathcal{A}\right)
}$ or $\mathcal{C\in}F\left(  U_{\mathcal{A}}\right)  $, which means that
$\mathcal{C=}F\left(  \mathcal{B}\right)  $ with a morphism
$\mathcal{A\rightarrow B}$ in $\mathcal{S}$.

Now take an open subset $U\subseteq\mathcal{S}$ with $\mathcal{A=}\wedge U$.
Since $F\left(  \mathcal{A}\right)  \leq F\left(  U\right)  $ in $\mathcal{G}%
$, it follows that $F\left(  \mathcal{A}\right)  \leq\wedge F\left(  U\right)
$, that is, there is a morphism $F\left(  \mathcal{A}\right)  \rightarrow
\wedge F\left(  U\right)  $ in $\mathcal{G}$. Using the fact that we have
pointed out above, we obtain that $\wedge F\left(  U\right)  =F\left(
\mathcal{B}\right)  $ for some morphism $\mathcal{A\rightarrow B}$ in
$\mathcal{S}$. Prove that $\mathcal{A=B}$. If that is not the case, then
$\mathcal{A\leq B}$, $\mathcal{A\neq B}$ and $\mathcal{B}$ is not a lower
bound of $U$. It follows that there exists a morphism $\mathcal{C\rightarrow
B}$ in $\mathcal{S}$ with $\mathcal{C\in}U$, which in turn implies that
$F\left(  \mathcal{C}\right)  \leq F\left(  \mathcal{B}\right)  $ and
$F\left(  \mathcal{C}\right)  \in F\left(  U\right)  $. But $F\left(
U\right)  $ is open, therefore $F\left(  \mathcal{B}\right)  \in F\left(
U\right)  $ and $F\left(  \mathcal{B}\right)  =F\left(  \mathcal{C}\right)  $.
Thus $F\left(  U\right)  =U_{F\left(  \mathcal{B}\right)  }=F\left(
U_{\mathcal{B}}\right)  $. But $U,U_{\mathcal{B}}\in\mathcal{S}_{\tau}$ and
$F:\mathcal{S}_{\tau}\rightarrow\mathcal{G}_{\tau}$ implements a category
equivalence of the posets. It follows that $U=U_{\mathcal{B}}$ and
$\mathcal{A=}\wedge U=\wedge U_{\mathcal{B}}=\mathcal{B}$, a contradiction.
Hence
\[
\mathcal{G}^{\tau}\left(  F\left(  U\right)  \right)  =\wedge F\left(
U\right)  =F\left(  \wedge U\right)  =F\left(  \mathcal{S}^{\tau}\left(
U\right)  \right)
\]
holds for every open subset $U\in\Omega$. But $F$ is a morphism, which imposes
that there are compatible morphisms $\mathcal{S}^{\tau}\left(  U\right)
\rightarrow F\left(  \mathcal{S}^{\tau}\left(  U\right)  \right)  $. Thus $F$
defines the presheaf morphism $\mathcal{S}^{\tau}\rightarrow\mathcal{G}^{\tau
}$ on $\Omega$.
\end{proof}

\begin{proposition}
\label{corCes1}The functor $\mathfrak{C}_{\Omega}\rightarrow\mathfrak{Fa}%
_{\Omega}$ is a category equivalence.
\end{proposition}

\begin{proof}
By Lemma \ref{lemCes2}, we have the well defined functor $\mathfrak{C}%
_{\Omega}\rightarrow\mathfrak{Fa}_{\Omega}$, $\mathcal{S\mapsto S}^{\tau}$.
Let $\mathcal{P}:\Omega^{\operatorname{op}}\rightarrow\mathfrak{Fa}$ be a
Fr\'{e}chet algebra presheaf on $\Omega$ with the Fr\'{e}chet algebra $A$ of
the global sections. Put $\widehat{\mathcal{P}}$ to be a poset category
defined by the presheaf $\mathcal{P}$, that is, the objects of
$\widehat{\mathcal{P}}$ are the algebras $\mathcal{P}\left(  V\right)  $ over
open subsets $V\in\Omega$ with their unique morphisms $\mathcal{P}\left(
V\right)  \rightarrow\mathcal{P}\left(  W\right)  $ given by the restriction
morphisms of the presheaf $\mathcal{P}$. One can easily verify that
$\mathcal{P}\left(  V\right)  \vee\mathcal{P}\left(  W\right)  =\mathcal{P}%
\left(  V\cap W\right)  $ and $\mathcal{P}\left(  V\right)  \wedge
\mathcal{P}\left(  W\right)  =\mathcal{P}\left(  V\cup W\right)  $ \ Moreover,
every subset $\left\{  \mathcal{P}\left(  V_{i}\right)  \right\}  $ of
$\widehat{\mathcal{P}}$ is bounded below by $A$ and $\wedge\left\{
\mathcal{P}\left(  V_{i}\right)  \right\}  =\mathcal{P}\left(  \cup_{i}%
V_{i}\right)  $. Hence $\widehat{\mathcal{P}}$ is a unital complete lattice
$A$-category. An open subset $U^{\prime}\subseteq\widehat{\mathcal{P}}$
corresponds to some $U\in\Omega$, that is, it consists of those $\mathcal{P}%
\left(  V\right)  $ with $V\subseteq U$. Thus $\widehat{\mathcal{P}}_{\tau}$
consists of $U^{\prime}=\left\{  \mathcal{P}\left(  V\right)  :V\subseteq
U\right\}  $ with $U\in\Omega$. In this case, we have $\wedge U^{\prime
}=\mathcal{P}\left(  U\right)  $ (see Lemma \ref{lemCes1}), and the
correspondence $\mathcal{P}\longrightarrow\widehat{\mathcal{P}}$ is
functorial. Namely, assume that $\mathcal{P}_{1}\rightarrow\mathcal{P}_{2}$ is
a presheaf morphism over $\Omega$, which is a functor transformation
\[
\left(  \mathcal{P}_{1}:\Omega^{\operatorname{op}}\rightarrow\mathfrak{Fa}%
\right)  \rightarrow\left(  \mathcal{P}_{2}:\Omega^{\operatorname{op}%
}\rightarrow\mathfrak{Fa}\right)  .
\]
If $\widehat{\mathcal{P}_{i}}$, $i=1,2$ are the related unital complete
lattice categories, then%
\[
F:\widehat{\mathcal{P}_{1}}\rightarrow\widehat{\mathcal{P}_{2}},\quad
\mathcal{P}_{1}\left(  V\right)  \rightarrow F\left(  \mathcal{P}_{1}\left(
V\right)  \right)  =\mathcal{P}_{2}\left(  V\right)  ,\quad V\in\Omega
\]
is a morphism of these categories. Actually, $F$ is a strong morphism. Indeed,
if $U\in\Omega$ with its open subset representations $U_{i}=\left\{
\mathcal{P}_{i}\left(  V\right)  :V\subseteq U\right\}  $ in
$\widehat{\mathcal{P}_{i}}$, then
\[
F\left(  U_{1}\right)  =F\left\{  \mathcal{P}_{1}\left(  V\right)  :V\subseteq
U\right\}  =\left\{  \mathcal{P}_{2}\left(  V\right)  :V\subseteq U\right\}
=U_{2}.
\]
Thus $F:\left(  \widehat{\mathcal{P}_{1}}\right)  _{\tau}\rightarrow\left(
\widehat{\mathcal{P}_{2}}\right)  _{\tau}$ is a category equivalence of the
related poset categories of open subsets, which means that $F$ is a strong
morphism. Hence we come up with a functor $\mathfrak{Fa}_{\Omega}%
\rightarrow\mathfrak{C}_{\Omega}$.

If $\mathcal{S}$ is a unital complete lattice $A$-category with its
corresponding presheaf $\mathcal{S}^{\tau}$ on $\Omega$, then $\mathcal{S}%
^{\tau}\left(  U_{\mathcal{A}}\right)  =\wedge U_{\mathcal{A}}=\mathcal{A}$
for every object $\mathcal{A}$ of $\mathcal{S}$. Moreover, $\mathcal{A\leq B}$
in $\mathcal{S}$ iff $U_{\mathcal{A}}\supseteq U_{\mathcal{B}}$, which means
that there is a morphism $\mathcal{S}^{\tau}\left(  U_{\mathcal{A}}\right)
\rightarrow\mathcal{S}^{\tau}\left(  U_{\mathcal{B}}\right)  $. Thus the poset
category $\widehat{\mathcal{S}^{\tau}}$ of the presheaf $\mathcal{S}^{\tau}$
is reduced to $\mathcal{S}$ up to an order isomorphism. Conversely, if
$\mathcal{P}$ is a presheaf over $\Omega$ with its corresponding poset
category $\widehat{\mathcal{P}}$, then $\widehat{\mathcal{P}}$ is a unital
complete lattice $A$-category. Let $\left(  \widehat{\mathcal{P}}\right)
^{\tau}$ be the related presheaf on the irreducible topological space
$\widehat{\mathcal{P}}$ (see Lemma \ref{lemCes1}). Then the category of the
open subsets of $\widehat{\mathcal{P}}$ is reduced to $\Omega$, and $\left(
\widehat{\mathcal{P}}\right)  ^{\tau}\left(  U^{\prime}\right)  =\wedge
U^{\prime}=\mathcal{P}\left(  U\right)  $ for an open subset $U^{\prime
}\subseteq\widehat{\mathcal{P}}$ corresponding to $U\in\Omega$. Thus the
result follows.
\end{proof}

\subsection{The complex analytic geometries\label{subsecCAGM}}

Now let $A$ be a given Fr\'{e}chet algebra. There are unital complete-lattice
categories over $A$, for example, one can choose the trivial small category
containing just $A$. Among these $A$-categories we introduce that would be a
reasonable model for a (noncommutative) complex analytic geometry of $A$.

\begin{definition}
\label{defAG}A complex analytic geometry of $A$ is called a unital
complete-lattice $A$-category $\mathcal{S}$ such that for every open subset
$U\subseteq\mathcal{S}$ the inclusion $\operatorname{Spec}\left(  \wedge
U\right)  \subseteq U^{\sim}$ holds, where $U^{\sim}$ is the point completion
of $U$.
\end{definition}

Thus $\operatorname{Spec}\left(  \wedge U\right)  =U^{\sim}\cap\mathcal{T}$
for every $U\in\mathcal{S}_{\tau}$ whenever $\mathcal{S}$ is a complex
analytic geometry of $A$, where $\mathcal{T=}\operatorname{Spec}\left(
A\right)  $. Indeed, if $\mathbb{C}\left(  \lambda\right)  \in U^{\sim}%
\cap\mathcal{T}$ then there exists a morphism $\mathcal{B}\rightarrow
\mathbb{C}\left(  \lambda\right)  $ for some object $\mathcal{B}$ of $U$. The
compose morphism $\wedge U\rightarrow\mathcal{B}\rightarrow\mathbb{C}\left(
\lambda\right)  $ implies that $\mathbb{C}\left(  \lambda\right)
\in\operatorname{Spec}\left(  \wedge U\right)  $, that is, $U^{\sim}%
\cap\mathcal{T\subseteq}\operatorname{Spec}\left(  \wedge U\right)  $ holds.

It means that the poset category of all open subsets in $\mathcal{T}$ induced
from $\mathcal{S}_{\tau}$ consists of $\left\{  \operatorname{Spec}\left(
\mathcal{A}\right)  :\mathcal{A\in S}\right\}  $. More precisely, the functor
\[
\mathcal{S}_{\tau}\rightarrow\left(  \mathcal{S}^{\sim}\cap\mathcal{T}\right)
_{\tau},\quad U\mapsto U^{\sim}\cap\mathcal{T}%
\]
is an epimorphism of the poset categories.

\begin{lemma}
\label{lemAG1}Let $A$ be a Fr\'{e}chet algebra. A complex analytic geometry of
$A$ is equivalent to the presence of a Fr\'{e}chet algebra presheaf
$\mathcal{P}$ on a topological space $\omega$ containing $\operatorname{Spec}%
\left(  A\right)  $ such that $A=\Gamma\left(  \omega,\mathcal{P}\right)  $
and $U\cap\operatorname{Spec}\left(  A\right)  =\operatorname{Spec}\left(
\mathcal{P}\left(  U\right)  \right)  $ for every open subset $U\subseteq
\omega$.
\end{lemma}

\begin{proof}
First assume that a unital complete-lattice $A$-category $\mathcal{S}$
provides a complex analytic geometry of $A$. By Proposition \ref{corCes1}, the
Fr\'{e}chet algebra presheaf $\mathcal{S}^{\tau}$ that corresponds to
$\mathcal{S}$ can be extended to a Fr\'{e}chet algebra presheaf on the
topological space $\omega=\mathcal{S}^{\sim}$, which contains $\mathcal{T}$ or
$\operatorname{Spec}\left(  A\right)  $. Namely, $\mathcal{P}\left(  U^{\sim
}\right)  =\wedge U=\mathcal{S}^{\tau}\left(  U\right)  $ over the point
completion $U^{\sim}$ of $U\in\mathcal{S}_{\tau}$. One can easily see that
$A=\Gamma\left(  \omega,\mathcal{P}\right)  $. Moreover, $\operatorname{Spec}%
\left(  \mathcal{P}\left(  U^{\sim}\right)  \right)  =\operatorname{Spec}%
\left(  \wedge U\right)  =U^{\sim}\cap\mathcal{T}=U^{\sim}\cap
\operatorname{Spec}\left(  A\right)  $ by virtue of Definition \ref{defAG}.

Conversely, suppose that there exists a Fr\'{e}chet algebra presheaf
$\mathcal{P}$ on a topological space $\omega$ containing $\operatorname{Spec}%
\left(  A\right)  $ such that $A=\Gamma\left(  \omega,\mathcal{P}\right)  $
and $V\cap\operatorname{Spec}\left(  A\right)  =\operatorname{Spec}\left(
\mathcal{P}\left(  V\right)  \right)  $ (see below Remark \ref{remAG1}) for
every open subset $V\subseteq\omega$. By Proposition \ref{corCes1},
$\mathcal{P}$ defines a unital complete-lattice $A$-category
$\widehat{\mathcal{P}}=\left\{  \mathcal{P}\left(  V\right)  :V\subseteq
\omega\right\}  $ with their restriction morphisms. An open subset
$U_{1}\subseteq\widehat{\mathcal{P}}$ corresponds to some open $U\subseteq
\omega$, namely, $U_{1}=\left\{  \mathcal{P}\left(  V\right)  :V\subseteq
U\right\}  $ and $\wedge U_{1}=\mathcal{P}\left(  U\right)  $. If $\lambda\in
U\cap\operatorname{Spec}\left(  A\right)  $ then $\lambda\in V$ for a small
neighborhood $V\subseteq U$, that is, $\lambda\in V\cap\operatorname{Spec}%
\left(  A\right)  $. By assumption, we have $V\cap\operatorname{Spec}\left(
A\right)  \subseteq\operatorname{Spec}\left(  \mathcal{P}\left(  V\right)
\right)  $, that is, there is a morphism $\mathcal{P}\left(  V\right)
\rightarrow\mathbb{C}\left(  \lambda\right)  $ for $V\subseteq U$. It follows
that $\mathbb{C}\left(  \lambda\right)  \in U_{1}^{\sim}\cap\mathcal{T}$. Thus
$U\cap\operatorname{Spec}\left(  A\right)  \subseteq U_{1}^{\sim}%
\cap\mathcal{T}$. Finally,
\[
\operatorname{Spec}\left(  \wedge U_{1}\right)  =\operatorname{Spec}\left(
\mathcal{P}\left(  U\right)  \right)  \subseteq U\cap\operatorname{Spec}%
\left(  A\right)  \subseteq U_{1}^{\sim},
\]
which means (see Definition \ref{defAG}) that $\widehat{\mathcal{P}}$ stands
for a complex analytic geometry of $A$.
\end{proof}

\begin{remark}
\label{remAG1}The equality $U\cap\operatorname{Spec}\left(  A\right)
=\operatorname{Spec}\left(  \mathcal{P}\left(  U\right)  \right)  $ in Lemma
\ref{lemAG1} is treated as the natural continuous extension of a character
$\lambda\in U\cap\operatorname{Spec}\left(  A\right)  $ to $\mathcal{P}\left(
U\right)  $ through the restriction morphism $A\rightarrow\mathcal{P}\left(
U\right)  $.
\end{remark}

Based on Lemma \ref{lemAG1}, we also say that $\left(  \omega,\mathcal{P}%
\right)  $ is a complex analytic geometry of $A$. As a possible candidate for
the underlying topological space $\omega$ one can choose the set of all
irreducible Banach space representations of $A$ containing the set of all
$1$-dimensional representations $\operatorname{Spec}\left(  A\right)  $
equipped with a special (used to be a non-Hausdorff) topology. In the case of
$\omega=\operatorname{Spec}\left(  A\right)  $, we have a Fr\'{e}chet algebra
presheaf $\mathcal{P}$ on $\operatorname{Spec}\left(  A\right)  $ such that
$U=\operatorname{Spec}\left(  \mathcal{P}\left(  U\right)  \right)  $ for
every open subset $U\subseteq\operatorname{Spec}\left(  A\right)  $. In this
case, we say that $\mathcal{P}$ (or corresponding $A$-category
$\widehat{\mathcal{P}}$) is \textit{a standard analytic geometry of }$A$. For
example, if $A$ is an Arens-Michael-Fr\'{e}chet algebra, which is commutative
modulo its Jacobson radical, then $\omega=\operatorname{Spec}\left(  A\right)
$ and a possible pair $\left(  \operatorname{Spec}\left(  A\right)
,\mathcal{P}\right)  $ with $A=\Gamma\left(  \operatorname{Spec}\left(
A\right)  ,\mathcal{P}\right)  $ and $U=\operatorname{Spec}\left(
\mathcal{P}\left(  U\right)  \right)  $, $U\subseteq\operatorname{Spec}\left(
A\right)  $ stands for a standard analytic geometry of $A$. More concrete
examples of algebras $A$ with their standard geometric models are considered below.

If $\widehat{\mathcal{P}}$ is an $A$-category corresponding to a presheaf
$\mathcal{P}$ from Lemma \ref{lemAG1}, then we use the notations
$\sigma\left(  \mathcal{P},X\right)  $, $\sigma_{\operatorname{P}}\left(
\mathcal{P},X\right)  $ and $\sigma\left(  A,X\right)  $ instead of
$\sigma\left(  \widehat{\mathcal{P}},X\right)  $, $\sigma_{\operatorname{P}%
}\left(  \widehat{\mathcal{P}},X\right)  $ and $\sigma\left(  \mathcal{T}%
,X\right)  $ (or $\sigma\left(  \widehat{\mathcal{P}}^{\sim}\cap
\mathcal{T},X\right)  $), respectively.

\begin{proposition}
\label{remPut1}Let $\mathcal{P}$ be a Fr\'{e}chet algebra presheaf on a
topological space $\omega$ with $A=\mathcal{P}\left(  \omega\right)  $, and
let $X$ be a left $A$-module. Then $\operatorname{res}\left(
\widehat{\mathcal{P}},X\right)  $ being an open subset of $\omega$ consists of
those $\gamma\in\omega$ such that there exists an open neighborhood $V$ of
$\gamma$ so that $\mathcal{P}\left(  W\right)  \perp X$ for every open
$W\subseteq V$. If $\left(  \omega,\mathcal{P}\right)  $ is a complex analytic
geometry of $A$, then
\[
\sigma_{\operatorname{P}}\left(  \mathcal{P},X\right)  =\sigma\left(
\mathcal{P},X\right)  \cap\operatorname{Spec}\left(  A\right)
\]
holds. In particular, in the case of a standard analytic geometry $\left(
\operatorname{Spec}\left(  A\right)  ,\mathcal{P}\right)  $ of $A$, the
$\widehat{\mathcal{P}}$-spectrum $\sigma\left(  \mathcal{P},X\right)  $ is
reduced to the Putinar spectrum $\sigma_{\operatorname{P}}\left(
\mathcal{P},X\right)  $.
\end{proposition}

\begin{proof}
Notice that $\operatorname{res}\left(  \widehat{\mathcal{P}},X\right)  $ is an
open subset of $\widehat{\mathcal{P}}$. Based on Proposition \ref{corCes1}, it
is identified with an open subset of the topological space $\omega$. Namely,
$\mathcal{P}\left(  V\right)  \in\operatorname{res}\left(
\widehat{\mathcal{P}},X\right)  $ iff $\mathcal{P}\left(  W\right)  \perp X$
for every open $W\subseteq V$. Thus $\operatorname{res}\left(
\widehat{\mathcal{P}},X\right)  $ consists of those $\gamma\in\omega$ such
that there exists an open neighborhood $V$ of $\gamma$ so that $\mathcal{P}%
\left(  W\right)  \perp X$ for every open $W\subseteq V$.

If $\widehat{\mathcal{P}}$ stands for an analytic geometry $\left(
\omega,\mathcal{P}\right)  $ of $A$ (see Lemma \ref{lemAG1}), then
$\omega\supseteq\operatorname{Spec}\left(  A\right)  =\mathcal{T}$, and
$\operatorname{res}\left(  \widehat{\mathcal{P}},X\right)  $ is identified
with an open subset of $\omega$. Moreover, $\mathbb{C}\left(  \lambda\right)
\in\operatorname{res}\left(  \widehat{\mathcal{P}},X\right)  ^{\sim}$ iff
there exists a morphism $\mathcal{P}\left(  V\right)  \rightarrow
\mathbb{C}\left(  \lambda\right)  $ for some $\mathcal{P}\left(  V\right)
\in\operatorname{res}\left(  \widehat{\mathcal{P}},X\right)  $. In this case,
$\lambda\in V$ (see Definition \ref{defAG}), and $\mathcal{P}\left(  W\right)
\perp X$ for every open $W\subseteq V$, which means that $\lambda
\in\operatorname{res}\left(  \widehat{\mathcal{P}},X\right)  \cap
\operatorname{Spec}\left(  A\right)  $. Therefore
\[
\operatorname{res}_{\operatorname{P}}\left(  \widehat{\mathcal{P}},X\right)
=\operatorname{res}\left(  \widehat{\mathcal{P}},X\right)  ^{\sim}%
\cap\mathcal{T=}\operatorname{res}\left(  \widehat{\mathcal{P}},X\right)
\cap\operatorname{Spec}\left(  A\right)
\]
by its very definition (see Subsection \ref{subsecSSA}). Thus we have (see
Definition \ref{def12})
\begin{align*}
\sigma_{\operatorname{P}}\left(  \mathcal{P},X\right)   &  =\sigma
_{\operatorname{P}}\left(  \widehat{\mathcal{P}},X\right)  =\left(
\widehat{\mathcal{P}}^{\sim}\cap\mathcal{T}\right)  \backslash
\operatorname{res}_{\operatorname{P}}\left(  \widehat{\mathcal{P}},X\right)
=\mathcal{T}\backslash\operatorname{res}_{\operatorname{P}}\left(
\widehat{\mathcal{P}},X\right) \\
&  =\operatorname{Spec}\left(  A\right)  \backslash\left(  \operatorname{res}%
\left(  \widehat{\mathcal{P}},X\right)  \cap\operatorname{Spec}\left(
A\right)  \right)  =\sigma\left(  \widehat{\mathcal{P}},X\right)
\cap\operatorname{Spec}\left(  A\right) \\
&  =\sigma\left(  \mathcal{P},X\right)  \cap\operatorname{Spec}\left(
A\right)  ,
\end{align*}
that is, $\sigma_{\operatorname{P}}\left(  \mathcal{P},X\right)
=\sigma\left(  \mathcal{P},X\right)  \cap\operatorname{Spec}\left(  A\right)
$ holds.
\end{proof}

A point basis $\mathfrak{t}$ for an analytic geometry $\left(  \omega
,\mathcal{P}\right)  $ of $A$ (or for $\widehat{\mathcal{P}}^{\sim}$)
corresponds to a (countable) topology base $\left\{  U_{i}:i\in I\right\}  $
in $\omega$ by Proposition \ref{corCes1}. Therefore, the presence of a point
basis for $\left(  \omega,\mathcal{P}\right)  $ is routine. Moreover, if
$A\rightarrow\mathcal{P}\left(  U_{i}\right)  $, $i\in I$ are localizations,
then $\mathfrak{t}$ is a localizing basis. In this case, we say that
$\mathcal{P}$ is\textit{ a localizing presheaf}.

\begin{corollary}
\label{corM1}Let $A$ be a finite type algebra with its complex analytic
geometry $\left(  \omega,\mathcal{P}\right)  $ such that $\mathcal{P}$ is a
localizing presheaf of nuclear algebras, and let $X$ be a left $A$-module.
Then $\sigma\left(  A,X\right)  \subseteq\sigma_{\operatorname{P}}\left(
\mathcal{P},X\right)  $, and $\sigma_{\operatorname{P}}\left(  \mathcal{P}%
,X\right)  =\sigma\left(  A,X\right)  ^{-}$ in $\omega$ whenever $X$ is a
$\widehat{\mathcal{P}}$-local left $A$-module.
\end{corollary}

\begin{proof}
One needs to apply Propositions \ref{propLocG} and \ref{remPut1}. Namely,
first we have
\[
\sigma\left(  A,X\right)  =\sigma\left(  \widehat{\mathcal{P}}^{\sim}%
\cap\mathcal{T},X\right)  \subseteq\sigma_{\operatorname{P}}\left(
\widehat{\mathcal{P}},X\right)  =\sigma_{\operatorname{P}}\left(
\mathcal{P},X\right)  .
\]
Further, since $\sigma\left(  \mathcal{P},X\right)  $ is a closed subset of
$\omega$, it follows that $\sigma_{\operatorname{P}}\left(  \mathcal{P}%
,X\right)  =\sigma\left(  \mathcal{P},X\right)  \cap\operatorname{Spec}\left(
A\right)  $ is a closed subset of $\operatorname{Spec}\left(  A\right)  $
equipped with the topology inherited from $\omega$. Hence $\sigma
_{\operatorname{P}}\left(  \mathcal{P},X\right)  =\sigma\left(  A,X\right)
^{-}$ in $\omega$ whenever $X$ is an $\widehat{\mathcal{P}}$-local left $A$-module.
\end{proof}

If $\left(  \omega,\mathcal{P}\right)  $ is an analytic geometry of $A$, and
$U\subseteq\omega$ an open subset, then $\left(  U,\mathcal{P}|_{U}\right)  $
defines an analytic geometry of the algebra $\mathcal{P}\left(  U\right)  $.
In this case, $U\supseteq\operatorname{Spec}\left(  \mathcal{P}\left(
U\right)  \right)  $ holds automatically by Lemma \ref{lemAG1}. Moreover, for
every open subset $V\subseteq U$ we have $V\cap\operatorname{Spec}\left(
\mathcal{P}\left(  U\right)  \right)  =V\cap\operatorname{Spec}\left(
A\right)  \cap\operatorname{Spec}\left(  \mathcal{P}\left(  U\right)  \right)
=\operatorname{Spec}\left(  \mathcal{P}\left(  V\right)  \right)  $.

\begin{corollary}
\label{corM2}Let $A$ be a finite type algebra with its complex analytic
geometry $\left(  \omega,\mathcal{P}\right)  $ such that $\mathcal{P}$ is a
presheaf of nuclear algebras, and let $X$ be a left $A$-module. If $\left(
\mathcal{P}\left(  U\right)  ,\iota_{U}\right)  \gg X$ for an open subset
$U\subseteq\omega$, then
\[
\sigma_{\operatorname{P}}\left(  \mathcal{P},X\right)  |_{\mathcal{P}\left(
U\right)  }=\sigma_{\operatorname{P}}\left(  \mathcal{P}|_{U},X\right)
\quad\text{and}\quad\sigma\left(  A,X\right)  |_{\mathcal{P}\left(  U\right)
}=\sigma\left(  \mathcal{P}\left(  U\right)  ,X\right)
\]
where $\iota_{U}:A\rightarrow\mathcal{P}\left(  U\right)  $ is the restriction
morphism of the presheaf $\mathcal{P}$.
\end{corollary}

\begin{proof}
One needs to apply Proposition \ref{remPut1} and Theorem \ref{thSMT}.
\end{proof}

\section{\v{C}ech category and functional calculus}

In this section we propose a functional calculus model within the \v{C}ech
category over a Fr\'{e}chet algebra.

\subsection{A basis for a complete lattice $A$-category\label{subsecTBC}}

A countable family $\mathfrak{b=}\left\{  \mathcal{B}\right\}  $ of the
objects of a complete lattice $A$-category $\mathcal{S}$ is said to be
\textit{a basis for }$\mathcal{S}$ if for every nontrivial open subset
$U\subseteq\mathcal{S}$ bounded below, we have $\wedge U=\wedge\mathfrak{b}%
_{U}$ with $\mathfrak{b}_{U}=U\cap\mathfrak{b}$. If $\mathcal{S}%
=\widehat{\mathcal{P}}$ for a Fr\'{e}chet algebra presheaf $\mathcal{P}%
:\Omega^{\operatorname{op}}\rightarrow\mathfrak{Fa}$, then a countable
topology basis $\mathfrak{b=}\left\{  V_{i}\right\}  $ in $\Omega$ corresponds
to a basis for $\mathcal{S}$ (see Proposition \ref{corCes1}). Namely, if
$U\subseteq\mathcal{S}$ is open then $U=\cup\left\{  V_{i}:V_{i}\subseteq
U\right\}  $ and $\mathfrak{b}_{U}=\left\{  \mathcal{P}\left(  V_{i}\right)
:V_{i}\subseteq U\right\}  $, which in turn implies that%
\[
\wedge U=\mathcal{P}\left(  U\right)  =\mathcal{P}\left(  \cup\left\{
V_{i}:V_{i}\subseteq U\right\}  \right)  =\wedge\left\{  \mathcal{P}\left(
V_{i}\right)  :V_{i}\subseteq U\right\}  =\wedge\mathfrak{b}_{U}.
\]
Thus $\mathfrak{b}$ is a basis for the category $\widehat{\mathcal{P}}$. If
$\left(  \omega,\mathcal{P}\right)  $ is an analytic geometry of $A$ (see
Lemma \ref{lemAG1}), then $\widehat{\mathcal{P}}^{\sim}\cap\mathcal{T}%
=\mathcal{T}=\operatorname{Spec}\left(  A\right)  \subseteq\omega$ and the
open subsets of $\widehat{\mathcal{P}}^{\sim}$ correspond to open subsets of
$\omega$. Hence a basis $\mathfrak{b}$ for $\widehat{\mathcal{P}}$ is just a
point basis for $\widehat{\mathcal{P}}^{\sim}$ (see Subsection \ref{subsecPB}%
), and vice-versa. It is worth to notice that a basis for a complete lattice
$A$-category $\mathcal{S}$ is not a topology basis for the right order
topology of $\mathcal{S}$.

\subsection{Augmented \v{C}ech complex of a basis}

As above let $\mathcal{S}$ be a complete lattice $A$-category with its basis
$\mathfrak{b}$, and $U\subseteq\mathcal{S}$ an open subset. For every tuple
$\mathcal{B}\in\mathfrak{b}_{U}^{p+1}$ with $\mathcal{B}=\left(
\mathcal{B}_{0},\ldots,\mathcal{B}_{p}\right)  $, we define $\vee
\mathcal{B}=\mathcal{B}_{0}\vee\cdots\vee\mathcal{B}_{p}$ to be an element of
the lattice $\mathcal{S}$ (or $U$) with their unique morphisms $\mathcal{B}%
_{j}\rightarrow\vee\mathcal{B}$, $0\leq j\leq p$. For the tuple $\left(
\mathcal{B}_{0},\ldots,\widehat{\mathcal{B}_{j}},\ldots,\mathcal{B}%
_{p}\right)  \in\mathfrak{b}_{U}^{p}$ with the removed term $\mathcal{B}_{j}$
from $\mathcal{B}$, we use the notation $\mathcal{B}\left(  j\right)  $. In
particular, $\mathcal{B}\left(  j,k\right)  $ is a tuple from $\mathfrak{b}%
_{U}^{p-1}$ with the removed two distinct terms $\mathcal{B}_{j}$ and
$\mathcal{B}_{k}$. Actually, $\mathcal{B}\left(  j,k\right)  =\mathcal{B}%
\left(  j\right)  \left(  k\right)  =\mathcal{B}\left(  k\right)  \left(
j\right)  $ for all $j$, $k$. Since $\vee\mathcal{B}\left(  j\right)  \leq
\vee\mathcal{B}$, there is a unique morphism $m_{j}:\vee\mathcal{B}\left(
j\right)  \rightarrow\vee\mathcal{B}$ for each $0\leq j\leq p$. In particular,
there are morphisms $m_{j,k}:\vee\mathcal{B}\left(  j,k\right)  \rightarrow
\vee\mathcal{B}\left(  j\right)  $ and $m_{k,j}:\vee\mathcal{B}\left(
j,k\right)  \rightarrow\vee\mathcal{B}\left(  k\right)  $. Since
$\vee\mathcal{B}\left(  j,k\right)  \leq\vee\mathcal{B}\left(  j\right)
\leq\vee\mathcal{B}$ and $\vee\mathcal{B}\left(  j,k\right)  \leq
\vee\mathcal{B}\left(  k\right)  \leq\vee\mathcal{B}$, it follows that
$m_{j}m_{j,k}=m_{k}m_{k,j}=m_{\left(  j,k\right)  }$, which is the unique
morphism $m_{\left(  j,k\right)  }:\vee\mathcal{B}\left(  j,k\right)
\rightarrow\vee\mathcal{B}$. In particular, $m_{\left(  k,j\right)
}=m_{\left(  j,k\right)  }$ for all $j,k$.

Put $\mathcal{E}^{p}\left(  \mathfrak{b}_{U}\right)  =\prod\left\{
\vee\mathcal{B}:\mathcal{B}\in\mathfrak{b}_{U}^{p+1}\right\}  $, $p\geq0$ to
be the Fr\'{e}chet spaces with the connecting morphisms
\[
\partial_{U}^{p}:\mathcal{E}^{p}\left(  \mathfrak{b}_{U}\right)
\rightarrow\mathcal{E}^{p+1}\left(  \mathfrak{b}_{U}\right)  ,\quad\left(
\partial_{U}^{p}f\right)  \left(  \mathcal{B}\right)  =\sum_{j=0}^{p}\left(
-1\right)  ^{j}m_{j}f\left(  \mathcal{B}\left(  j\right)  \right)  ,
\]
where $f\in\mathcal{E}^{p}\left(  \mathfrak{b}_{U}\right)  $, $f\left(
\mathcal{B}\right)  \in\vee\mathcal{B}$ for every $\mathcal{B}\in
\mathfrak{b}_{U}^{p+1}$. If $f=\partial_{U}^{p-1}g$ for some $g\in
\mathcal{E}^{p-1}\left(  \mathfrak{b}_{U}\right)  $, then
\[
\left(  \partial_{U}^{p}f\right)  \left(  \mathcal{B}\right)  =\sum_{j=0}%
^{p}\left(  -1\right)  ^{j}m_{j}\partial_{U}^{p-1}g\left(  \mathcal{B}\left(
j\right)  \right)  =\sum_{j,k}\left(  -1\right)  ^{j+k}z_{j,k},
\]
where $z_{j,k}=m_{j}m_{j,k}g\left(  \mathcal{B}\left(  j,k\right)  \right)  $.
For every couple $k<j$, the latter sum contains $\left(  -1\right)
^{j+k}z_{j,k}$ and $\left(  -1\right)  ^{j+k-1}z_{k,j}$. But $z_{j,k}%
=m_{\left(  j,k\right)  }g\left(  \mathcal{B}\left(  j,k\right)  \right)
=m_{\left(  k,j\right)  }g\left(  \mathcal{B}\left(  k,j\right)  \right)
=z_{k,j}$, therefore $\left(  \partial_{U}^{p}f\right)  \left(  \mathcal{B}%
\right)  =0$. Thus $\partial_{U}^{p}\partial_{U}^{p-1}=0$, which means that
$\mathcal{E}^{\bullet}\left(  \mathfrak{b}_{U}\right)  =\left\{
\mathcal{E}^{p}\left(  \mathfrak{b}_{U}\right)  ,\partial_{U}^{p}%
:p\geq0\right\}  $ is a cochain complex of the Fr\'{e}chet spaces. Note that
$\mathcal{E}^{0}\left(  \mathfrak{b}_{U}\right)  =\prod\mathfrak{b}_{U}$,
$\wedge U=\wedge\mathfrak{b}_{U}$, and the (unique) morphisms $m_{\mathcal{B}%
}:\wedge U\rightarrow\mathcal{B}$, $\mathcal{B}\in\mathfrak{b}_{U}$ define in
turn the morphism
\[
\varepsilon_{U}:\wedge U\rightarrow\mathcal{E}^{0}\left(  \mathfrak{b}%
_{U}\right)  ,\quad\varepsilon_{U}=\prod\left\{  m_{\mathcal{B}}%
:\mathcal{B}\in\mathfrak{b}_{U}\right\}  ,
\]
that is, $\varepsilon_{U}\left(  a\right)  \left(  \mathcal{B}\right)
=m_{\mathcal{B}}\left(  a\right)  $, $a\in\wedge U$. Since $\mathcal{A}%
_{U}\leq\mathcal{B}_{j}\leq\mathcal{B}_{0}\vee\mathcal{B}_{1}$, $j=0,1$ for
every $2$-tuple $\mathcal{B=}\left(  \mathcal{B}_{0},\mathcal{B}_{1}\right)
\in\mathfrak{b}_{U}^{2}$, it follows that
\[
\partial_{U}^{0}\left(  \varepsilon_{U}\left(  a\right)  \right)  \left(
\left(  \mathcal{B}_{0},\mathcal{B}_{1}\right)  \right)  =m_{0}m_{\mathcal{B}%
_{1}}\left(  a\right)  -m_{1}m_{\mathcal{B}_{0}}\left(  a\right)  =0,\quad
a\in\wedge U.
\]
Thus $\varepsilon_{U}:\wedge U\rightarrow\mathcal{E}^{\bullet}\left(
\mathfrak{b}_{U}\right)  $ is an augmentation of the complex $\mathcal{E}%
^{\bullet}\left(  \mathfrak{b}_{U}\right)  $. The augmented complex
$\mathcal{E}^{\bullet}\left(  \mathfrak{b}_{U}\right)  $ is called the
\v{C}ech complex of the (countable) family $\mathfrak{b}_{U}$. Notice that it
is a complex of the $\wedge U$-bimodules, that is, $\wedge U\rightarrow
\mathcal{E}^{\bullet}\left(  \mathfrak{b}_{U}\right)  $ is an object of the
category $\overline{\wedge U\text{-}\operatorname*{mod}\text{-}\wedge U}$. If
$\mathcal{S}$ is unital then $A=\wedge\mathcal{S=}\wedge\mathfrak{b}$ for
$U=\mathcal{S}$. In this case, $\mathcal{E}^{\bullet}\left(  \mathfrak{b}%
\right)  =\left\{  \mathcal{E}^{p}\left(  \mathfrak{b}\right)  ,\partial
^{p}:p\geq0\right\}  $ is the \v{C}ech complex of the basis $\mathfrak{b}$
with the augmentation $\varepsilon:A\rightarrow\mathcal{E}^{\bullet}\left(
\mathfrak{b}\right)  $, $\varepsilon=\varepsilon_{\mathcal{S}}$. It is an
object of the category $\overline{A\text{-}\operatorname*{mod}\text{-}A}$.

\subsection{The projection morphism}

Let $\mathcal{S}$ be a unital complete lattice $A$-category with its basis
$\mathfrak{b}$, $U\subseteq\mathcal{S}$ an open subset with $\wedge
U=\wedge\mathfrak{b}_{U}$, where $\mathfrak{b}_{U}=U\cap\mathfrak{b}$. Let us
define the projection morphism
\[
\varphi_{p}:\mathcal{E}^{p}\left(  \mathfrak{b}\right)  \longrightarrow
\mathcal{E}^{p}\left(  \mathfrak{b}_{U}\right)  ,\quad\varphi_{p}\left(
f\right)  \left(  \mathcal{B}\right)  =f\left(  \mathcal{B}\right)
,\quad\mathcal{B}\in\mathfrak{b}_{U}^{p+1},
\]
whose kernel is reduced to $\prod\left\{  \vee\mathcal{B}:\mathcal{B}%
\in\mathfrak{k}_{p}\right\}  $, where $\mathfrak{k}_{p}$ consists of those
$\mathcal{B}=\left(  \mathcal{B}_{0},\ldots,\mathcal{B}_{p}\right)
\in\mathfrak{b}^{p+1}$ such that $\mathcal{B}_{i}\in\mathfrak{b-b}_{U}$ for
some $i$. The family $\overline{\varphi}=\left\{  \varphi_{p}\right\}  $
defines the morphism $\mathcal{E}^{\bullet}\left(  \mathfrak{b}\right)
\rightarrow\mathcal{E}^{\bullet}\left(  \mathfrak{b}_{U}\right)  $ of the
\v{C}ech complexes, that is,
\[
\left(  \varphi_{p+1}\left(  \partial^{p}f\right)  \right)  \left(
\mathcal{B}\right)  =\left(  \partial^{p}f\right)  \left(  \mathcal{B}\right)
=\sum_{j=0}^{p}\left(  -1\right)  ^{j}m_{j}f\left(  \mathcal{B}\left(
j\right)  \right)  =\partial_{U}^{p}\left(  \varphi_{p}\left(  f\right)
\right)  \left(  \mathcal{B}\right)
\]
for all $\mathcal{B}\in\mathfrak{b}_{U}^{p+1}$, which means that
$\varphi_{p+1}\left(  \partial^{p}f\right)  =\partial_{U}^{p}\left(
\varphi_{p}\left(  f\right)  \right)  $ for all $f\in\mathcal{E}^{p}\left(
\mathfrak{b}\right)  $. Thus $\overline{\varphi}$ is a morphism of the
category $\overline{A\text{-}\operatorname*{mod}\text{-}A}$. Since
$A=\wedge\mathfrak{b\leq\wedge}\mathfrak{b}_{U}=\wedge U\leq\mathfrak{b}_{U}$,
it follows that the diagram
\[%
\begin{array}
[c]{ccc}%
\mathcal{E}^{0}\left(  \mathfrak{b}\right)  =\prod\mathfrak{b} &
\overset{\varphi_{0}}{\longrightarrow} & \prod\mathfrak{b}_{U}=\mathcal{E}%
^{0}\left(  \mathfrak{b}_{U}\right) \\
_{\varepsilon}\uparrow &  & _{\varepsilon_{U}}\uparrow\\
A & \longrightarrow & \wedge U
\end{array}
\]
commutes. Thus $\overline{\varphi}$ is a morphism of the augmented \v{C}ech
complexes from $\overline{A\text{-}\operatorname*{mod}\text{-}A}$.

Now let $X$ be a finite-free left $A$-module with its free resolution
$\mathcal{R}^{\bullet}=A\widehat{\otimes}\mathfrak{e}^{\bullet}$, where
$\mathfrak{e}^{\bullet}=\left\{  E_{k}:-n\leq k\leq0\right\}  $ is a family of
Fr\'{e}chet spaces. The complex $A$-morphism $\overline{\varphi}$ in turn
induces a Fr\'{e}chet (space) bicomplex epimorphism
\[
\overline{\varphi}\otimes_{A}1:\overline{\mathcal{E}^{\bullet}\left(
\mathfrak{b}\right)  \widehat{\otimes}_{A}\mathcal{R}^{\bullet}}%
\rightarrow\overline{\mathcal{E}^{\bullet}\left(  \mathfrak{b}_{U}\right)
\widehat{\otimes}_{A}\mathcal{R}^{\bullet}}%
\]
with its kernel $\overline{\mathcal{N}}$. The $p$-th column of the bicomplex
$\overline{\mathcal{E}^{\bullet}\left(  \mathfrak{b}\right)  \widehat{\otimes
}_{A}\mathcal{R}^{\bullet}}$ (respectively, $\overline{\mathcal{E}^{\bullet
}\left(  \mathfrak{b}_{U}\right)  \widehat{\otimes}_{A}\mathcal{R}^{\bullet}}%
$) is the complex $\mathcal{E}^{p}\left(  \mathfrak{b}\right)
\widehat{\otimes}_{A}\mathcal{R}^{\bullet}$ (respectively, $\mathcal{E}%
^{p}\left(  \mathfrak{b}_{U}\right)  \widehat{\otimes}_{A}\mathcal{R}%
^{\bullet}$) (see Subsection \ref{subsecDomM}). Note that%
\[
\mathcal{E}^{p}\left(  \mathfrak{b}\right)  \widehat{\otimes}_{A}%
\mathcal{R}^{\bullet}\mathcal{=E}^{p}\left(  \mathfrak{b}\right)
\widehat{\otimes}\mathfrak{e}^{\bullet}\text{,\quad}\mathcal{E}^{p}\left(
\mathfrak{b}_{U}\right)  \widehat{\otimes}_{A}\mathcal{R}^{\bullet
}\mathcal{=E}^{p}\left(  \mathfrak{b}_{U}\right)  \widehat{\otimes
}\mathfrak{e}^{\bullet}.
\]
The action of $\overline{\varphi}\otimes_{A}1$ over $p$-th column is given by
the morphisms%
\[
\varphi_{p}\otimes1_{k}:\mathcal{E}^{p}\left(  \mathfrak{b}\right)
\widehat{\otimes}E_{k}\rightarrow\mathcal{E}^{p}\left(  \mathfrak{b}%
_{U}\right)  \widehat{\otimes}E_{k},\quad-n\leq k\leq0.
\]
Since the sequence
\[
0\rightarrow\prod\left\{  \vee\mathcal{B}:\mathcal{B}\in\mathfrak{k}%
_{p}\right\}  \rightarrow\mathcal{E}^{p}\left(  \mathfrak{b}\right)
\overset{\varphi_{p}}{\longrightarrow}\mathcal{E}^{p}\left(  \mathfrak{b}%
_{U}\right)  \rightarrow0
\]
splits (or admissible), it follows that
\[
\ker\left(  \varphi_{p}\otimes1_{k}\right)  =\prod\left\{  \vee\mathcal{B}%
:\mathcal{B}\in\mathfrak{k}_{p}\right\}  \widehat{\otimes}E_{k}=\prod\left\{
\left(  \vee\mathcal{B}\right)  \widehat{\otimes}E_{k}:\mathcal{B}%
\in\mathfrak{k}_{p}\right\}
\]
(see \cite[2.5.19]{HelHom}). Thus
\[
0\rightarrow\mathcal{N}^{p}\rightarrow\mathcal{E}^{p}\left(  \mathfrak{b}%
\right)  \widehat{\otimes}_{A}\mathcal{R}^{\bullet}\overset{\varphi_{p}%
\otimes_{A}1}{\longrightarrow}\mathcal{E}^{p}\left(  \mathfrak{b}_{U}\right)
\widehat{\otimes}_{A}\mathcal{R}^{\bullet}\rightarrow0
\]
is an exact sequence of the complexes in $\overline{A\text{-}%
\operatorname*{mod}}$, where $\mathcal{N}^{p}$ is the $p$-th column of the
bicomplex $\overline{\mathcal{N}}$. Moreover, using the well known fact
\cite[Ch 2.5, Theorem 5.19]{HelHom} that the projective tensor product and the
direct product are compatible operations, we obtain that
\[
\mathcal{N}^{p}=\prod\left\{  \left(  \vee\mathcal{B}\right)  \widehat{\otimes
}_{A}\mathcal{R}^{\bullet}:\mathcal{B}\in\mathfrak{k}_{p}\right\}
=\prod_{\mathcal{B}\in\mathfrak{k}_{p}}\left(  \vee\mathcal{B}\right)
\widehat{\otimes}\mathfrak{e}^{\bullet}%
\]
is the direct product of the complexes $\left(  \vee\mathcal{B}\right)
\widehat{\otimes}\mathfrak{e}^{\bullet}$, $\mathcal{B}\in\mathfrak{k}_{p}$.

\begin{lemma}
\label{td1}Let $\mathcal{S}$ be a unital complete lattice $A$-category with a
basis $\mathfrak{b}$, $X$ a finite-free $A$-module, and let $U\subseteq
\mathcal{S}$ be an open neighborhood of the spectrum $\sigma\left(
\mathcal{S},X\right)  $. If $\left(  \mathcal{E}^{\bullet}\left(
\mathfrak{b}\right)  ,\varepsilon\right)  \gg X$ then $\left(  \mathcal{E}%
^{\bullet}\left(  \mathfrak{b}_{U}\right)  ,\varepsilon\right)  \gg X$.
\end{lemma}

\begin{proof}
The projection morphism $\overline{\varphi}$ generates the exact sequence of
the bicomplexes%
\begin{equation}
0\rightarrow\overline{\mathcal{N}}\rightarrow\overline{\mathcal{E}^{\bullet
}\left(  \mathfrak{b}\right)  \widehat{\otimes}_{A}\mathcal{R}^{\bullet}%
}\rightarrow\overline{\mathcal{E}^{\bullet}\left(  \mathfrak{b}_{U}\right)
\widehat{\otimes}_{A}\mathcal{R}^{\bullet}}\rightarrow0\text{,} \label{fcb*}%
\end{equation}
where $\mathcal{R}^{\bullet}=A\widehat{\otimes}\mathfrak{e}^{\bullet}$ is a
finite free resolution of $X$. As we have seen above the $p$-th column
$\mathcal{N}^{p}$ of the kernel $\overline{\mathcal{N}}$ is $\prod\left\{
\left(  \vee\mathcal{B}\right)  \widehat{\otimes}\mathfrak{e}^{\bullet
}:\mathcal{B}\in\mathfrak{k}_{p}\right\}  $. If $\mathcal{B}\in\mathfrak{k}%
_{p}$ then $\mathcal{B}=\left(  \mathcal{B}_{0},\ldots,\mathcal{B}_{p}\right)
\in\mathfrak{b}^{p+1}$ such that $\mathcal{B}_{i}\in\mathfrak{b-b}_{U}$ for
some $i$. In particular, $\mathcal{B}_{i}\notin U$, therefore $\mathcal{B}%
_{i}\in\operatorname{res}\left(  \mathcal{S},X\right)  $. Since $\mathcal{B}%
_{i}\leq\vee\mathcal{B}$, there is a (unique) morphism $\mathcal{B}%
_{i}\rightarrow\vee\mathcal{B}$. It follows that $\vee\mathcal{B}\perp X$
holds, which means that $\left(  \vee\mathcal{B}\right)  \widehat{\otimes
}\mathfrak{e}^{\bullet}=\left(  \vee\mathcal{B}\right)  \widehat{\otimes}%
_{A}\mathcal{R}^{\bullet}$ is exact. Thus $\mathcal{N}^{p}$ is the direct
product of the exact complexes $\left(  \vee\mathcal{B}\right)
\widehat{\otimes}\mathfrak{e}^{\bullet}$, $\mathcal{B}\in\mathfrak{k}_{p}$,
therefore it is exact. The bicomplex $\overline{\mathcal{N}}$ with its exact
columns has the exact total complex $\mathcal{N}$. Thus (\ref{fcb*}) is an
exact sequence of bicomplexes such that the total complex of $\overline
{\mathcal{N}}$ is exact. Therefore, $\left(  \mathcal{E}^{\bullet}\left(
\mathfrak{b}\right)  ,\varepsilon\right)  \gg X$ (see Subsection
\ref{subsecDomM}) implies $\left(  \mathcal{E}^{\bullet}\left(  \mathfrak{b}%
_{U}\right)  ,\varepsilon_{U}\right)  \gg X$ by virtue of \cite[Lemma
2.3]{DosJOT10}.
\end{proof}

\subsection{\v{C}ech $A$-category}

Let $\mathcal{S}$ be a unital complete lattice $A$-category with a basis
$\mathfrak{b}$. Then we come up with the augmented \v{C}ech complex
\begin{equation}
0\rightarrow A\overset{\varepsilon}{\longrightarrow}\mathcal{E}^{\bullet
}\left(  \mathfrak{b}\right)  \label{ACc}%
\end{equation}
of the basis $\mathfrak{b}$. Let us introduce the following categories.

\begin{definition}
\label{defCech}A unital complete lattice $A$-category $\mathcal{S}$ is called
\textit{a \v{C}ech }$A$-\textit{category if (\ref{ACc}) is exact for a certain
}basis $\mathfrak{b}$ for $\mathcal{S}$.
\end{definition}

If $\mathcal{P}$ is a Fr\'{e}chet algebra presheaf corresponding (see
Proposition \ref{corCes1}) to a \v{C}ech $A$-category $\mathcal{S}$, then
$\mathcal{P}$ has properties which make it closer to a sheaf rather than just
a presheaf. Namely, if an element of the algebra $A$ being a global section of
the presheaf $\mathcal{P}$ vanishing over all open subsets from $\mathfrak{b}%
$, then it is zero. Moreover, a compatible over $\mathfrak{b}$ family of
sections of $\mathcal{P}$ defines a unique global section from $A$.

If $\mathfrak{b}$ consists of nuclear algebras additionally, then we say that
$\mathcal{S}$ is a\textit{ \v{C}ech }$A$\textit{-category with a nuclear
basis. }Since (\ref{ACc}) is exact and $\mathcal{E}^{0}\left(  \mathfrak{b}%
\right)  =\prod\mathfrak{b}$, it follows that $A$ is nuclear automatically
whenever $\mathfrak{b}$ is a nuclear basis. As above take a finite free
resolution $\mathcal{R}^{\bullet}\mathcal{=}A\widehat{\otimes}\mathfrak{e}%
^{\bullet}$ of a left $A$-module $X$. The augmentation $\overline{\varepsilon
}:A\rightarrow\mathcal{E}^{\bullet}\left(  \mathfrak{b}\right)  $ being a
morphism of the category $\overline{A\text{-}\operatorname*{mod}\text{-}A}$,
defines the morphism
\begin{equation}
\overline{\varepsilon}\otimes_{A}1:\overline{A\widehat{\otimes}_{A}%
\mathcal{R}^{\bullet}}\longrightarrow\overline{\mathcal{E}^{\bullet}\left(
\mathfrak{b}\right)  \widehat{\otimes}_{A}\mathcal{R}^{\bullet}} \label{Epst1}%
\end{equation}
of the bicomplexes. The $k$-th row $\mathcal{E}^{\bullet}\left(
\mathfrak{b}\right)  \widehat{\otimes}_{A}\left(  A\widehat{\otimes}%
E_{k}\right)  $ of the bicomplex $\overline{\mathcal{E}^{\bullet}\left(
\mathfrak{b}\right)  \widehat{\otimes}_{A}\mathcal{R}^{\bullet}}$ is reduced
to $\mathcal{E}^{\bullet}\left(  \mathfrak{b}\right)  \widehat{\otimes}E_{k}$,
whereas the related row of the first bicomplex $\overline{A\widehat{\otimes
}_{A}\mathcal{R}^{\bullet}}$ (see Subsection \ref{subsecDomM}) is
$0\rightarrow A\widehat{\otimes}E_{k}\rightarrow0\rightarrow0\rightarrow
\cdots$ (the first bicomplex has only one nontrivial column $\mathcal{R}%
^{\bullet}$). Thus the action of the morphism $\overline{\varepsilon}%
\otimes_{A}1$ over $k$-th row is given by $\varepsilon\otimes
1:A\widehat{\otimes}E_{k}\rightarrow\mathcal{E}^{\bullet}\left(
\mathfrak{b}\right)  \widehat{\otimes}E_{k}$.

\begin{lemma}
\label{td2}Let $\mathcal{S}$ be a \v{C}ech $A$-category with a nuclear basis
$\mathfrak{b}$, and let $X$ be a finite-free left $A$-module. Then $\left(
\mathcal{E}^{\bullet}\left(  \mathfrak{b}\right)  ,\varepsilon\right)  \gg X$ holds.
\end{lemma}

\begin{proof}
By assumption, $0\rightarrow A\overset{\varepsilon}{\longrightarrow
}\mathcal{E}^{\bullet}\left(  \mathfrak{b}\right)  $ is an exact complex of
the nuclear Fr\'{e}chet algebras. It follows that the complex%
\[
0\rightarrow A\widehat{\otimes}E_{k}\overset{\varepsilon\otimes
1}{\longrightarrow}\mathcal{E}^{\bullet}\left(  \mathfrak{b}\right)
\widehat{\otimes}E_{k}%
\]
remains exact \cite[A1.6 (d)]{EP} for every $k$. Thus the morphism
$\overline{\varepsilon}\otimes_{A}1$ from (\ref{Epst1}) is an embedding, and
put $\overline{\mathcal{M}}=\overline{\mathcal{E}^{\bullet}\left(
\mathfrak{b}\right)  \widehat{\otimes}_{A}\mathcal{R}^{\bullet}}%
/\operatorname{im}\left(  \overline{\varepsilon}\otimes_{A}1\right)  $ to be
the quotient bicomplex. The $k$-th row of $\overline{\mathcal{M}}$ is the
following exact complex
\[
\mathcal{E}^{\bullet}\left(  \mathfrak{b}\right)  \widehat{\otimes}%
E_{k}/\operatorname{im}\left(  \overline{\varepsilon}\otimes1\right)
:0\rightarrow\mathcal{E}^{0}\left(  \mathfrak{b}\right)  \widehat{\otimes
}E_{k}/\operatorname{im}\left(  \varepsilon\otimes1\right)  \rightarrow
\mathcal{E}^{1}\left(  \mathfrak{b}\right)  \widehat{\otimes}E_{k}%
\rightarrow\mathcal{E}^{2}\left(  \mathfrak{b}\right)  \widehat{\otimes}%
E_{k}\rightarrow\cdots\text{.}%
\]
The bicomplex $\overline{\mathcal{M}}$ with its exact rows has the exact total
complex. Thus
\[
0\rightarrow\overline{A\widehat{\otimes}_{A}\mathcal{R}^{\bullet}}%
\rightarrow\overline{\mathcal{E}^{\bullet}\left(  \mathfrak{b}\right)
\widehat{\otimes}_{A}\mathcal{R}^{\bullet}}\rightarrow\overline{\mathcal{M}%
}\rightarrow0
\]
is an exact sequence of the bicomplexes with the exact total complex of
$\overline{\mathcal{M}}$. Taking into account that $\left(  A,\iota
_{A}\right)  \gg X$ (see Subsection \ref{subsecDomM}), we conclude that
$\left(  \mathcal{E}^{\bullet}\left(  \mathfrak{b}\right)  ,\varepsilon
\right)  \gg X$ holds \cite[Lemma 2.3]{DosJOT10} too.
\end{proof}

Now assume that $\mathcal{S}$ is a unital complete lattice $A$-category with a
basis $\mathfrak{b}$, which corresponds to a Fr\'{e}chet algebra presheaf
$\mathcal{P}:\Omega^{\operatorname{op}}\rightarrow\mathfrak{Fa}$ on a
topological space $\omega$ (see Proposition \ref{corCes1}), that is,
$\mathcal{S}=\widehat{\mathcal{P}}$. Thus $\Omega$ is the category of all open
subsets of $\omega$, $\mathfrak{b}=\left\{  V_{i}:i\in I\right\}  $ is a
countable base for the topological space $\omega$, and $\mathcal{P}\left(
V_{i_{0}}\right)  \vee\ldots\vee\mathcal{P}\left(  V_{i_{p}}\right)
=\mathcal{P}\left(  V_{\alpha}\right)  $, where $V_{\alpha}=V_{i_{0}}%
\cap\ldots\cap V_{i_{p}}$ and $\alpha=\left(  i_{0},\ldots,i_{p}\right)  \in
I^{p+1}$. Moreover, $\mathcal{E}^{p}\left(  \mathfrak{b}\right)
=\prod\left\{  \mathcal{P}\left(  V_{\alpha}\right)  :\alpha\in I^{p+1}%
\right\}  $, $p\geq0$, and the related augmented \v{C}ech complex (\ref{ACc})
is just the \v{C}ech complex of the covering $\mathfrak{b}$. If $\mathcal{P}$
is a sheaf, then the cohomology groups $H^{p}\left(  \mathfrak{b}%
,\mathcal{P}\right)  $, $p\geq0$\ of the \v{C}ech complex are linked to the
sheaf cohomology groups $H^{p}\left(  \omega,\mathcal{P}\right)  $, $p\geq0$
of the topological space $\omega$. Namely, there are natural morphisms
$H^{p}\left(  \mathfrak{b},\mathcal{P}\right)  \rightarrow H^{p}\left(
\omega,\mathcal{P}\right)  $, $p\geq0$ functorial in $\mathcal{P}$ (see
\cite[4.4.4]{Harts}). The well known theorem of Larey \cite[Exercise
4.4.11]{Harts} asserts that these cohomology morphisms are isomorphisms indeed
whenever all finite intersections $V_{\alpha}$ of $\mathfrak{b}$ are
$\mathcal{P}$-acyclic, that is, $H^{p}\left(  V_{\alpha},\mathcal{P}%
|V_{\alpha}\right)  =\left\{  0\right\}  $, $p\geq1$.

\begin{proposition}
\label{propAcyclic}Let $\mathcal{F}:\Omega^{\operatorname{op}}\rightarrow
\mathfrak{Fa}$ be a Fr\'{e}chet algebra sheaf on a topological space $\omega$
with its countable topology base $\mathfrak{b}$. Suppose that all finite
intersections of $\mathfrak{b}$ are $\mathcal{F}$-acyclic. Then
$\widehat{\mathcal{F}}$ is \textit{a \v{C}ech }$A$-\textit{category} with the
basis $\mathfrak{b}$ if and only if $\omega$ is $\mathcal{F}$-acyclic.
\end{proposition}

\begin{proof}
By Theorem of Larey, we have $H^{p}\left(  \mathfrak{b},\mathcal{F}\right)
=H^{p}\left(  \omega,\mathcal{F}\right)  $, $p\geq0$ up to the natural
isomorphims. Moreover, $H^{0}\left(  \mathfrak{b},\mathcal{F}\right)
=H^{0}\left(  \omega,\mathcal{F}\right)  =\Gamma\left(  \omega,\mathcal{F}%
\right)  =A$. Hence the \v{C}ech complex (\ref{ACc}) of $\mathfrak{b}$ is
exact iff $H^{p}\left(  \omega,\mathcal{F}\right)  =0$, $p\geq1$, which means
that $\omega$ is $\mathcal{F}$-acyclic.
\end{proof}

Finally, notify that not every \v{C}ech $A$-category corresponds to a
Fr\'{e}chet algebra sheaf. For example, if $\mathcal{S}$ is a\textit{
}\v{C}ech $A$-category with a nuclear basis corresponding to a Fr\'{e}chet
algebra sheaf $\mathcal{F}:\Omega^{\operatorname{op}}\rightarrow\mathfrak{Fa}$
on a topological space $\omega$ (see Proposition \ref{propAcyclic}), then
(\ref{ACc}) is an exact complex of nuclear Fr\'{e}chet spaces. If
$\mathcal{P}:\Omega^{\operatorname{op}}\rightarrow\mathfrak{Fa}$ is a
Fr\'{e}chet algebra presheaf (on the same topological space $\omega$), which
is the constant $C$ over the basis $\mathfrak{b}$, then $\mathcal{F}%
\widehat{\otimes}\mathcal{P}:\Omega^{\operatorname{op}}\rightarrow
\mathcal{\mathfrak{Fa}\ }$is a Fr\'{e}chet algebra presheaf (which is not a
sheaf in general) and $\prod\left\{  \left(  \mathcal{F}\widehat{\otimes
}\mathcal{P}\right)  \left(  V_{\alpha}\right)  :\alpha\in I^{n+1}\right\}
=\prod\left\{  \mathcal{F}\left(  V_{\alpha}\right)  \widehat{\otimes}%
C:\alpha\in I^{n+1}\right\}  =\mathcal{E}^{n}\left(  \mathfrak{b}\right)
\widehat{\otimes}C$ for all $n$. Since the complex $\mathcal{E}^{\bullet
}\left(  \mathfrak{b}\right)  \widehat{\otimes}C$ remains exact, we conclude
that the presheaf $\mathcal{F}\widehat{\otimes}\mathcal{P}$ defines a\textit{
}\v{C}ech $A$-category.

\subsection{The functional calculus}

Now we can prove the following functional calculus theorem for a finite-free
left $A$-module with respect to a \v{C}ech $A$-category with a nuclear basis.

\begin{theorem}
\label{thcoreFC}Let $\mathcal{S}$ be a \v{C}ech $A$-category with a nuclear
basis, $X$ a finite-free left $A$-module and let $U$ be an open neighborhood
of the spectrum $\sigma\left(  \mathcal{S},X\right)  $ in $\mathcal{S}$. Then
$X$ turns out to be a left $\wedge U$-module extending its $A$-module
structure through the morphism $A\rightarrow\wedge U$.
\end{theorem}

\begin{proof}
By assumption there is a nuclear basis $\mathfrak{b}$ for $\mathcal{S}$ whose
augmented \v{C}ech complex (\ref{ACc}) is exact (see Definition \ref{defCech}%
). By Lemma \ref{td2}, the dominance property $\left(  \mathcal{E}^{\bullet
}\left(  \mathfrak{b}\right)  ,\varepsilon\right)  \gg X$ holds. Since
$\sigma\left(  \mathcal{S},X\right)  \subseteq U$, based on Lemma \ref{td1},
we deduce that $\left(  \mathcal{E}^{\bullet}\left(  \mathfrak{b}_{U}\right)
,\varepsilon_{U}\right)  \gg X$ holds too. But $\mathcal{E}^{\bullet}\left(
\mathfrak{b}_{U}\right)  $ is an object of the category $\overline{\wedge
U\text{-}\operatorname*{mod}\text{-}A}$, and the augmentation $\varepsilon
:A\rightarrow\mathcal{E}^{\bullet}\left(  \mathfrak{b}_{U}\right)  $ is a
morphism of the left $A$-modules. It remains to apply Theorem \ref{tHTf}.
\end{proof}

Thus, if $\mathcal{S}$ is a \v{C}ech $A$-category with a nuclear basis
corresponding to a Fr\'{e}chet algebra presheaf $\left(  \Omega,\mathcal{F}%
\right)  $ on a topological space $\omega$, and $X$ is a finite-free left
$A$-module, then the original $\mathcal{F}\left(  \omega\right)  $-calculus on
$X$ is extended up to a $\mathcal{F}\left(  U\right)  $-calculus on an open
subset $U\subseteq\omega$, whenever $\sigma\left(  \mathcal{S},X\right)
\subseteq U$. More precisely, suppose that $\left(  \omega,\mathcal{F}\right)
$ is a geometric model of $A$ with its nuclear Fr\'{e}chet algebra presheaf
$\mathcal{F}$, and $\omega$ has a countable topology base $\mathfrak{b}%
=\left\{  V_{i}\right\}  $ such that its augmented \v{C}ech complex
(\ref{ACc}) is exact. In this case, the functional calculus property from
Theorem \ref{thcoreFC} holds. In the case of a standard geometric model with
its structure sheaf, we obtain that the following assertion.

\begin{corollary}
\label{corkeyy1}Let $\left(  \operatorname{Spec}\left(  A\right)
,\mathcal{F}\right)  $ be a standard analytic geometry of $A$, whose structure
sheaf $\mathcal{F}$ is a nuclear Fr\'{e}chet algebra sheaf. Suppose
$\operatorname{Spec}\left(  A\right)  $ is $\mathcal{F}$-acyclic and
$\operatorname{Spec}\left(  A\right)  $ has a countable basis $\mathfrak{b}$,
whose all finite intersections are $\mathcal{F}$-acyclic. If $X$ is a
finite-free left $A$-module, and $U$ is an open neighborhood of the Putinar
spectrum $\sigma_{\operatorname{P}}\left(  \mathcal{F},X\right)  $, then $X$
turns out to be a Fr\'{e}chet left $\mathcal{F}\left(  U\right)  $-module
extending its $A$-module structure.
\end{corollary}

\begin{proof}
One needs to apply Proposition \ref{remPut1}, Proposition \ref{propAcyclic},
and Theorem \ref{thcoreFC}.
\end{proof}

In the case of an algebra of finite type $A$ we have both the functional
calculus and the spectral mapping theorem for a complex analytic geometry of
$A$.

\begin{theorem}
\label{corkeyy2}Let $A$ be a finite type algebra with its complex analytic
geometry $\left(  \omega,\mathcal{P}\right)  $ such that $\mathcal{P}$ is a
nuclear Fr\'{e}chet algebra presheaf, $\mathfrak{b}=\left\{  V_{i}\right\}  $
a countable topology base of $\omega$ such that its augmented \v{C}ech complex
is exact, \textit{and let }$X$ be a left $A$-module. If $U\subseteq\omega$ is
an open subset containing the spectrum $\sigma\left(  \mathcal{P},X\right)  $,
then $X$ turns out to be a left $\mathcal{A}$-module extending its $A$-module
structure through the restriction morphism $A\rightarrow\mathcal{A}$, where
$\mathcal{A=P}\left(  U\right)  $. Moreover, if $A\rightarrow\mathcal{A}$ is a
localization, then $\mathcal{A}\in U$ and%
\[
\sigma\left(  \mathcal{P},X\right)  =\sigma\left(  \mathcal{P}|_{U},X\right)
,\quad\sigma_{\operatorname{P}}\left(  \mathcal{P},X\right)  =\sigma
_{\operatorname{P}}\left(  \mathcal{P}|_{U},X\right)  ,\quad\sigma\left(
A,X\right)  \cap U=\sigma\left(  \mathcal{P}\left(  U\right)  ,X\right)  .
\]
If $\mathfrak{b}$ is a localizing base additionally, then $\sigma\left(
A,X\right)  =\sigma\left(  \mathcal{P}\left(  U\right)  ,X\right)  $ holds too.
\end{theorem}

\begin{proof}
First notice that $\left(  \omega,\mathcal{P}\right)  $ defines the nuclear
\v{C}ech $A$-category $\widehat{\mathcal{P}}$ with its basis $\mathfrak{b}%
$\textit{ }(see Subsection \ref{subsecTBC}), and the open set $U$ is
identified with an open subcategory of $\widehat{\mathcal{P}}$ such that
$\wedge U=\mathcal{A}$. Since every left $A$-module if finite-free
automatically, Theorem \ref{thcoreFC} is applicable. Thus $X$ is a left
$\mathcal{A}$-module extending its $A$-module structure through the
restriction morphism $\iota_{U}:A\rightarrow\mathcal{A}$.

Further, assume that $\iota_{U}:A\rightarrow\mathcal{A}$ is a localization.
Based on Lemma \ref{lemDom12}, we deduce that $\left(  \mathcal{A},\iota
_{U}\right)  \gg X$ holds, and the spectral mapping Theorem \ref{thSMT} (see
also Corollary \ref{corM2}) is applicable. In particular, $\mathcal{A}%
\in\sigma\left(  \mathcal{P},X\right)  $. It follows that $\mathcal{A}\in U$
(or $U=U_{\mathcal{A}}$) and
\begin{align*}
\sigma\left(  \mathcal{P},X\right)   &  =\sigma\left(  \mathcal{P},X\right)
\cap U=\sigma\left(  \mathcal{P},X\right)  \cap U_{\mathcal{A}}=\sigma\left(
\mathcal{P},X\right)  |_{\mathcal{A}}=\sigma\left(  U_{\mathcal{A}},X\right)
=\sigma\left(  \mathcal{P}|_{U},X\right)  ,\\
\sigma_{\operatorname{P}}\left(  \mathcal{P},X\right)   &  =\sigma\left(
\mathcal{P},X\right)  \cap\operatorname{Spec}\left(  A\right)  =\sigma\left(
\mathcal{P},X\right)  \cap\operatorname{Spec}\left(  A\right)  \cap
U=\sigma\left(  \mathcal{P},X\right)  \cap\operatorname{Spec}\left(
\mathcal{A}\right) \\
&  =\sigma_{\operatorname{P}}\left(  \mathcal{P},X\right)  \cap
\operatorname{Spec}\left(  \mathcal{A}\right)  =\sigma_{\operatorname{P}%
}\left(  \mathcal{P},X\right)  |_{\mathcal{A}}=\sigma_{\operatorname{P}%
}\left(  U_{\mathcal{A}},X\right)  =\sigma_{\operatorname{P}}\left(
\mathcal{P}|_{U},X\right)  ,\\
\sigma\left(  A,X\right)  \cap U  &  =\sigma\left(  A,X\right)  \cap
\operatorname{Spec}\left(  \mathcal{A}\right)  =\sigma\left(  \mathcal{T}%
,X\right)  |_{\mathcal{A}}=\sigma\left(  \mathcal{P}\cap\mathcal{T},X\right)
|_{\mathcal{A}}\\
&  =\sigma\left(  \operatorname{Spec}\left(  \mathcal{A}\right)  ,X\right)
=\sigma\left(  \mathcal{P}\left(  U\right)  ,X\right)  .
\end{align*}
Finally, if $\mathfrak{b}$ is a localizing base additionally, that is,
$\mathcal{P}$ is a localizing presheaf, then $\sigma\left(  A,X\right)
=\sigma\left(  \mathcal{P}\cap\mathcal{T},X\right)  \subseteq\sigma
_{\operatorname{P}}\left(  \mathcal{P},X\right)  \subseteq\sigma\left(
\mathcal{P},X\right)  \subseteq U$ thanks to Corollary \ref{corM1}. Hence
$\sigma\left(  A,X\right)  =\sigma\left(  A,X\right)  \cap U=\sigma\left(
\mathcal{P}\left(  U\right)  ,X\right)  $.
\end{proof}

\section{The standard analytic geometries of the contractive quantum
plane\label{sectionCQP}}

Recall that a quantum plane $\mathfrak{A}_{q}$ is the quotient of the free
algebra $\mathbb{C}\left\langle x,y\right\rangle $ modulo the identity
$xy=q^{-1}yx$, $\left\vert q\right\vert \leq1$. If $\left\vert q\right\vert
<1$, it is called a contractive quantum plane. In this section we consider the
standard geometric models of some envelopes of $\mathfrak{A}_{q}$ depending on
the classes of its Banach algebra actions.

\subsection{The multinormed envelopes of an algebra\label{subsecMEQ}}

Let $\mathfrak{A}$ be a pure (unital or not) base algebra. A class
$\mathcal{C}$ of algebra homomorphisms $\mathfrak{A}\rightarrow B$ into Banach
algebras $B$ define the related envelope $\widehat{\mathfrak{A}}^{\mathcal{C}%
}$ to be an Arens-Michael algebra with the canonical algebra homomorphisms
$\iota:\mathfrak{A}\rightarrow\widehat{\mathfrak{A}}^{\mathcal{C}}$ such that
$\left(  \widehat{\mathfrak{A}}^{\mathcal{C}},\iota\right)  $ possesses the
following universal projective property. Every homomorphism $\kappa
:\mathfrak{A}\rightarrow B$ from the class can be factorized as $\kappa
=\widehat{\kappa}\iota$ for a unique continuous algebra homomorphism
$\widehat{\kappa}:\widehat{\mathfrak{A}}^{\mathcal{C}}\rightarrow B$.
Actually, every $\kappa:\mathfrak{A}\rightarrow B$ from the class
$\mathcal{C}$ defines a submultiplicative seminorm $p_{\kappa}=\left\Vert
\kappa\left(  \cdot\right)  \right\Vert _{B}$ on $\mathfrak{A}$. The Hausdorff
completion of $\mathfrak{A}$ with respect to the multinormed topology given by
the family $\left\{  p_{\kappa}:\kappa\in\mathcal{C}\right\}  $ results in the
$\mathcal{C}$-envelope $\widehat{\mathfrak{A}}^{\mathcal{C}}$ of
$\mathfrak{A}$ (see \cite{Aris}, \cite{DosiSS}). If $\mathcal{C}$ is the class
of all Banach algebra actions $\mathfrak{A}\rightarrow B$ we obtain the
Arens-Michael envelope $\widehat{\mathfrak{A}}$ of $\mathfrak{A}$.

Let $B$ be a unital Banach algebra $B$. If $B$ has a closed nilpotent ideal
$I$, whose quotient $B/I$ is commutative, we say that $B$ is a \textit{Banach}
\textit{N-algebra. }Every Banach algebra $B$ with the finite-dimensional
Jacobson radical $\operatorname{Rad}B$ and the commutative quotient
$B/\operatorname{Rad}B$ is a Banach N-algebra with its ideal
$I=\operatorname{Rad}B$ (see \cite[1.3.59]{Hel}). Every unital commutative
Banach algebra is a Banach N-algebra. Notice that if $B$ is a Banach N-algebra
with its ideal $I$, then $I$ is a closed ideal of nilpotent elements, and
$I\subseteq\operatorname{Rad}B$. In particular, $B/\operatorname{Rad}B$ is a
unital commutative Banach algebra being a quotient of $B/I$. If $\mathcal{C}$
is the class of all homomorphisms $\mathfrak{A}\rightarrow B$ from a unital
algebra $\mathfrak{A}$ into Banach N-algebras $B$, then we come up with the
N-envelope of $\mathfrak{A}$ denoted by $\widehat{\mathfrak{A}}^{\text{N}}$.

In the case of a non-unital base algebra $\mathfrak{A}$ one can introduce its
RN and RF envelopes of its unitization $\mathfrak{A}_{+}$ \cite{DosiSS}.
Recall that a unital algebra $A$ is called a local algebra if it has a unique
maximal (left or right) ideal, which is just $\operatorname{Rad}A$. In
particular, $A/\operatorname{Rad}A=\mathbb{C}$. If $A$ has no unit and
$A=\operatorname{Rad}A$, then $A$ is called a radical algebra. In this case,
its unitization $A_{+}$ turns out to be a local algebra. By a morphism of the
radical algebras we always mean the local morphism $\varphi_{+}:A_{+}%
\rightarrow B_{+}$ obtained by the unitization of a homomorphism
$\varphi:A\rightarrow B$ of the radical algebras. Let us consider the class
$\mathcal{C}$ of all homomorphisms $\kappa:\mathfrak{A}\rightarrow B$ into
nilpotent (radical) Banach algebras $B$. Actually, $\mathcal{C}$ is the class
of homomorphism $\kappa_{+}:\mathfrak{A}_{+}\rightarrow B_{+}$ into local
Banach algebras $B_{+}$ with their nilpotent radicals, which in turn defines
the RN-envelope $\widehat{\mathfrak{A}_{+}}^{\text{RN}}$ of $\mathfrak{A}_{+}%
$. If we choose $\mathcal{C}$ to be the class of all finite dimensional
radical algebras, then we come up with the RF-envelope $\widehat{\mathfrak{A}%
_{+}}^{\text{RF}}$ of $\mathfrak{A}_{+}$.

In the case of the quantum plane $\mathfrak{A=A}_{q}$, $\left\vert
q\right\vert \leq1$, we have (see \cite{DosiSS})
\begin{equation}
\widehat{\mathfrak{A}_{q}}^{\text{RN}}=\widehat{\mathfrak{A}_{q}}^{\text{RF}%
}=\mathbb{C}_{q}\left[  \left[  x,y\right]  \right]  , \label{Cq}%
\end{equation}
where $\mathfrak{A}_{q}$ is considered as the unitization of the subalgebra
generated by $x$ and $y$, and $\mathbb{C}_{q}\left[  \left[  x,y\right]
\right]  $ is the (local) Arens-Michael-Fr\'{e}chet algebra of all formal
power series in variables $x$ and $y$ equipped with the $q$-multiplication
(see below (\ref{lm1})). In the case of a contractive quantum plane
$\mathfrak{A}_{q}$ the class $\mathcal{C}$ of all homomorphisms $\kappa
:\mathfrak{A}_{q}\rightarrow B$ into unital Banach algebras $B$ with the
nilpotent $\kappa\left(  y\right)  $ (resp., $\kappa\left(  x\right)  $)
defines the related envelope of $\mathfrak{A}_{q}$ denoted by
$\widehat{\mathfrak{A}_{q}}^{y}$ (resp., $\widehat{\mathfrak{A}_{q}}^{x}$).
Actually, $\widehat{\mathfrak{A}_{q}}^{\text{N}}$ is the envelope with respect
to the class of all homomorphisms $\kappa:\mathfrak{A}_{q}\rightarrow B$ with
the nilpotent $\kappa\left(  xy\right)  $ (see \cite{DosiSS}).

\subsection{Noncommutative analytic space $\left(  \mathbb{C}_{xy}%
,\mathcal{O}_{q}\right)  $}

The Arens-Michael envelope $A=\widehat{\mathfrak{A}_{q}}$ of the contractive
quantum plane $\mathfrak{A}_{q}$ turns out to be commutative modulo its
Jacobson radical \cite[Theorem 5.1]{Dosi24}, and the canonical embedding
$\mathfrak{A}_{q}\rightarrow A$ is an absolute localization \cite{Pir}. Taking
into account that $\mathfrak{A}_{q}$ has a finite free $\mathfrak{A}_{q}%
$-bimodule resolution \cite{Tah}, we deduce that the Arens-Michael-Fr\'{e}chet
algebra $A$ is of finite type. In particular, the irreducible Banach
$A$-modules (the spectrum $\mathfrak{X}$ of $A$) are just the trivial
$A$-modules $\mathbb{C}\left(  \gamma\right)  $ given by the continuous
characters $\gamma$ (see \cite[6.2.1]{Hel}), that is, $\mathfrak{X}%
=\operatorname{Spec}A$. Moreover, we have
\[
\mathfrak{X}=\mathbb{C}_{xy}:=\mathbb{C}_{x}\cup\mathbb{C}_{y}\quad
\text{with\quad}\mathbb{C}_{x}=\mathbb{C\times}\left\{  0\right\}
\subseteq\mathbb{C}^{2},\quad\mathbb{C}_{y}=\left\{  0\right\}  \times
\mathbb{C\subseteq C}^{2}%
\]
(see \cite{Dosi24}, \cite{Dosi242}), and $A/\operatorname{Rad}A=\mathcal{O}%
\left(  \mathbb{C}_{xy}\right)  $ is the commutative algebra of all
holomorphic functions on $\mathbb{C}_{xy}$. The standard geometry of $A$ was
revealed in \cite{Dosi24}. Namely, the spectrum $\mathbb{C}_{xy}$ is equipped
with a suitable topology affiliated to the noncommutative multiplication of
$\mathfrak{A}_{q}$ or $A$. Namely, $\mathbb{C}_{x}$ is endowed with the
$\mathfrak{q}$-topology given by the $q$-spiraling open subsets of
$\mathbb{C}$, whereas $\mathbb{C}_{y}$ is equipped with the disk topology
$\mathfrak{d}$ given by all open disks in $\mathbb{C}$ centered at the origin.
Both are non-Hausdorff topologies with their unique generic point zero, and
they are weaker than the original topology of the complex plane. The space
$\mathbb{C}_{xy}$ is equipped with the final topology so that both embedding
\begin{equation}
\left(  \mathbb{C}_{x},\mathfrak{q}\right)  \hookrightarrow\mathbb{C}%
_{xy}\hookleftarrow\left(  \mathbb{C}_{y},\mathfrak{d}\right)  \label{cemd}%
\end{equation}
are continuous, which is called $\left(  \mathfrak{q},\mathfrak{d}\right)
$\textit{-topology of} $\mathbb{C}_{xy}$. It is the union $\mathbb{C}%
_{xy}=\mathbb{C}_{x}\cup\mathbb{C}_{y}$ of two irreducible subsets, whose
intersection is a unique generic point. The direct image of the standard
Fr\'{e}chet sheaf $\mathcal{O}$ of germs of holomorphic functions on
$\mathbb{C}$ along the continuous identity map $\mathbb{C}\rightarrow\left(
\mathbb{C}_{x},\mathfrak{q}\right)  $ is denoted by $\mathcal{O}%
^{\mathfrak{q}}$. In a similar way, we have the sheaf $\mathcal{O}%
^{\mathfrak{d}}$ on $\left(  \mathbb{C}_{y},\mathfrak{d}\right)  $. The
projective tensor product of the ringed spaces $\left(  \mathbb{C}%
_{x},\mathcal{O}^{\mathfrak{q}}\right)  $ and $\left(  \mathbb{C}%
_{y},\mathcal{O}^{\mathfrak{d}}\right)  $ results in the ringed space $\left(
\mathbb{C}_{xy},\mathcal{O}_{q}\right)  $. The underlying space $\mathbb{C}%
_{xy}$ is equipped with the $\left(  \mathfrak{q},\mathfrak{d}\right)
$\textit{-}topology and its structure presheaf is defined as
\[
\mathcal{O}_{q}=\mathcal{O}^{\mathfrak{q}}\widehat{\otimes}\mathcal{O}%
^{\mathfrak{d}},
\]
where both sheaves $\mathcal{O}^{\mathfrak{q}}$ and $\mathcal{O}%
^{\mathfrak{d}}$ are drawn onto $\mathbb{C}_{xy}$ as their direct images along
the canonical embeddings (\ref{cemd}). It turns out that $\mathcal{O}_{q}$ is
a presheaf of noncommutative Fr\'{e}chet algebras equipped with the formal
$q$-multiplication such that $\mathcal{O}_{q}\left(  \mathbb{C}_{xy}\right)
=A$. Moreover, for every open subset $U\subseteq\mathbb{C}_{xy}$, the algebra
$\mathcal{O}_{q}\left(  U\right)  $ is commutative modulo its Jacobson radical
$\operatorname{Rad}\mathcal{O}_{q}\left(  U\right)  $ and $\operatorname{Spec}%
\left(  \mathcal{O}_{q}\left(  U\right)  \right)  =U$. Thus $\left(
\mathbb{C}_{xy},\mathcal{O}_{q}\right)  $ is a standard geometry of the
algebra $A$ (see Lemma \ref{lemAG1} and Definition \ref{defAG}). Thus
$A\rightarrow\mathcal{O}_{q}$ is a unital complete lattice $A$-category.

As a topology base $\mathfrak{b}$ for $\mathbb{C}_{xy}$, we consider the
countable family $\left\{  V_{i}:i\in I\right\}  $ of all Runge $q$-open
subsets (see \cite{DosiLoc}) in $\left(  \mathbb{C}_{x},\mathfrak{q}\right)
$, and the countable family $\left\{  W_{n}:n\in\mathbb{N}\right\}  $ of open
disks in $\left(  \mathbb{C}_{y},\mathfrak{d}\right)  $, where $W_{n}$ is the
open disk in the complex plane centered at the origin and of radius $n$. Then
$\mathfrak{b}=\left\{  U_{i,n}:\left(  i,n\right)  \in I\times\mathbb{N}%
\right\}  $ is a countable topology base for $\mathbb{C}_{xy}$, where
$U_{i,n}=V_{i}\cup W_{n}$.

\begin{proposition}
\label{propCex1}The unital complete lattice $A$-category $\mathcal{O}_{q}$
with the basis $\mathfrak{b}$ is a \v{C}ech $A$-category. If $X$ is a left
Fr\'{e}chet $A$-module, $U\subseteq\mathbb{C}_{xy}$ a open subset containing
$\sigma_{\operatorname{P}}\left(  \mathcal{O}_{q},X\right)  $, then the
$A$-action on $X$ can be lifted to a left $\mathcal{O}_{q}\left(  U\right)
$-module structure on $X$.
\end{proposition}

\begin{proof}
The augmented \v{C}ech complex (\ref{ACc}) of $\mathcal{O}_{q}$ is just the
\v{C}ech complex of the covering $\mathfrak{b}$. Namely, put $J=I\times
\mathbb{N}$ and every tuple $\gamma\in J^{p+1}$ is an ordered pair $\left(
\alpha,\beta\right)  $ of tuples $\alpha=\left(  i_{0},\ldots,i_{p}\right)
\in I^{p+1}$ and $\beta=\left(  n_{0},\ldots,n_{p}\right)  \in\mathbb{N}%
^{p+1}$. Put $U_{\gamma}=\cap\left\{  U_{i_{k},n_{k}}:0\leq k\leq p\right\}
$, that is, $U_{\gamma}=V_{\alpha}\cup W_{\beta}$ with $V_{\alpha}=V_{i_{0}%
}\cap\ldots\cap V_{i_{p}}$ and $W_{\beta}=W_{n_{0}}\cap\ldots\cap W_{n_{p}}$.
Let us consider the polydomains $D_{i,n}=V_{i}\times W_{n}$, $\left(
i,n\right)  \in J$ in the standard complex space $\mathbb{C}^{2}$. Then
$\mathfrak{b}^{\times}=\left\{  D_{i,n}:\left(  i,n\right)  \in J\right\}  $
is an open covering of $\mathbb{C}^{2}$ with their finite intersections
$D_{\gamma}=\cap\left\{  D_{i_{k},n_{k}}:0\leq k\leq p\right\}  =V_{\alpha
}\times W_{\beta}$, $\gamma=\left(  \alpha,\beta\right)  $, which are
polydomains too. By the proper construction of the presheaf $\mathcal{O}_{q}$,
we have%
\[
\mathcal{O}_{q}\left(  U_{\gamma}\right)  =\mathcal{O}_{q}\left(  V_{\alpha
}\cup W_{\beta}\right)  =\mathcal{O}\left(  V_{\alpha}\right)
\widehat{\otimes}\mathcal{O}\left(  W_{\beta}\right)  =\mathcal{O}\left(
V_{\alpha}\times W_{\beta}\right)  =\mathcal{O}\left(  D_{\gamma}\right)
\]
for all $\gamma=\left(  \alpha,\beta\right)  $, where $\mathcal{O}$ is the
(commutative) sheaf of stalks of holomorphic functions on $\mathbb{C}^{2}$. It
follows that
\[
\mathcal{E}^{p}\left(  \mathfrak{b}\right)  =\prod\left\{  \mathcal{O}%
_{q}\left(  U_{\gamma}\right)  :\gamma\in J^{p+1}\right\}  =\prod\left\{
\mathcal{O}\left(  D_{\gamma}\right)  :\gamma\in J^{p+1}\right\}  ,\quad
p\geq0,
\]
and one can easily verify that the \v{C}ech complex $\mathcal{E}^{\bullet
}\left(  \mathfrak{b}\right)  $ of the covering $\mathfrak{b}$ is reduced to
the standard \v{C}ech complex $\mathcal{E}^{\bullet}\left(  \mathfrak{b}%
^{\times}\right)  $ of the covering $\mathfrak{b}^{\times}$. But every
$D_{\gamma}$ and all space $\mathbb{C}^{2}$ are $\mathcal{O}$-acyclic spaces,
therefore the augmented \v{C}ech complex $0\rightarrow\mathcal{O}\left(
\mathbb{C}^{2}\right)  \longrightarrow\mathcal{E}^{\bullet}\left(
\mathfrak{b}^{\times}\right)  $ of the covering $\mathfrak{b}^{\times}$ is
exact (see Proposition \ref{propAcyclic}). In particular, so is (\ref{ACc})
for the algebra $A$. By Definition \ref{defCech}, $\mathcal{O}_{q}$ is a
\v{C}ech $A$-category.

Finally, let $X$ be a left $A$-module, which is finitely-free automatically
($A$ is of finite type). By Proposition \ref{remPut1}, the $\mathcal{O}_{q}%
$-spectrum $\sigma\left(  \mathcal{O}_{q},X\right)  $ is reduced to the
Putinar spectrum $\sigma_{\operatorname{P}}\left(  \mathcal{O}_{q},X\right)
$. It remains to use Theorem \ref{thcoreFC}.
\end{proof}

Notice that a left Fr\'{e}chet $\mathcal{O}_{q}\left(  \mathbb{C}_{xy}\right)
$-module $X$ is given by a couple $\left(  T,S\right)  $ of continuous linear
operator on $X$ such that $TS=q^{-1}ST$. For brevity we say that $X$ is a left
Fr\'{e}chet $q$-module. If $X$ is a left Banach $\mathfrak{A}_{q}$-module,
then the homomorphism $\mathfrak{A}_{q}\rightarrow\mathcal{B}\left(  X\right)
$, $x\mapsto T$, $y\mapsto S$ into the Banach algebra $\mathcal{B}\left(
X\right)  $ has a unique continuous (algebra homomorphism) extension
$\mathcal{O}_{q}\left(  \mathbb{C}_{xy}\right)  \rightarrow\mathcal{B}\left(
X\right)  $ to its Arens-Michael envelope $\mathcal{O}_{q}\left(
\mathbb{C}_{xy}\right)  $, that is, $X$ is a left Banach $q$-module
automatically. In the case of a left Banach $q$-module $X$, the $\mathcal{O}%
_{q}\left(  U\right)  $-functional calculus $\left(  T,S\right)  $ holds if
and only if $\sigma\left(  T\right)  \subseteq U_{x}$ and $\sigma\left(
S\right)  \subseteq U_{y}$ \cite[Theorem 5.2]{Dosi24}, which does not involve
a joint spectrum. But in the general case of the left Fr\'{e}chet $q$-modules,
one needs to use the Putinar spectrum $\sigma_{\operatorname{P}}\left(
\mathcal{O}_{q},X\right)  $. We also use the notation $\sigma\left(
T,S\right)  $ for the Taylor spectrum of a left Fr\'{e}chet $q$-module $X$
given by the couple $\left(  T,S\right)  $ (see Definition \ref{def11}), that
is, $\sigma\left(  T,S\right)  =\sigma\left(  \mathcal{O}_{q}\left(
\mathbb{C}_{xy}\right)  ,X\right)  $.

\subsection{Noncommutative analytic spaces $\left(  \mathbb{C}_{x}%
,\mathcal{O}\left[  \left[  y\right]  \right]  \right)  $ and $\left(
\mathbb{C}_{y},\left[  \left[  x\right]  \right]  \mathcal{O}\right)  $}

As above let $\mathfrak{A}_{q}$ be the contractive quantum plane, and let
$\mathcal{O}^{\mathfrak{q}}\left[  \left[  y\right]  \right]  =\mathcal{O}%
^{\mathfrak{q}}\widehat{\otimes}\mathbb{C}\left[  \left[  y\right]  \right]  $
be the projective tensor product of the Fr\'{e}chet sheaf $\mathcal{O}%
^{\mathfrak{q}}$ and the constant Fr\'{e}chet sheaf $\mathbb{C}\left[  \left[
y\right]  \right]  $ on the topological space $\left(  \mathbb{C}%
_{x},\mathfrak{q}\right)  $. Thus $\mathcal{O}^{\mathfrak{q}}\left[  \left[
y\right]  \right]  $ is a Fr\'{e}chet sheaf on $\left(  \mathbb{C}%
_{x},\mathfrak{q}\right)  $. It turns out that $\mathcal{O}^{\mathfrak{q}%
}\left[  \left[  y\right]  \right]  $ is a Fr\'{e}chet algebra sheaf equipped
with the formal $q$-multiplication (see \cite{Dosi25}). If $f=\sum_{n}%
f_{n}\left(  z\right)  y^{n}$ and $g=\sum_{n}g_{n}\left(  z\right)  y^{n}$ are
sections of the sheaf $\mathcal{O}^{\mathfrak{q}}\left[  \left[  y\right]
\right]  $ over a $q$-open subset $U_{x}\subseteq\mathbb{C}_{x}$, then we put
\begin{equation}
f\cdot g=\sum_{n}\left(  \sum_{i+j=n}f_{i}\left(  z\right)  g_{j}\left(
q^{i}z\right)  \right)  y^{n}. \label{lm1}%
\end{equation}
Notice that $\left\{  q^{i}z:i\in\mathbb{Z}_{+}\right\}  \cup\left\{
0\right\}  =\left\{  z\right\}  _{q}\subseteq U_{x}$ whenever $z\in U_{x}$,
and $f_{i}\left(  z\right)  g_{j}\left(  q^{i}z\right)  $ is the standard
multiplication from the commutative algebra $\mathcal{O}\left(  U_{x}\right)
$. If $\kappa:\mathfrak{A}_{q}\rightarrow B$ is a homomorphism into a a unital
Banach algebra with nilpotent $\kappa\left(  y\right)  $, then using very
similar arguments from \cite{DosiSS}, we deduce that $\kappa$ has a unique
extension up to be a continuous algebra homomorphism $\mathcal{O}\left(
\mathbb{C}_{x}\right)  \left[  \left[  y\right]  \right]  \rightarrow B$.
Hence $\mathcal{O}\left(  \mathbb{C}_{x}\right)  \left[  \left[  y\right]
\right]  =\widehat{\mathfrak{A}_{q}}^{y}$ is the multinormed envelope of
$\mathfrak{A}_{q}$. In a similar way, we come up with the Fr\'{e}chet algebra
sheaf $\left[  \left[  x\right]  \right]  \mathcal{O}^{\mathfrak{q}%
}=\mathbb{C}\left[  \left[  x\right]  \right]  \widehat{\otimes}%
\mathcal{O}^{\mathfrak{q}}$ on $\left(  \mathbb{C}_{y},\mathfrak{q}\right)  $,
and $\left[  \left[  x\right]  \right]  \mathcal{O}\left(  \mathbb{C}%
_{y}\right)  =\widehat{\mathfrak{A}_{q}}^{x}$.

If $A=\mathcal{O}\left(  \mathbb{C}_{x}\right)  \left[  \left[  y\right]
\right]  $ then the canonical homomorphism $\mathfrak{A}_{q}\rightarrow A$ is
a localization \cite{DosiLoc}, and $\operatorname{Spec}\left(  \mathcal{O}%
\left(  U_{x}\right)  \left[  \left[  y\right]  \right]  \right)  =U_{x}$ for
every $\mathfrak{q}$-open subset $U_{x}\subseteq\mathbb{C}_{x}$ (see
\cite[Corollary 5.4]{Dosi25}). Thus $A$ is of finite type and $\left(
\mathbb{C}_{x},\mathcal{O}\left[  \left[  y\right]  \right]  \right)  $ is a
standard analytic geometry of $A$. The topology base $\mathfrak{b}%
_{x}=\left\{  V_{x,i}:i\in I\right\}  $ for $\left(  \mathbb{C}_{x}%
,\mathfrak{q}\right)  $ consists of all Runge $q$-open subsets.

\begin{proposition}
\label{propCex2}The unital complete lattice $A$-category $\mathcal{O}\left[
\left[  y\right]  \right]  $ with the basis $\mathfrak{b}_{x}$ is a \v{C}ech
$A$-category. If $X$ is a left Fr\'{e}chet $A$-module given by an operator
couple $\left(  T,S\right)  $ with $TS=q^{-1}ST$, $U_{x}\subseteq
\mathbb{C}_{x}$ an open subset containing $\sigma_{\operatorname{P}}\left(
\mathcal{O}\left[  \left[  y\right]  \right]  ,X\right)  $, then the
$A$-action on $X$ can be lifted to a left $\mathcal{O}\left(  U_{x}\right)
\left[  \left[  y\right]  \right]  $-module structure on $X$, and
$\sigma\left(  T,S\right)  \subseteq\sigma_{\operatorname{P}}\left(
\mathcal{O}\left[  \left[  y\right]  \right]  ,X\right)  $. If $X$ is a left
Banach $A$-module given by a couple $\left(  T,S\right)  $ from $\mathcal{B}%
\left(  X\right)  $ with the nilpotent $S$, then
\[
\sigma_{\operatorname{P}}\left(  \mathcal{O}\left[  \left[  y\right]  \right]
,X\right)  =\sigma\left(  T,S\right)  ^{-\mathfrak{q}}.
\]

\end{proposition}

\begin{proof}
The augmented \v{C}ech complex (\ref{ACc}) of $\mathcal{O}\left[  \left[
y\right]  \right]  $ is just the \v{C}ech complex of the covering
$\mathfrak{b}_{x}$. As above for every tuple $\alpha=\left(  i_{0}%
,\ldots,i_{p}\right)  \in I^{p+1}$ we have $V_{x,\alpha}=V_{x,i_{0}}\cap
\ldots\cap V_{x,i_{p}}$. As above, every $V_{x,\alpha}$ and all space
$\mathbb{C}_{x}$ are $\mathcal{O}$-acyclic spaces, therefore the augmented
\v{C}ech complex $0\rightarrow\mathcal{O}\left(  \mathbb{C}_{x}\right)
\longrightarrow\mathcal{E}_{\mathbb{C}}^{\bullet}\left(  \mathfrak{b}%
_{x}\right)  $ of the covering $\mathfrak{b}_{x}$ for the standard sheaf
$\mathcal{O}$ is exact. But $\mathbb{C}\left[  \left[  y\right]  \right]  $ is
a nuclear Fr\'{e}chet space, therefore $0\rightarrow\mathcal{O}\left(
\mathbb{C}_{x}\right)  \widehat{\otimes}\mathbb{C}\left[  \left[  y\right]
\right]  \longrightarrow\mathcal{E}_{\mathbb{C}}^{\bullet}\left(
\mathfrak{b}_{x}\right)  \widehat{\otimes}\mathbb{C}\left[  \left[  y\right]
\right]  $. One can easily verify that $\mathcal{E}^{\bullet}\left(
\mathfrak{b}_{x}\right)  =\mathcal{E}_{\mathbb{C}}^{\bullet}\left(
\mathfrak{b}_{x}\right)  \widehat{\otimes}\mathbb{C}\left[  \left[  y\right]
\right]  $, that is, $\mathcal{O}\left[  \left[  y\right]  \right]  $ is a
\v{C}ech $A$-category.

Now let $X$ be a left Fr\'{e}chet $A$-module, which in turn is a left
Fr\'{e}chet $q$-module automatically. By Proposition \ref{remPut1}, the
$\mathcal{O}\left[  \left[  y\right]  \right]  $-spectrum $\sigma\left(
\mathcal{O}\left[  \left[  y\right]  \right]  ,X\right)  $ is reduced to the
Putinar spectrum $\sigma_{\operatorname{P}}\left(  \mathcal{O}\left[  \left[
y\right]  \right]  ,X\right)  $, and the functional calculus follows due to
Theorem \ref{thcoreFC}.

Since both $\mathfrak{A}_{q}\rightarrow\mathcal{O}_{q}\left(  \mathbb{C}%
_{xy}\right)  $ and $\iota:\mathfrak{A}_{q}\rightarrow A$ are localizations,
it follows that so is $\mathcal{O}_{q}\left(  \mathbb{C}_{xy}\right)
\rightarrow A$ (see \cite[Proposition 1.8]{Tay2}). Therefore $\sigma\left(
T,S\right)  =\sigma\left(  \mathcal{O}_{q}\left(  \mathbb{C}_{xy}\right)
,X\right)  =\sigma\left(  A,X\right)  $. Further, $\mathcal{O}\left[  \left[
y\right]  \right]  $ is a nuclear $A$-category with the localizing basis
$\mathfrak{b}_{x}$ \cite{DosiLoc}, and $\mathcal{O}\left[  \left[  y\right]
\right]  ^{\sim}\cap\mathcal{T}=\operatorname{Spec}\left(  A\right)
=\mathbb{C}_{x}$. Using Proposition \ref{propLocG}, we deduce that
\[
\sigma\left(  T,S\right)  =\sigma\left(  A,X\right)  \subseteq\sigma
_{\operatorname{P}}\left(  \mathcal{O}\left[  \left[  y\right]  \right]
,X\right)  \subseteq\mathbb{C}_{x}.
\]
Notice that in the case of a left Banach $q$-module, the fact that
$\sigma\left(  T,S\right)  \subseteq\mathbb{C}_{x}$ also follows from the
$q$-projection property proven in \cite{Dosi242}. Namely, since
$A=\widehat{\mathfrak{A}_{q}}^{y}$, it follows that $S$ is a nilpotent
operator and
\begin{align*}
\sigma\left(  T,S\right)   &  \subseteq\left(  \left(  \sigma\left(  T\right)
\cup\sigma\left(  q^{-1}T\right)  \right)  \times\left\{  0\right\}  \right)
\cup\left(  \left\{  0\right\}  \times\left(  \sigma\left(  S\right)
\cup\sigma\left(  qS\right)  \right)  \right) \\
&  =\left(  \sigma\left(  T\right)  \cup\sigma\left(  q^{-1}T\right)  \right)
\times\left\{  0\right\}  \subseteq\mathbb{C}_{x}.
\end{align*}
Finally, a left Banach $q$-module $X$ is an $\mathcal{O}\left[  \left[
y\right]  \right]  $\textit{-}local left $A$-module \cite[Theorem 6.3]%
{Dosi25}. Using again Proposition \ref{propLocG}, we obtain that
$\sigma_{\operatorname{P}}\left(  \mathcal{O}\left[  \left[  y\right]
\right]  ,X\right)  =\sigma\left(  T,S\right)  ^{-}$ with respect to the
topology of the point completion $\mathcal{O}\left[  \left[  y\right]
\right]  ^{\sim}$, which is just the $\mathfrak{q}$-topology of $\mathbb{C}%
_{x}$ (see Lemma \ref{lemAG1}).
\end{proof}

\begin{remark}
\label{remxO}In a similar way, $\left(  \mathbb{C}_{y},\left[  \left[
x\right]  \right]  \mathcal{O}\right)  $ stands for the multinormed envelope
$A=\widehat{\mathfrak{A}_{q}}^{x}$. In this case, $\sigma\left(  T,S\right)
\subseteq\sigma_{\operatorname{P}}\left(  \left[  \left[  x\right]  \right]
\mathcal{O},X\right)  \subseteq\mathbb{C}_{y}$, and if $X$ is a left Banach
$A$-module given by a couple $\left(  T,S\right)  $ from $\mathcal{B}\left(
X\right)  $ with the nilpotent $T$, then $\sigma_{\operatorname{P}}\left(
\left[  \left[  x\right]  \right]  \mathcal{O},X\right)  =\sigma\left(
T,S\right)  ^{-\mathfrak{q}}$ in $\left(  \mathbb{C}_{y},\mathfrak{q}\right)
$.
\end{remark}

\subsection{Noncommutative analytic space $\left(  \left\{  \left(
0,0\right)  \right\}  ,\mathbb{C}_{q}\left[  \left[  x,y\right]  \right]
\right)  $}

Now let $\mathfrak{A}_{q}$ be the quantum plane with $\left\vert q\right\vert
\leq1$, and consider its RN or RF envelope $A$ (see (\ref{Cq})), which is just
the nuclear Fr\'{e}chet algebra $\mathbb{C}_{q}\left[  \left[  x,y\right]
\right]  $ of all formal power series in variables $x$ and $y$ equipped with
the $q$-multiplication from (\ref{lm1}) with $x$ replaced by $z$. The
canonical homomorphism $\iota:\mathfrak{A}_{q}\rightarrow A$ is a localization
\cite{DosiLoc}, and $\operatorname{Spec}\left(  A\right)  =\left\{  \left(
0,0\right)  \right\}  $ (see \cite[Lemma 5.1]{Dosi25}). Thus $A$ is of finite
type and $\left(  \left\{  \left(  0,0\right)  \right\}  ,\mathbb{C}%
_{q}\left[  \left[  x,y\right]  \right]  \right)  $ is a standard analytic
geometry of $A$ with its constant sheaf and localizing basis $\mathfrak{b}%
_{0}=\left\{  A\right\}  $. Obviously, it is a \v{C}ech $A$-category with the
basis $\mathfrak{b}_{0}$.

\begin{proposition}
\label{propCex3}Let $X$ be a nontrivial left Fr\'{e}chet $q$-module given by
an operator couple $\left(  T,S\right)  $ with $TS=q^{-1}ST$, $\left\vert
q\right\vert \leq1$. Then $X$ is a left $A$-module if and only if
$\sigma_{\operatorname{P}}\left(  A,X\right)  \neq\varnothing$. In this case,%
\[
\sigma\left(  A,X\right)  =\sigma_{\operatorname{P}}\left(  A,X\right)
=\sigma\left(  T,S\right)  =\left\{  \left(  0,0\right)  \right\}  .
\]

\end{proposition}

\begin{proof}
As above taking into account that both $\mathfrak{A}_{q}\rightarrow
\mathcal{O}_{q}\left(  \mathbb{C}_{xy}\right)  $ and $\iota:\mathfrak{A}%
_{q}\rightarrow A$ are localizations, we conclude that so is $\mathcal{O}%
_{q}\left(  \mathbb{C}_{xy}\right)  \rightarrow A$ (see \cite[Proposition
1.8]{Tay2}) and $\sigma\left(  T,S\right)  =\sigma\left(  A,X\right)  $
whenever $X$ is a left $A$-module. Using Lemma \ref{lemDom12}, we deduce that
$\left(  A,\iota\right)  \gg X$ holds iff $X$ is a left $A$-module. In
particular, if $X$ is a left $A$-module, then $A\perp X$ does not hold,
therefore $A\in\sigma\left(  A,X\right)  $. As above, by Proposition
\ref{remPut1}, we obtain that $\sigma\left(  A,X\right)  =\sigma
_{\operatorname{P}}\left(  A,X\right)  $, $A^{\sim}\cap\mathcal{T}%
=\operatorname{Spec}\left(  A\right)  =\left\{  \left(  0,0\right)  \right\}
$ and $\sigma\left(  T,S\right)  =\sigma\left(  \operatorname{Spec}\left(
A\right)  ,X\right)  \subseteq\sigma_{\operatorname{P}}\left(  A,X\right)
=\left\{  \left(  0,0\right)  \right\}  $. In particular, $\sigma
_{\operatorname{P}}\left(  A,X\right)  \neq\varnothing$. Moreover, it was
proved in \cite[Proposition 6.4]{Dosi25}, that $\left(  0,0\right)
\notin\sigma\left(  T,S\right)  $ iff $A\perp X$ holds. Therefore $\left(
0,0\right)  \in\sigma\left(  T,S\right)  $ and we have $\sigma\left(
A,X\right)  =\sigma_{\operatorname{P}}\left(  A,X\right)  =\sigma\left(
T,S\right)  =\left\{  \left(  0,0\right)  \right\}  $.

Conversely, if $\sigma_{\operatorname{P}}\left(  A,X\right)  $ is not empty,
then $A\in\sigma\left(  A,X\right)  $ or $X$ is a left $A$-module. Optionally,
one can use Theorem \ref{thcoreFC} to deduce the presence of $A$-calculus on
$X$.
\end{proof}

\begin{remark}
In the case of a left Banach $q$-module $X$ with the nilpotent operators $T$
and $S$, we obtain that $X$ is a left $A$-module automatically. Indeed, using
\cite[Lemma 3.1, 3.2]{Dosi24}, we derive that the (non-unital) subalgebra
$B\subseteq\mathcal{B}\left(  X\right)  $ generated by $T$ and $S$ is a finite
dimensional nilpotent algebra, and $\mathfrak{A}_{q}\rightarrow B_{+}$,
$x\mapsto T$, $y\mapsto S$ is a local morphism. Taking into account that
$A=\widehat{\mathfrak{A}_{q}}^{\text{RF}}$ (see (\ref{Cq})), we derive that
$X$ is a left Banach $A$-module given by the couple $\left(  T,S\right)  $. By
Proposition \ref{propCex3}, we conclude that $\sigma\left(  T,S\right)
=\left\{  \left(  0,0\right)  \right\}  $. Actually, the inclusion
$\sigma\left(  T,S\right)  \subseteq\left\{  \left(  0,0\right)  \right\}  $
follows from the $q$-projection property mentioned above.
\end{remark}

\subsection{The formal $q$-geometry}

Let $\mathfrak{A}_{q}$ be the contractive quantum plane. The formal geometry
of $\mathfrak{A}_{q}$ \cite{Dosi25} is given by a couple $\left(
\mathbb{C}_{xy},\mathcal{F}_{q}\right)  $ of the topological space
$\mathbb{C}_{xy}$, and a noncommutative structure (Fr\'{e}chet) algebra sheaf
$\mathcal{F}_{q}$ on $\mathbb{C}_{xy}$. The space $\mathbb{C}_{xy}$ is
equipped with the final $\mathfrak{q}$-topology so that both embeddings%
\[
\left(  \mathbb{C}_{x},\mathfrak{q}\right)  \hookrightarrow\mathbb{C}%
_{xy}\hookleftarrow\left(  \mathbb{C}_{y},\mathfrak{q}\right)
\]
are continuous. Thus $U\subseteq\mathbb{C}_{xy}$ is open iff $U=U_{x}\cup
U_{y}$ for $\mathfrak{q}$-open components $U_{x}\subseteq\mathbb{C}_{x}$ and
$U_{y}\subseteq\mathbb{C}_{y}$. The structure sheaf $\mathcal{F}_{q}$ on
$\mathbb{C}_{xy}$ is defined (see \cite{Dosi25}) as the fibered product
\[
\mathcal{F}_{q}=\mathcal{O}\left[  \left[  y\right]  \right]
\underset{\mathbb{C}_{q}\left[  \left[  x,y\right]  \right]  }{\times}\left[
\left[  x\right]  \right]  \mathcal{O}%
\]
of the Fr\'{e}chet algebra sheaves $\mathcal{O}\left[  \left[  y\right]
\right]  $ (on $\mathbb{C}_{x}$) and $\left[  \left[  x\right]  \right]
\mathcal{O}$ (on $\mathbb{C}_{y}$) over the constant sheaf $\mathbb{C}%
_{q}\left[  \left[  x,y\right]  \right]  $ equipped with the $q$%
-multiplication extending one of $\mathfrak{A}_{q}$. The algebra
$A=\Gamma\left(  \mathbb{C}_{xy},\mathcal{F}_{q}\right)  $ of all global
sections is reduced to the N-envelope $\widehat{\mathfrak{A}_{q}}^{\text{N}}$
(or PI-envelope) of $\mathfrak{A}_{q}$, the canonical embedding $\mathfrak{A}%
_{q}\rightarrow A$ is a localization \cite{DosiLoc}, and $\operatorname{Spec}%
\left(  \mathcal{F}_{q}\left(  U\right)  \right)  =U$ for every $\mathfrak{q}%
$-open subset $U\subseteq\mathbb{C}_{xy}$ \cite{DosiSS}. Thus $\left(
\mathbb{C}_{xy},\mathcal{F}_{q}\right)  $ is a standard analytic geometry (see
Definition \ref{defAG}) of $A$. Moreover, every $q$-open subset of
$\mathbb{C}_{xy}$ is $\mathcal{F}_{q}$-acyclic (see \cite{DosiSS}), and
$\mathbb{C}_{xy}$ has a countable topology base $\mathfrak{b}$ of Runge
$q$-open subsets, which is a localizing basis indeed (see \cite{DosiLoc}).

\begin{proposition}
\label{propQP1}The unital complete lattice $A$-category $\mathcal{F}_{q}$ is
\textit{a \v{C}ech }$A$-\textit{category} with its nuclear localizing basis
$\mathfrak{b}$. If $X$ is a left Fr\'{e}chet $A$-module given by an operator
couple $\left(  T,S\right)  $ with $TS=q^{-1}ST$, $\left\vert q\right\vert
<1$, and $U\subseteq\mathbb{C}_{xy}$ is a Runge $q$-open subset containing
$\sigma_{\operatorname{P}}\left(  \mathcal{F}_{q},X\right)  $, then $\left(
\mathcal{F}_{q}\left(  U\right)  ,\iota_{U}\right)  \gg X$ and
\[
\sigma\left(  \mathcal{F}_{q}\left(  U\right)  ,X\right)  =\sigma\left(
T,S\right)  \subseteq\sigma_{\operatorname{P}}\left(  \mathcal{F}%
_{q},X\right)  =\sigma_{\operatorname{P}}\left(  \mathcal{F}_{q}%
|_{U},X\right)  ,
\]
where $\iota_{U}:A\rightarrow\mathcal{F}_{q}\left(  U\right)  $ is the
restriction morphism. In the case of a left Banach $A$-module $X$, we have
$\sigma\left(  T,S\right)  ^{-\mathfrak{q}}=\sigma_{\operatorname{P}}\left(
\mathcal{F}_{q},X\right)  $.
\end{proposition}

\begin{proof}
First notice that every finite intersection from $\mathfrak{b}$ being a
$q$-open subset turns out to be $\mathcal{F}_{q}$-acyclic. Moreover,
$\mathbb{C}_{xy}$ is $\mathcal{F}_{q}$-acyclic too. By Proposition
\ref{propAcyclic}, we deduce that $\mathcal{F}_{q}$ is a \v{C}ech $A$-category
with the nuclear, localizing basis $\mathfrak{b}$. The $\mathcal{F}_{q}%
$-spectrum $\sigma\left(  \mathcal{F}_{q},X\right)  $ is reduced to the
Putinar spectrum $\sigma_{\operatorname{P}}\left(  \mathcal{F}_{q},X\right)  $
thanks to Proposition \ref{remPut1}.

Now let $U\subseteq\mathbb{C}_{xy}$ be a Runge $q$-open subset containing
$\sigma_{\operatorname{P}}\left(  \mathcal{F}_{q},X\right)  $, and put
$\mathcal{A=F}_{q}\left(  U\right)  $. Since $\iota_{U}:A\rightarrow
\mathcal{A}$ is a localization \cite{DosiLoc}, it follows that $\left(
\mathcal{A},\iota_{U}\right)  \gg X$ holds iff $X$ has a left $\mathcal{A}%
$-module structure extending its original $A$-module one (see Lemma
\ref{lemDom12}). But $X$ has that left $\mathcal{A}$-module structure thanks
to Corollary \ref{corkeyy1}, therefore $\left(  \mathcal{A},\iota_{U}\right)
\gg X$ holds.

As above, the canonical embedding $\mathcal{O}_{q}\left(  \mathbb{C}%
_{xy}\right)  \rightarrow A$ is a localization too. It follows that
$\sigma\left(  T,S\right)  =\sigma\left(  \mathcal{O}_{q}\left(
\mathbb{C}_{xy}\right)  ,X\right)  =\sigma\left(  A,X\right)  $, and
$\sigma\left(  A,X\right)  \subseteq\sigma_{\operatorname{P}}\left(
\mathcal{F}_{q},X\right)  $ thanks to Proposition \ref{propLocG}. If $X$ is a
left Banach $q$-module $X$ then it is an $\mathcal{F}_{q}$\textit{-}local left
$A$-module \cite[Theorem 6.5]{Dosi25}. By Proposition \ref{propLocG}, we
obtain that $\sigma_{\operatorname{P}}\left(  \mathcal{F}_{q},X\right)
=\sigma\left(  T,S\right)  ^{-\mathfrak{q}}$ with respect to the topology of
the point completion $\mathcal{F}_{q}^{\sim}$, which is just the
$\mathfrak{q}$-topology of $\mathbb{C}_{xy}$ (see Lemma \ref{lemAG1}).

Finally, taking into account $\left(  \mathcal{A},\iota_{U}\right)  \gg X$ and
$\sigma\left(  T,S\right)  \subseteq\sigma_{\operatorname{P}}\left(
\mathcal{F}_{q},X\right)  \subseteq U$, one can use the spectral mapping
Theorem \ref{corkeyy2}. Namely, the equalities $\sigma_{\operatorname{P}%
}\left(  \mathcal{F}_{q}|_{U},X\right)  =\sigma_{\operatorname{P}}\left(
\mathcal{F}_{q},X\right)  $ and $\sigma\left(  T,S\right)  =\sigma\left(
\mathcal{F}_{q}\left(  U\right)  ,X\right)  $ hold.
\end{proof}

Now we can prove the key result on the $\mathcal{F}_{q}$-functional calculus
for a Fr\'{e}chet space representation of the contractive quantum plane
$\mathfrak{A}_{q}$.

\begin{theorem}
\label{thcoreFC2}Let $X$ be a left Fr\'{e}chet $A$-module given by an operator
couple $\left(  T,S\right)  $ with $TS=q^{-1}ST$, $\left\vert q\right\vert
<1$, and let $U\subseteq\mathbb{C}_{xy}$ be a $\mathfrak{q}$-open subset
containing the Putinar spectrum $\sigma_{\operatorname{P}}\left(
\mathcal{F}_{q},X\right)  $. Then $X$ turns out to be a left $\mathcal{F}%
_{q}\left(  U\right)  $-module lifting its $A$-module structure. In
particular, if $X$ is a left Banach $A$-module with the nilpotent $TS$, and
$U$ is a $\mathfrak{q}$-open subset containing the $\mathfrak{q}$-closure
$\sigma\left(  T,S\right)  ^{-\mathfrak{q}}$ of the Taylor spectrum
$\sigma\left(  T,S\right)  $, then there is a continuous functional calculus
$\mathcal{F}_{q}\left(  U\right)  \rightarrow\mathcal{B}\left(  X\right)  $,
$f\mapsto f\left(  T,S\right)  $.
\end{theorem}

\begin{proof}
By Proposition \ref{propQP1}, $\mathcal{F}_{q}$ is a \v{C}ech category with
its nuclear basis $\mathfrak{b}$. The $q$-open subset $U$ corresponds to an
open subset category in $\mathcal{F}_{q}$ (see Proposition \ref{corCes1})
\ Moreover, the Putinar spectrum $\sigma_{\operatorname{P}}\left(
\mathcal{F}_{q},X\right)  $ corresponds to the $\mathcal{F}_{q}$-spectrum
$\sigma\left(  \mathcal{F}_{q},X\right)  $ thanks to Proposition
\ref{remPut1}. Using Theorem \ref{thcoreFC}, we deduce that $X$ turns out to
be a left $\mathcal{F}_{q}\left(  U\right)  $-module lifting its $A$-module
structure. The rest follows from Proposition \ref{propQP1} and the fact that
$A=$ $\widehat{\mathfrak{A}_{q}}^{\text{N}}$.
\end{proof}

\begin{remark}
Thus for the $\mathcal{F}_{q}$-functional calculus problem for the
$q$-operator couple $T,S$ acting on a Banach space $X$, the $\mathfrak{q}%
$-closure $\sigma\left(  T,S\right)  ^{-\mathfrak{q}}$ of $\sigma\left(
T,S\right)  $ plays a crucial role rather than the Taylor spectrum
$\sigma\left(  T,S\right)  $ itslef. More concrete examples of these spectra
and their closures were considered in \cite{Dosi242} and \cite{DosiSS}.
\end{remark}

The canonical morphisms of the multinormed envelopes of $\mathfrak{A}_{q}$
from Subsection \ref{subsecMEQ} define in turn the morphisms of the related
\v{C}ech\textit{ }categories (or analytic geometries) considered above.
Namely, using Propositions \ref{propCex2}, \ref{propCex3} \ref{propQP1} and
Remark \ref{remxO}, we have the commutative diagrams
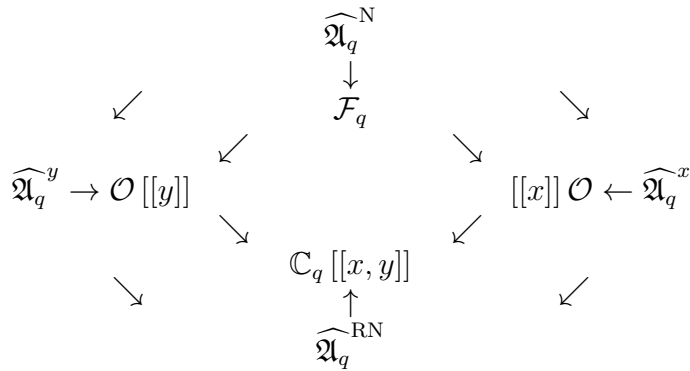
\[%
\begin{array}
[c]{ccccc}%
\begin{array}
[c]{c}%
\\
\\
\end{array}%
\begin{array}
[c]{c}%
\\
\\
\end{array}%
\begin{array}
[c]{c}%
\\
\\
\swarrow
\end{array}
&  &
\begin{array}
[c]{c}%
\widehat{\mathfrak{A}_{q}}^{\text{N}}\\
\downarrow\\
\mathcal{F}_{q}%
\end{array}
&  &
\begin{array}
[c]{c}%
\\
\\
\searrow
\end{array}%
\begin{array}
[c]{c}%
\\
\\
\end{array}%
\begin{array}
[c]{c}%
\\
\\
\end{array}
\\
& \swarrow &  & \searrow & \\
\widehat{\mathfrak{A}_{q}}^{y}\rightarrow\mathcal{O}\left[  \left[  y\right]
\right]  &  &  &  & \left[  \left[  x\right]  \right]  \mathcal{O}%
\leftarrow\widehat{\mathfrak{A}_{q}}^{x}\\
& \searrow &  & \swarrow & \\%
\begin{array}
[c]{c}%
\\
\\
\end{array}%
\begin{array}
[c]{c}%
\\
\\
\end{array}%
\begin{array}
[c]{c}%
\searrow\\
\\
\end{array}
&  &
\begin{array}
[c]{c}%
\mathbb{C}_{q}\left[  \left[  x,y\right]  \right] \\
\uparrow\\
\widehat{\mathfrak{A}_{q}}^{\text{RN}}%
\end{array}
&  &
\begin{array}
[c]{c}%
\swarrow\\
\\
\end{array}%
\begin{array}
[c]{c}%
\\
\\
\end{array}%
\begin{array}
[c]{c}%
\\
\\
\end{array}
\end{array}
\]
of the canonical homomorphisms and the category morphisms.

Finally, let us confirm that in the case of the universal enveloping algebra
$\mathfrak{A}=\mathcal{U}\left(  \mathfrak{g}\right)  $ of a nilpotent Lie
algebra $\mathfrak{g}$, the proposed framework including Theorem
\ref{thcoreFC} result in the complex analytic geometry of
$A=\widehat{\mathfrak{A}}^{\text{N}}$ and the related functional calculus
considered in \cite{DosJOT10}.

\end{document}